\documentclass[amsart]{amsart}

\usepackage{amssymb,amsbsy,amsmath,amsfonts,amssymb,amsthm,amscd,mathtools,url}
\usepackage{latexsym,mathrsfs,bm}
\usepackage{graphicx,color}
\usepackage[english]{babel}
\usepackage{paralist}
\usepackage{amsthm,amsfonts,amssymb,amsmath}
\usepackage[latin1]{inputenc}
\newcommand{\eps}{\varepsilon}
\newcommand{\R}{\mathbb{R}}
\newcommand{\Q}{\mathbb{Q}}
\newcommand{\F}{\mathbb{F}}
\newcommand{\C}{\mathbb{C}}
\newcommand{\N}{\mathbb{N}}

\newcommand{\Z}{\mathbb{Z}}

\newcommand{\es}[1]{\begin{equation}\begin{split}#1\end{split}\end{equation}}
\newcommand{\est}[1]{\begin{equation*}\begin{split}#1\end{split}\end{equation*}}

\newcommand{\PP}{\mathbb{P}}

\newcommand{\LL}{\mathcal{L}}

\newcommand{\tn}[1]{\textnormal{#1}}
\newcommand{\leg}[2]{\left(\frac{#1}{#2}\right)}

\renewcommand{\mod}[1]{~\pr{\textnormal{mod}~#1}}
\newtheorem*{theo*}{Theorem}
\newtheorem{theo}{Theorem}
\newtheorem{theorem}[theo]{Theorem}
\newtheorem{prop}{Proposition}
\newtheorem{conj}{Conjecture}
\newtheorem{defin}[prop]{Definition}

\newtheorem{lemma}[prop]{Lemma}
\newtheorem{corol}[prop]{Corollary}
\newtheorem{remark}{Remark}
\newtheorem*{rem*}{Remark}

\newcommand{\pr}[1]{\left( #1\right)}

\newcommand{\sgn}{\operatorname{sgn}}
\DeclareMathOperator{\Av}{Av}
\DeclareMathOperator{\Gal}{Gal}
\DeclareMathOperator{\Sel}{Sel}

\DeclareMathOperator{\tr}{tr}

\newcommand{\Res}{\operatorname*{Res}}

\newcommand{\sign}{\operatorname*{sign}}
\newcommand{\rank}{\operatorname*{rank}}
\renewcommand{\r}{s}
\newcommand{\hr}{a}
\newcommand{\n}{m}
\newcommand{\rt}{\eps}
\newcommand{\rf}{{r_{\GFF}}}
\newcommand{\rootf}{\rt_{\GFF}}
\newcommand{\Avf}{\Av_\Z({\rootf})}
\newcommand{\Wa}{\mathcal W_{a}}
\newcommand{\was}{\mathcal{W}}
\newcommand{\Mw}{E_{\Wa}}
\newcommand{\Mh}{E_{\Hold}}
\newcommand{\Hold}{\mathcal{V}_{\hr}}
\newcommand{\NF}{{\mathcal{W}^\dagger}}
\newcommand{\NFa}{{\mathcal{W}_a^\dagger}}
\newcommand{\GFF}{\mathcal{F}}
\newcommand{\PFF}{\mathcal{F}}
\newcommand{\GG}{\mathcal{G}}
\newcommand{\II}{\mathcal{I}}
\newcommand{\HH}{\mathcal{H}}
\newcommand{\JJ}{\mathcal{J}}
\newcommand{\EE}{\mathcal{E}}

\newcommand{\lra}{\leftrightarrow}
\let\originalleft\left
\let\originalright\right
\renewcommand{\left}{\mathopen{}\mathclose\bgroup\originalleft}
\renewcommand{\right}{\aftergroup\egroup\originalright}
\newcommand{\kommentar}[1]{}

\definecolor{pink}{rgb}{1,.2,.6}
\definecolor{orange}{rgb}{0.7,0.3,0}
\definecolor{blue}{rgb}{.2,.6,.75}
\definecolor{green}{rgb}{.4,.7,.4}

\newcommand{\expl}[1]{%{\color{pink}{\footnote{\color{pink}More details: #1}}}
}

\newcommand{\suppress}[1]{}

\input cyracc.def \font\tencyr=wncyr10 \def\russe{\tencyr\cyracc}
\def\Sha{\text{\russe{Sh}}}

\numberwithin{equation}{section}

\title[Biased Families of Elliptic Curves]{Non-isotrivial elliptic surfaces with non-zero average root number}

\author{Sandro Bettin}
\address{DIMA - Dipartimento di Matematica, Via Dodecaneso, 35, 16146 Genova - ITALY}
\email{bettin@dima.unige.it}

\author{Chantal David}
\address{Department of Mathematics and Statistics Concordia University, 1455 de Maisonneuve West Montr\'eal, Qu\'ebec Canada H3G 1M8}
\email{cdavid@mathstat.concordia.ca}

\author{Christophe Delaunay}
\address{Laboratoire de Math\'ematiques de Besan\c con, Univ. Bourgogne Franche-Comt\'e, CNRS UMR 6623, 16 route de Gray, 25030 Besan\c con cedex, France}
\email{christophe.delaunay@univ-fcomte.fr}

\begin{document}

\begin{abstract}
We consider the problem of finding {non-isotrivial} $1$-parameter families of elliptic curves whose root number does not average to zero as the parameter varies in $\Z$. We classify all such families when the degree of the coefficients (in the parameter $t$)  is less than or equal to $2$ and we compute the rank over $\Q(t)$ of all these families. Also, we compute explicitly the average of the root numbers for some of these families highlighting some special cases. Finally, we prove some results on the possible values average root numbers can take, showing for example that all rational number in $[-1,1]$ are average root numbers for some {non-isotrivial} $1$-parameter family.

\end{abstract}


\maketitle

\section{Introduction}

This article is concerned with  families of elliptic curves defined
over $\Q$ such that the root number of the specializations does not behave, on average, as expected in
the classical cases.

More precisely, by a family of elliptic curves, we mean an elliptic surface over $\Q$ or, equivalently, an elliptic curve defined over $\Q(t)$ given
by a Weierstrass equation
\begin{equation} \label{eq_family}
{\GFF} \colon y^2 = x^3 +a_2(t)x^2 + a_4(t)x +a_6(t)
\end{equation}
where  $a_2(t), a_4(t)$ and $a_6(t)$ are polynomials with coefficients  in $\Z$. We denote by
 $\rf$ the rank of ${\GFF}$ over $\Q(t)$.

For each $t\in \Q$, we denote by ${\GFF}(t)$ the associated curve over $\Q$ defined by the specialization at $t$ of ${\GFF}$.
Then, for all but finitely many values of $t$, ${\GFF}(t)$ is an elliptic curve defined over $\Q$ and we let
$\rf(t)$ and $\varepsilon_{\GFF}(t)$ denote its rank over $\Q$ and its root number respectively. The parity conjecture predicts that $(-1)^{\rf(t)} = \rootf(t)$ and Silverman's specialization theorem gives that $\rf(t) \geq \rf$ for all but finitely many values of $t$.  One also conjectures that, up to a zero density subset of $\Q$, $\rf$ is the smallest integer compatible with the parity conjecture and thus that $\rf(t)$ is equal to either $\rf$ or $\rf+1$ depending on the parity given by $\rootf(t)$.

We define the average root number of ${\GFF}$ over $\Z$ as %and over $\Q$ as
\begin{equation} \label{average-Z}
\Avf := \lim_{T \rightarrow \infty} \frac{1}{2T} \sum_{|t| \leq T} \rootf(t),\qquad
%\Av_\Q({\GFF}) := \lim_{T \rightarrow \infty} \frac{\pi^2}{12T^2} \sum_{\substack{|h|,|k| \leq T,\\ k>0, (h,k)=1}} \varepsilon(h/k)
\end{equation}
if the limit exists (and where we define $\rootf(t)=0$ if ${\GFF}(t)$ is not an elliptic curve).\footnote{Alternatively one could define $\Avf$ with the symmetric average $\frac{1}{2T} \sum_{|t| \leq T}$ replaced by $\frac{1}{T} \sum_{0\leq t \leq T}$. All the same considerations we make in the paper works in this case as well mutatis mutandis.
}
 %In this paper we mostly deal with averages over $\Z$, but most results can be extended to the case of $\Q$-averages.

The work of Helfgott (\cite{helfgott_thesis,helfgott}) implies conjecturally
(and unconditionally in some cases) that $\Avf=0$ as soon as there exists a place, other
than $-\deg$, of multiplicative reduction of ${\GFF}$ over $\Q(t)$. Indeed, assuming the square-free sieve conjecture, one sees that in this case $\rootf(t)$ behaves roughly like $\lambda(M(t))$ where $\lambda$ is the Liouville's function and $M(t)$ is a certain non-constant square-free polynomial, so that Chowla's conjecture implies that  $\Avf=0$.
This is the typical case, which occurs for ``most'' families $\GFF$.

When the family $\GFF$ has no place of multiplicative reduction (other
than possibly $-\deg$), the average root number could be non-zero. {The case when the family is isotrivial\footnote{Isotrivial means that the $j$-invariant of ${\GFF}$ is constant. For isotrivial families, one can take, for instance, the quadratic twist of a fixed elliptic curve $E/\Q$ by a polynomial $d(t)\in \Z[t]$,
$E^{d(t)} \colon d(t)y^2=y^2 = x^3 +a_2x^2 + a_4x +a_6$, where  $a_i \in \Z$ for $i=2,5,6$. In this case, it is easier to deal with the root number, for example if $d(t)$ is coprime with the conductor of
$E$, then the root number is simply given by some congruence relations.
}, and more precisely the case of quadratic twist families, has been the subject of several studies in the literature (see e.g.~\cite{rohrlich,rizzo_2,DD,KMR}). There are however very few examples of non-isotrival families with $\Avf\neq0$.}
There is the Washington's family~(\cite{washington}) for which Rizzo proved (\cite{rizzo}) that $\rootf(t)=-1$ for
all $t\in\Z$. % (the rank of Washington's family is 1 and so for this example $\varepsilon(t)=(-1)^\rf$ is
%compatible with the generic rank).
Rizzo (\cite{rizzo}) also gave an example of a family
${\GFF}$ with $j$-invariant $j_{\GFF}(t)=t$ and $\Avf \not \in \{-1,0,1\}$
(however for this family the degree of the
polynomials $a_i(t)$ of the model given in the form~\eqref{eq_family} are
quite large: for example $\deg a_6(t)~=~8$). Romano~\cite{romano} considered a slight generalization of Washington's family obtaining an infinite sequence of families $\PFF_{s}$ all with rational average root number and with $\lim_{s\to\infty}\Av_\Z(\rt_{\PFF_s})=\frac34$.
Finally, Helfgott (\cite{helfgott}) gave an example of a non-isotrivial
family with {\em average root numbers over $\Q$} not in $\{-1, 0, 1\}$
(the degree of the coefficients $a_i(t)$ are quite large in this case too).\footnote{See the end of the introduction for the precise definition of the average root number over $\Q$.}

\medskip
{In this paper we study non-isotrivial families ${\GFF}$ of elliptic curves with non-zero $\Avf$ in a more systematic way, with particular attention to the case where we have control on the rank of $\GFF$ over $\Q(t)$.} 
One of our first motivations was to illustrate several questions on elliptic curves and on their associated $L$-functions where $\Avf$ appears naturally as well as to be able to provide
numerical experimentation. %\footnote{This study will appear in a forthcoming paper.}
For example, in a forthcoming work we show under several conjectures that
the one-level density function corresponding to a family ${\GFF}$ is
$$
W_{\GFF}(t) = \rf\delta_0(\tau) + \left( \frac{1+(-1)^\rf\Avf}{2}\right) W_{\mathrm{SO(even)}}(\tau) + \left( \frac{1-(-1)^\rf\Avf}{2}\right) W_{\mathrm{SO(odd)}}(\tau)
$$
where $\delta_0$ is the Dirac measure at $0$ and $W_{\mathrm{SO(even)}}$
(resp. $W_{\mathrm{SO(odd)}}$) is the one-level density function of the classical orthogonal group of even size
(resp. odd size). We can also rewrite $W_{\GFF}$ as
$$
W_{\GFF}(t) = \left(\rf + \frac{1-(-1)^\rf\Avf}{2}\right) \delta_0(\tau) + 1+ (-1)^\rf \Avf
\frac{\sin 2 \pi \tau}{2 \pi \tau},
$$
where $\rf + \frac{1-(-1)^\rf\Avf}{2}$ is (conjecturally) the
average rank of the specialization. Notice that if $\Avf\not\in\{0,\pm1\}$, then $W_{\GFF}(t)$ doesn't reduce to $W_{\mathrm{SO(even)}}$ or $W_{\mathrm{SO(even)}}$ plus some multiple of $\delta_0(t)$, and so $W_{\GFF}(t)$ is not the $1$-level density function of one of the classical compact groups (of course one can divide the family into two subfamilies according to the sign of $\rootf(t)$ going back to $W_{\mathrm{SO(even)}}$ or $W_{\mathrm{SO(even)}}$; see~\cite{Farmer} and~\cite{Sarnak} for two proposed definitions of ``families of $L$-functions'' where this division is requested, see also \cite{kowalski} for a discussion on families). Another interesting case is that of ``elevated rank'', i.e. when $\Avf=-(-1)^{\rf}$ and so almost all specializations satisfy $\rf(t)>\rf$. Notice that when this happens then $W_{\GFF}(t)$ is, up to Dirac functions, equal to the $1$-level density of the orthogonal group with size of parity opposite to that of $\rf$.

The knowledge of $\Avf$ is also useful for the study of the average behavior of the Selmer and Tate-Shafarevich groups of ${\GFF}(t)$.
There are several conjectures and heuristics for questions on this topic (\cite{bhargava_all}, \cite{poonen_rains}, \cite{delaunay_jouhet}). For example, let $p$ be a prime
number, one of the classical conjecture predicts the probability that the $p$-part of the Tate-Shafarevich group is trivial or not; this original prediction also depends on the rank $\rf(t)$ (at least when $\rf=0$ and $\Avf = 0$ and so when $\rf(t)=0$ or $1$ almost always).
Now, the $p$-Selmer group, denoted by  $\Sel_p({\GFF}(t))$, and the Tate-Shafarevich group,  denoted by $\Sha({\GFF}(t))$, are related by the
following exact sequence
$$
1 \to {\GFF}(t)(\Q)/p{\GFF}(t)(\Q) \to \Sel_p({\GFF}(t)) \to \Sha({\GFF}(t))[p] \to 1
$$
so that  $|\Sel_p({\GFF}(t))| = p^{\rf(t)}  |\Sha({\GFF}(t))[p]|$ if ${\GFF}(t)(\Q)$ has no $p$-torsion point  (in general, $|\Sel_p({\GFF}(t))| = p^{\rf(t)+d}  |\Sha({\GFF}(t))[p]|$ where $d$ is the dimension of
 ${\GFF}(t)(\Q)[p]$ over $\F_p$). So, the $p$-divisibility of  $|\Sel_p({\GFF}(t))|$ and $|\Sha({\GFF}(t))[p]|$ are correlated and depend on the rank, on the
 parity of the root number and so on the average root number. In the case where $\rf(t) >0$, the group $\Sel_p({\GFF}(t))$ is forced to be large because of the presence of $\rf(t)$ generic
 points in ${\GFF}(t)(\Q)$ and one can naturally wonder if those $\rf(t)$ points do contribute in $\Sha({\GFF}(t))[p]$ or not (the answer seems to be no as discussed in a
 forthcoming study).

\medskip
We are then led to define the following notions.
\begin{defin}\label{def_strange}
 Let ${\GFF}$ be a family of elliptic curves with rank $\rf$ over $\Q(t)$. We say that
\begin{itemize}
\item ${\GFF}$ is potentially parity-biased (or also potentially biased) over $\Z$ if it has no place of multiplicative reduction except possibly for the place corresponding to $-\deg$;
\item ${\GFF}$ is parity-biased over $\Z$ if $\Avf$ exists and is non-zero;
\item ${\GFF}$ has elevated rank over $\Z$ if $\Avf$ exists and $\Avf = -(-1)^\rf$.
\end{itemize}
\end{defin}
We shall describe in Section~\ref{ppf} the relation between potentially parity-biased and parity-biased families.

{In this article we focus on potentially parity-biased families of elliptic curves.} In particular,
we classify all non-isotrivial potentially parity-biased families with $\deg a_i(t) \leq 2$ for $i=2,4,6$. We prove that there are essentially 6 different classes of such families\footnote{By the work of Helfgott one also has that, under the square-free sieve and Chowla's conjectures for homogeneous polynomials in two variables,  the only non-isotrivial families with $\deg a_i(t) \leq 2$ for $i=2,4,6$ which can have non-zero average root number over $\Q$ are ${\PFF}_{\r}$ and ${\GG_w}$. See Corollary~\ref{averageq} for the precise statement.}
(cf. Theorem~\ref{classification_1} and Theorem~\ref{classification_2}):
\es{\label{classification}
&{\PFF}_{\r} \colon y^2 =x^3+ 3 t x^2+ 3 \r x  +  \r t  , \mbox{ with } \r\in\Z_{\neq0};\\
&{\GG_w} \colon wy^2 = x^3 + 3tx^2 + 3tx + t^2 , \mbox{ with } w \in\Z_{\neq0};\\
&\HH_{w}\colon w y^2=x^3 + (8 t^2-7t+3) x^2 - 3(2t-1) x + (t+1),\mbox{ with }  w\in\Z_{\neq0};\\
&\II_{w} \colon w y^2=x^3 + t(t-7) x^2 - 6t(t-6) x + 2t(5t-27),\mbox{ with }  w\in\Z_{\neq0};\\
&\JJ_{\n,w} \colon w y^2=x^3 + 3 t^2 x^2 - 3\n t x + \n^2,\mbox{ with }  \n,w\in\Z_{\neq0}.\\
&{\LL}_{w,\r,v}\colon wy^2=x^3+3  (t^2+v) x^2+ 3 \r x +  \r (t^2+v), \mbox{ with }v\in\Z, \r,w\in\Z_{\neq0};\\
}
In Section~\ref{ranks} we will compute the ranks and give generic points for all of families given in~\eqref{classification}. We will see that the rank over $\Q(t)$ of all these families is either $0$ or $1$ (depending on the parameters) except for the family ${\LL}_{w,\r,v}$ for which  the
rank can also be $0$, $1$, $2$, or $3$. We remark that both ${\LL}_{w,\r,v}(t)$ and  ${\GG}_w(t)$ could be expressed in terms of ${\PFF}_{\r}$. Indeed, we have that ${\LL}_{w,\r,v}(t)$ and ${\GG}_{w}(t)$ are isomorphic to ${\PFF}_{\r w^2}(w(t^2+v))$ and $\PFF_{tw^2}(wt)$ respectively.\smallskip
~\\
We can also compute the root number for all the specializations of the above families (the results are quite long to express, so we only give the ones for
${\PFF}_{\r}$ is Appendix~\ref{appen_1}). We use these results to compute their average root numbers in some representative cases, pin-pointing the cases of
parity-biased families and of families with elevated rank over $\Z$ (we are able to provide families of elliptic curves of these types with rank equal to $0, 1, 2$ and $3$).
We postpone the precise statements of our results to Section~\ref{ppf},~\ref{ranks} and~\ref{rootnumbers}. We state here only some examples.

For $a\in\Z_{\neq0}$ we define
$$\Wa \colon y^2 = x^3 + tx^2 - a(t+3a )x + a^3.$$
Notice that the family $\Wa $ is a particular case of the family ${\PFF}_{\r}$. Indeed, one has that $\Wa (t)$  is isomorphic to $\PFF_{- 3a^2/4}(t/3+a/2)$
or, equivalently, to $\PFF_{- 3^54a^2}(12t+18a)$ if one wants a model defined over $\Z$. In Section~\ref{ranks} we shall see that all the families ${\PFF}_{\r}$ which have rank $1$ are all of type $\Wa $.
We will study the root number and the average root number of $\Wa $ in full generality and extract from them several consequences.

In the following the letter $p$ will always denote a prime number. Also, if $n\in \Z_{\neq0}$ then we let $n_p$ be such that $n=p^{v_p(n)}n_p$ where $v_p(n)$ is the $p$-adic valuation of $n$.

\begin{theo} \label{average-periodic}
Let $a\in\Z_{\neq0}$. Then $\Wa $ has rank $1$ over $\Q(t)$ if and only if $a = \pm k^2$ for some $k \in \Z_{\neq0}$, and rank $0$ otherwise. Also, $\eps_{\Wa}(t)$ is periodic modulo $4|a|$ and one has
\es{\label{mf_average-periodic}
\eps_{\Wa}(t) \equiv -  s_a(t) \, \gcd(a_2,t) \,\prod_{p|\frac{a_2}{\gcd(a_2,t)}}(-1)^{1+v_p(t)}\pr{\frac{t_p}{p}}^{1+v_p(t)}\mod 4,\\
}
where $s_a(t)$ is defined in Proposition~\ref{sat}, and is a periodic function modulo $2^{v_2(a)+2}$.
The average root number of the family $\Wa $ is
$$
\Av_\Z(\eps_{\Wa}) = - \prod_{p \mid 2a} \Mw(p),
$$
where $\Mw(p)$ is defined in Proposition~\ref{average_root_W}. In particular, $\Wa $ is a parity-biased family if and only if $v_2(a)\neq1$.
Furthermore, if $a$ is odd and square-free then the average root number of $\Wa $ is
\es{\label{average_A}
\Av_\Z(\eps_{\Wa})= \left\{ \begin{array}{rl}
-1/a & \mbox{ if } a \equiv 1 \pmod{8}, \\
1/(2a) & \mbox{ if } a \equiv 3 \pmod{8}, \\
-1/(2a) & \mbox{ if } a \equiv 5 \pmod{8}, \\
1/a & \mbox{ if } a \equiv 7 \pmod{8}.
\end{array}
\right.
}
\end{theo}

The family $\Wa $ can be seen as a generalization of the well-known Washington's family associated to simplest cubic field and defined by
$$
\was_{1}\colon y^2 = x^3 + tx^2 - (t+3)x + 1.
$$
One can see  that $\was_{1}$ has rank one over $\Q(t)$ (the point $(0, 1)$ is a point of infinite order) and Rizzo proved that for all $t\in\Z$ one has $\eps_1(t) = -1$ (\cite{rizzo}), whence $\Av_\Z(\eps_{\was_1})=-1$. As a consequence of Theorem~\ref{average-periodic}, one can see that $|\Av_\Z(\eps_{\Wa})|=1$ if and only if $a = \pm 1$ and in that case $\Av_\Z(\eps_{\Wa})=-1$ and the rank of $\Wa$ over $\Q(t)$ is 1. So, $\Wa$ can not directly provide
families with elevated rank over $\Z$. However, one can obtain examples of families with elevated rank using subfamilies of $\Wa$; indeed, in Section~\ref{ex_W} we shall prove the following result.
\begin{corol}\label{exess_rank_W_1} Let $p$ be a prime with $p \equiv \pm 1 \mod 8$, and let $a, b \in \Z$ be (non zero) quadratic residue and quadratic non-residue modulo $p$ respectively. Then,
the families
\begin{eqnarray*}
\was^{*}_{p,a} \;:\; y^2 = x^3 + (pt+a)x^2 - (p^3 t + ap^2 + 3p^4) + p^6\\
\was^{**}_{p,b} \;:\; y^2 = x^3 + (pt+b)x^2 - (p^2t + 3p^2 + pb) x + p^3
\end{eqnarray*}
are both families with elevated rank over $\Z$. More precisely, $\was^{*}_{p,a}$ has rank $1$ over $\Q(t)$ with $\varepsilon_{\was^{*}_{p,a}}(t) = 1$ for all $t \in \Z$, and
$\was^{**}_{p,b}$ has rank $0$ over $\Q(t)$ and $\varepsilon_{\was^{**}_{p,b}}(t) = -1$ for all $t \in \Z$.
\end{corol}
We in fact have $\was^{*}_{p,a}(t) = \was_{p^2}(pt+a)$ and $\was^{**}_{p,b}(t) =\was_p(pt+b)$. One can also use
$\Wa (t)$ to construct families with elevated rank and with rank $2$ or $3$ over $\Q(t)$ (see Section~\ref{ex_W}).

We shall also focus on another subfamily of $\PFF_{\r}$, namely the subfamily
$$\Hold \;:\; y^2 =x^3 + 3tx^2 + 3\hr tx + \hr^2 t.$$
Notice that $\Hold(t)$ is isomorphic to $\PFF_{ 4\hr^2}(4t-2\hr)$.%$\Hold = \PFF_{\rm{new}, r^2/4}(t-r/2),$ or

\begin{theo} \label{average-nonperiodic} Let $\hr \in \Z_{\neq0}$. Then,
\begin{eqnarray*}
\varepsilon_{\Hold}(t) &=& w_2(t)  \;w_3(t)\hspace{-1em} \prod_{\substack{p \geq 5 \\ 0 \leq v_p(\hr) \leq v_p(t)}}\hspace{-0.5em}
 \begin{cases}\leg{-3}{p} \leg{3}{p}^{v_p(t)+v_p(t-\hr)+v_p(\hr)} & \mbox{ if } 6 \nmid v_p(t-\hr)-v_p(t)+3v_p(\hr), \\[0.6em]
\ 1 & \mbox{ if } 6 \mid  v_p(t-\hr)-v_p(t)+3v_p(\hr),
\end{cases}\\
%%\leg{3}{p}^{v_p(t)+v_p(t-\hr)+v_p(\hr)} \begin{cases}
 %%\leg{-3}{p} & \mbox{if } 6 \nmid (v_p(t-\hr)-v_p(t)+3v_p(\hr) )\\ \ 1 &
%% \mbox{if } 6 \mid (v_p(t-\hr)-v_p(t)+3v_p(\hr)) \end{cases}  \\
 && \times \prod_{\substack{p \geq 5 \\ 0 \leq v_p(t) < v_p(\hr)}} \begin{cases}
-\left(\frac{3t_p}{p}\right) & \mbox{ if $v_p(t)$ is even} \\
\left(\frac{-1}{p}\right) & \mbox{ if $v_p(t)$ is odd,} \end{cases}
\end{eqnarray*}
where $w_{2}(t)$ and $w_{3}(t)$ are given by Proposition \ref{app-RN2} and \ref{app-RN3} of Appendix \ref{appen_2}.
Also,
$$\Av_\Z(\eps_{\Hold})   =  - \prod_{p\tn{ prime}} \Mh(p), $$
where $\Mh(p)$ are defined in Proposition~\ref{average_root_H}.
In particular, $\Hold$ is a parity-biased family if and only if $v_2(\hr)\neq1$. Finally, if $\hr=\pm 1$ we have
\es{\label{average_B}
\Av_\Z(\eps_{\Hold})  =  - \frac{1}{21} \prod_{\substack {p \geq 5 \\p \equiv 2 \mod 3}}\pr{ 1 - \frac{4(p-1)(p^3+p)}{p^6-1}}\approx 0.038562\dots
}
\end{theo}
We remark that the same method used to prove~\eqref{average_A} and~\eqref{average_B} can also be used for computing the average root number for all the other families given in~\eqref{classification} (conditionally on the square-free sieve conjecture in the case of ${\LL}_{w,\r,v}$ and unconditionally in all other cases, cf. Remark~\ref{remark_squarefree}).

We are also able to obtain new results on the possible numbers that can arise as average root numbers. Before stating them we need some more notation.

We let $\mathfrak F$ be the set of all families of elliptic curves over $\Q$, and $\mathfrak F_{\mathrm i}$ and $\mathfrak F'$ be the subset of $\mathfrak F$ consisting of the isotrivial and of the non-isotrivial families respectively. Furthermore, we let $\mathfrak F_{\Z}$ be the subset of  $\mathfrak F$ consisting of the families $\GFF$ such that $\Avf$ exists. Similarly, define $\mathfrak F_{\mathrm i,\Z}$ and $\mathfrak F_{\Z}'$. By the work of Helfgott (see Theorem~\ref{Helfgott} below) we know, under Chowla's and the square-free sieve conjectures (see the next section for their statements), that $\mathfrak F=\mathfrak F_{\Z}$ and thus %the analogous equalities hold true for $\mathfrak F_{\mathrm i}$ and $\mathfrak F'$ as well
also $\mathfrak F_{\mathrm i}=\mathfrak F_{\mathrm i,\Z}$ and $\mathfrak F'=\mathfrak F'_{\Z}$. Furthermore, we indicate by $\mathfrak F_{\mathrm p}$ the set of families $\GFF$ such that $\eps_\GFF(t)$ is a periodic function for almost all $t\in\Z$ (i.e. the set of exceptional $t$ with $|t|\leq T$ is $o(T)$ as $T\to\infty$).

Finally, with a slight abuse of notation we write $\Av_\Z(\mathfrak F_{\Z}):=\{\Av_\Z(\eps_\GFF)\mid \GFF\in \mathfrak F_{\Z}\}$ and similarly for $\Av_\Z(\mathfrak F_{\mathrm i,\Z})$, etc.
\begin{theorem}\label{rarn}
We have
\est{
\{\Av_\Z(\eps_\GFF)\mid \GFF\in \mathfrak F_{\Z}\}\supseteq \Q\cap[-1,1].
}
In particular, $\Av_\Z(\mathfrak F_{\Z})$ is dense in $[-1,1]$. Moreover, the same result holds true also for $\Av_\Z(\mathfrak F'_{\Z})$ and $\Av_\Z(\mathfrak F_{\mathrm i,\Z})$.
\end{theorem}
Under Chowla's and the square-free sieve conjectures we can also classify all average root numbers that can arise from families with periodic root number.
\begin{theorem}\label{rarn3}
We have
\es{\label{eqfar}
\Av(\mathfrak F_{\mathrm p,\Z})\supseteq \{h/k\in\Q\cap[-1,1]\mid h\text{ odd, and if $k$ even then } |h/k|\leq 1-2^{-v_2(k)}\}.
}
Moreover, assuming Conjecture \ref{chowla} and Conjecture \ref{square-free-1}, the equality holds.
\end{theorem}

We can also obtain an analogue of Theorem~\ref{rarn} in the case of averages over $\Q$. Analogously to $\Avf$, we define
\begin{equation}  \label{averageQ}
\Av_\Q(\eps_{\GFF}) := \lim_{T \rightarrow \infty} \frac{\pi^2}{12T^2} \sum_{\substack{|r|,|s| \leq T,\\ s>0, (r,s)=1}} \varepsilon_{\GFF}(r/s)
\end{equation}
if the limit exists. Also, we let $\mathfrak F_{\Q}$ be the set of families ${\GFF}$ where $\Av_\Q({\GFF})$ exists, and $\mathfrak F_{\mathrm i,\Q}$ and $\mathfrak F'_{\Q}$ be the subsets of $\mathfrak F_{\Q}$ consisting of the isotrivial and non-isotrivial families.

Rizzo~\cite{rizzo_2}, building on the work of Rohrlich~\cite{rohrlich}, proved that $\Av_\Q(\mathfrak F_{\mathrm i,\Q})$ (and thus  $\Av_\Q(\mathfrak F_{\mathrm \Q})$) is dense in $[-1,1]$. {In a similar direction, Klagsbrun, Mazur and Rubin~\cite{KMR} prove an analogous result for the parity of the ranks of the 2-Selmer groups 
in isotrivial families defined over number fields (with techniques which are completely different from the ones of the present paper). The following theorem refine Rizzo's result, showing $\Av_\Q(\mathfrak F_{\mathrm i,\Q})$ actually contains  $[-1,1]\cap\Q$. Moreover, also in this case we are able to address the case of non-isotrivial families, showing that $\Av_\Q(\mathfrak F'_{\Q})$ is dense in $[-1,1]$.
}
\begin{theorem}\label{rarn2}
We have that $\Av_\Q(\mathfrak F_{\Q}')$ is dense in $[-1,1]$. Moreover, we have that $\Av_\Z(\mathfrak F_{\mathrm i,\Q})\supseteq [-1,1]\cap\Q$ (and thus, a fortiori, $\Av_\Q(\mathfrak F_{\Q})\supseteq [-1,1]\cap\Q$).
\end{theorem}

Notice that in the case of $\Av_\Q(\mathfrak F_{\Q}')$ we do not get $[-1,1]\cap\Q$. We remark that there are reasons to believe that in fact $\Av_\Q(\mathfrak F'_{\Q})\cap \Q=\{0\}$. Indeed, by the work of Helfgott (cf. \cite[Appendix A]{CCH}) and Desjardins~\cite{desjardins_thesis} one has (conjecturally) that for a non-isotrivial family $\GFF$ with $\Av_\Q(\rootf)\neq0$, the average root number of $\GFF$ over $\Q$ can be written as a convergent Euler product $\Av_\Q(\rootf)=c_\infty \prod_p(1-r_p)$ for some $c_\infty \in\overline{\Q}\cap[-1,1]$ and some $r_p$ which are rational polynomials in $p$ satisfying $0\leq r_p(p)<1$, $r_p\ll p^{-2}$ for all $p$ and $r_p>0$ for infinitely many $p$.\footnote{In the proof of Corollaire~2.5.4. of~\cite{desjardins_thesis} it is shown that $r_p>0$ for at least one $p$, but the same proof actually carries over to show that there are infinitely many such $p$.}
In particular, one has $\Av_\Q(\rootf)\neq \pm1$, and one also expects that such an infinite product is not algebraic.

We shall prove Theorem~\ref{rarn},~\ref{rarn3} and~\ref{rarn2} by considering subfamilies of $\Wa(t)$ where both $t$ and the parameter $a$ are replaced by polynomials in $\Z[t]$. By Theorem~\ref{average-periodic}, we know exactly the root number for all the elliptic curves in these families, and so the problem becomes that of choosing suitably these polynomials.
In the case of averages over $\Q$, we can reduce to the case where the $\infty$-factor of the root number essentially determines $\rootf(t)$, whereas in the the case of Theorem~\ref{rarn} we work with the $p$-factor of the root number for a suitably chosen $p$. The proof of Theorem~\ref{rarn3} is a bit more elaborate and requires dealing with the factors of the root number corresponding to all prime divisors of $6k$.

\medskip
The organization of the paper is as follow. In Section~\ref{ppf} we discuss more in depth the work of Helfgott and we give our classification of the potentially parity-biased families with coefficients of low degree. In Section~\ref{ranks} we compute the ranks for the families given in~\eqref{classification}. In Section~\ref{rootnumbers} we compute the average root numbers of the families $\Wa$ and $\Hold$. In Section~\ref{densityaverage} we use the results proven in Section~\ref{rootnumbers} to prove Theorem~\ref{rarn},~\ref{rarn3} and~\ref{rarn2}. Finally, in Appendix~\ref{appen_1} and~\ref{appen_2} we give the local root numbers of the families $\PFF_{\r}$ and $\Hold$.
Finally, this work led to many technical computation (root number, local average of root number, ...), we used the
PARI/GP software (\cite{pari}) to intensively check them when it was relevant.

\medskip
{\bf Acknowledgements.} The authors would like to thank Jake Chinis, Julie Desjardins, Ottavio Rizzo, Joseph Silverman and Jamie Weigandt for helpful discussions.
This work was initiated while the first author was a post-doctoral fellow at the Centre de Recherche Math\'ematiques (CRM) in Montr\'eal, and completed during several visits of the third author at the CRM, and we are grateful to the CRM for providing very good facilities. The research of the second author is partially supported by the National Science and Engineering Research Council of Canada (NSERC). The third author was partially supported by the R\'egion Franche-Comt\'e (Projet R\'egion). 

\section{The classification of potentially parity-biased families of low degree}\label{ppf}

\subsection{The work of Helfgott and its consequences}
We start with a more detailed discussion of the work of Helfgott (\cite{helfgott,helfgott_thesis}) which gives (conditionally) a necessary condition for a family to be potentially parity-biased.
First, we state the following conjectures.

\begin{conj}[Chowla's conjecture] \label{chowla} Let $P(x)\in\Z[x]$ be square-free. Then, $\sum_{n\leq N}\lambda(P(n))=o(N)$ as $N\to\infty$, where $\lambda(n)$ is the Liouville function $\lambda(n):=\prod_{p|n}(-1)^{v_p(n)}$.
\end{conj}
Moreover, by \emph{strong Chowla's conjecture} for a polynomial $P$ we mean the assumption that Chowla's conjecture holds for $P(ax+b)$ for all $a,b\in\Z$, $a\neq0$.

\begin{conj}[Square-free sieve conjecture] \label{square-free-1} Let $P(x)$ be a square-free polynomial in $\Z[x]$. Then, the set of integers $n$ such that $P(n)$ is divisible by the square of a prime which is larger than $\sqrt n$ has density $0$.
\end{conj}

Conjectures~\ref{chowla} and~\ref{square-free-1} are believed to hold for all square-free polynomials $P$. Chowla's conjecture is known for polynomials of degree $1$ only, whereas the square-free sieve conjecture is known for polynomials whose irreducible factors have degrees $\leq 3$ (\cite{Helfgott_square_sieve}).

\begin{theorem}[Helfgott]\label{Helfgott}
Let $\GFF$ be a family of elliptic curves. Let $M_{\GFF}(t)$ and $B_{\GFF}(t)$ be the polynomials defined by
\es{\label{dfn_MB}
M_{\GFF}(t):=\prod_{\substack{v\text{ mult},\\ v\neq-\deg}}Q_v(t),\qquad B_{\GFF}(t):=\prod_{\substack{v\text{ quite bad},\\ v\neq-\deg}}Q_v(t)
}
where the products are over the valuations $v$ of $\Q(t)$ where $\GFF$ has multiplicative and quite bad\footnote{That is, if no quadratic twist of $\GFF$ has good reduction at $v$.} reductions respectively and where $Q_v(t)$ is the polynomial associated to the place $v$.
Then for all but finitely many $t\in\Z$ one has
\est{
\eps_\GFF(t)=\sign (g_\infty(t))\lambda(M_{\GFF}(t))\prod_{p\text{ prime}}g_p(t)
}
where $g_\infty(t)$ is a polynomial, $S$ is a finite set of (rational) primes depending on ${\GFF}$ and $g_p:\Q_p\to\{\pm1\}$ are functions satisfying
\begin{itemize}
\item[a)] $g_p(t)$ is locally constant outside a finite set of points;
\item[b)] if $p\notin S$ then $g_p(t)=1$ unless $v_p(B_{\GFF}(t))\geq2$.
\end{itemize}

Moreover, if $\GFF$ has at least one place of multiplicative reduction other than $-\deg$, then assuming the square-free sieve conjecture for $B_\GFF(t)$ and the strong Chowla's conjecture for $M_\GFF(t)$, one has $\Avf=0$.

If $\GFF$ has no place of multiplicative reduction other than $-\deg$ (i.e. $\GFF$ is potentially parity-biased), then assuming the square-free sieve conjecture for $B_\GFF(t)$ we have
\es{\label{formsn}
\Avf=\frac{c_-+c_+}{2}\prod_p\int_{\Z_p}g_p(t)\,dt,
}
where $dt$ denotes the usual $p$-adic measure and $c_{\pm}=\lim_{t\to\pm\infty} g_{\infty}( t)$.

\end{theorem}
The case where $\GFF$ is potentially parity-biased was also considered by Rizzo~\cite{rizzo} in two examples which already contain several of the important ideas for the general result. We also mention the recent work of Desjardins~\cite{desjardins} who revisited Helfgott's result, and relaxed some of the assumptions.

We now give a sketch of the proof of Helfgott's result, as it reveals quite clearly the way to proceed when computing $\Avf$.	
\begin{proof}[Sketch of the proof of Theorem~\ref{Helfgott}]
The root number of an elliptic curve $\GFF(t)$ in the family $\GFF$ is defined as a product of local root numbers $\rootf(t)=-\prod_pw_{p}(t)$, where $w_{p}(t)$ depends on the reduction type of $\GFF(t)$ modulo $p$.
Naively, one might expect that
$\Avf=-\prod_p\int_{\Z_p}w_{p}(t) dt$ however this is false in general (the product on the right is typically non-convergent). One can however modify the $w_{p}(t)$ to some $w^*_p(t)$ so that one still has $\eps(t)=-\prod_pw^*_p(t)$, but in this case (conjecturally) $\Avf=-\prod_p\int_{\Z_p}w^*_p(t) dt$.

First, we recall that for $p\geq 5$ one has that $w_p(t)=1$ if $\GFF(t)$ has good reduction at $p$, $w_p(t)=\leg{j_p(t)}p$ (where $\leg{ \cdot}{\cdot}$ is the Kroeneker symbol) if $\GFF(t)$ has bad, non-multiplicative reduction at $p$, where $-j_p(t)=1,2,3$ depending on the reduction type\footnote{For example $j_p=-3$ if $v_p(c_4(t),c_6(t),\Delta(t))\equiv (r,2,4)\mod {12}$ for some $r\geq2$ or $v_p(c_4,c_6,\Delta)\equiv (r,4,8)\mod{12}$ for some $r\geq3$).}, and $w_p(t)=-\pr{\frac {j_p(t)}p}$ if $\GFF(t)$ has multiplicative reduction at $p$ where in this case $j_p(t)$ is the first non-zero $p$-adic digit of the invariant $c_6(t)$.
%$\eps_p(t)=-\pr{\frac {c_6(t)'}p}$ if $\GFF(t)$ has multiplicative reduction at $p$ (where $c_6(t)'$ indicates the first non-zero $p$-adic digit of $c_6(t)$ and $\pr{\frac{\cdot}{\cdot}}$ is the Kroeneker symbol), $\eps_p(t)=\pr{\frac {-j_p(t)}p}$ if $\GFF(t)$ has bad, non-multiplicative reduction at $p$, where $j_p(t)=1,2,3$ depending on the type of bad reduction, and $\eps_p=1$ if $\GFF(t)$ has good reduction at $p$.
Thus,
\es{\label{formularn}
\rootf(t)=-w_2(t)w_3(t)(-1)^{\#\{p:\GFF(t)\text{ has mult. red. at }p\}}\prod_{p\text{ bad}}\pr{\frac{j_p(t)}{p}}
}
where the product is over primes $p$ such that $\GFF(t)$ has bad reduction at $p$. Now, the key step is to observe that essentially $\GFF(t)$ has a certain reduction type at $p$ if and only if there is a place $v\neq -\deg$ over $\Q(t)$ where $\GFF$ has the same reduction type and $p|Q_v(t)$. The only exceptions to this are when $p$ is in a finite set of primes $S$ (depending on the family $\GFF$ only, essentially this amounts to excluding the finitely many primes that divide more than one $Q_v$) and when $p$ divides $Q_v(t)$ with multiplicity greater than $1$. It follows that~
\est{
\rootf(t)&=-(-1)^{\#\{p\notin S,\, p||Q_v(t)\text{ $v$ place of mult red.}\}}\prod_{p\in S}w_p(t)\cdot\prod_{\substack{v\text{ bad},\\ v\neq -\deg}}\prod_{\substack{p||Q_v(t),\\ p\notin S}}\pr{\frac{j_p(t)}{p}}\cdot\prod_{\substack{v\text{ bad},\\ v\neq -\deg}}\prod_{\substack{p^2|Q_v(t),\\ p\notin S}}h_p(t)\\
&=(-1)^{\#\{p|M_{\GFF}(t)\}}\prod_{p\in S}w^*_p(t)\cdot \prod_{i=1}^3
\prod_{\substack{ v\in V_i}}\prod_{p}\pr{\frac{-i}{p}}^{v_p(Q_v(t))}\\[-0.5em]
&\hspace{19em}\times \prod_{\substack{v\text{ mult.}}}\prod_{p}\pr{\frac{c_6'(t)}{p}}^{v_p(Q_v(t))}
\cdot\prod_{\substack{v\text{ bad},\\ v\neq -\deg}}\prod_{\substack{p^2|Q_v(t),\\ p\notin S}}h^*_p(t)\\
}
for some finite set of primes $S$, some functions $h_{p}(t), h^*_{p}(t), w_p(t), w^*_{p}(t)$ which are $p$-locally constants on $\Z_p$ outside a finite set of points, and a suitable partition $V_1\cup V_2\cup V_3$ of the set $\{v\text{ bad}, v\neq -\deg\}$.\footnote{For example $V_3$ is the set of places $w\neq -\deg$ of $\Q(t)$ such that $v_{w}(c_4(t),c_6(t),\Delta(t))\equiv (r,2,4)\mod {12}$ for some $r\geq2$ or $v_p(c_4,c_6,\Delta)\equiv (r,4,8)\mod{12}$ for some $r\geq3$).} Then,\footnote{We ignore the minor issue of the case where the top and bottom of the various Legendre symbol are not coprime.}
\est{
\rootf(t) &=(-1)^{\#\{p|M_{\GFF}(t)\}}\prod_{p\in S}w^*_p(t)\cdot \prod_{i=1}^3
\prod_{\substack{ v\in V_i}}\pr{\frac{-i}{Q_v(t)}}\cdot \prod_{\substack{v\text{ mult.}}}\pr{\frac{c_6(t)}{Q_v(t)}}
\cdot\prod_{\substack{v\text{ bad},\\ v\neq -\deg}}\prod_{\substack{p^2|Q_v(t),\\ p\notin S}}h^*_p(t)\\
}
Now, applying repeatedly quadratic reciprocity one sees that the factor  $\leg{c_6(t)}{Q_v(t)}$ also depends on the $\Z_q$ expansion of $t$ at finitely many primes $q$ and on the sign of a polynomial and the same is true for $\leg{-i}{Q_v(t)}$.\footnote{For example, if $Q_v(t)\neq0$, then $\leg{-1}{Q_v(t)}=\sign(Q_v(t))\chi(Q_v(t)_{2})$, where $\chi$ is the non-principal character $\mod 4$ and $Q_v(t)_{2}=Q_v(t)2^{-v_2(Q_v(t))}$.
} Finally, one can verify directly that $h^*_p(t)=1$ if $\GFF(t)$ has bad but not quite-bad reduction at $p$. Thus,
\est{
\rootf(t)&=\lambda(M_{\GFF}(t))\sign(h_\infty(t))\prod_{p\in S'}w^{***}_p(t)\prod_{\substack{p^2|B_{\GFF}(t),\\ p\notin S'}}h^{*}_p(t)\\
}
for a finite set of primes $S'$, some $w_p^{***}(t)$ $p$-locally constant outside a finite set of points and a polynomial $h_\infty(t)$. Thus, we obtain the first assertion of Theorem~\ref{Helfgott}. The other assertions are easy once one observes that the square-free sieve conjecture for $B_{\GFF}(t)$ gives
\est{
\lim_{T\to\infty}\frac1{T}\sum_{0\leq \pm t\leq T}\rootf(t)=\lim_{X\to\infty}\lim_{T\to\infty}\frac{c_\pm}{T}\sum_{0\leq \pm t\leq T}\lambda(M_{\GFF}(\pm t))\prod_{p\in S'}w^{***}_p(\pm t)\prod_{\substack{p^2|B_{\GFF}(\pm t),\\ p\notin S',\\ p\leq X}}h^{*}_p(\pm t).
}
Notice that the product on the right involves finitely many primes (for each $X$). Thus, if $M_{\GFF}(t)\neq1$ then dividing into congruence classes modulo these primes one has that the strong Chowla's conjecture for $M_{\GFF}(t)$ gives that the average is $0$. Otherwise, the limit over $T$ coincides with the product of the $p$-adic integral (see~\cite{helfgott} or also~\cite{rizzo}) and, writing $h^{*}_p(t)=1$ if $ v_p(B_{\GFF}(t))<2$, we have
\est{
\lim_{T\to\infty}\frac1{T}\sum_{0\leq\pm t\leq T}\rootf(t)&=c_{\pm}\lim_{X\to\infty}\prod_{p\in S'}\int_{\Z_p}w^{***}_p(t)\,dt\prod_{\substack{ p\notin S',\\ p\leq X}}\int_{\substack{\Z_p}}h^{*}_p(t)\,dt\\
&=c_\pm\prod_{p\in S'}\int_{\Z_p}w^{***}_p(t)\,dt\prod_{\substack{ p\notin S'}}\int_{\substack{\Z_p}}h^{*}_p(t)\,dt
}
with $\int_{\substack{\Z_p}}h^{*}_p(t)\,dt=1+O(p^{-2})$.
\end{proof}

\medskip
Notice that Theorem~\ref{Helfgott} implies, under Chowla's and the square-free sieve conjectures, that $\Avf$ exists for all families $\GFF$. Moreover, recalling Definition~\ref{def_strange}, we have the following implications
$$
\begin{array}{c} \mbox{{elevated}} \\ \vspace*{-0.5cm} \\ \mbox{{rank}} \end{array} \Longrightarrow \quad \begin{array}{c} \mbox{{Parity}} \\ \vspace*{-0.5cm} \\ \mbox{{biased}} \end{array} \quad \displaystyle \xRightarrow[{\mathrm {\scriptstyle Helfgott}}]{{\rm Conj.}} \begin{array}{c} \mbox{{Potentially}} \\ \vspace*{-0.5cm} \\ \mbox{{parity-biased.}} \end{array}
$$
The first implication is trivial and the converse is false in general since there are examples with
$\Avf \not \in \{-1,0,1\}$ (see~\cite{rizzo, helfgott} or also Theorem~\ref{average-periodic}). The
second implication comes from Theorem~\ref{Helfgott} and is conjectural in
full generality. The converse is also false in general (see Theorem~\ref{average-periodic} with $a=2$), however assuming the square-free sieve conjecture
one has that every potentially biased family $\GFF$ has a parity-biased subfamily (obtained by taking $t$ to be in an arithmetic progression and/or restricting to $t>0$\footnote{Alternatively one can for example replace $t$ by $t^2$.}). Indeed, the potentially biased families are the ones for which some of the integrals in~\eqref{formsn} are equal to $0$ or with $c_+=-c_-$. Fixing the sign of $t$ and restricting $t$ to be in a suitably selected congruence class one can make those integrals (as well as all the other ones) non-zero.

\subsection{Potentially biased families}

In this section, we find all potentially biased families such that $\deg a_i(t) \leq 2$. We first start by the
case with a family ${\GFF}(t) \colon y^2 = x^3 + a_2(t) x^2 + a_4(t) x + a_6(t)$
where  $\deg a_2 \leq 1$ and $\deg a_4, a_6 \leq 2$.
We write $a_2(t)=ut+v$, $a_4(t)=at^2+bt+c$ and $a_6(t)=dt^2+et+f$. We denote by $c_4, c_6, \Delta$ and $j$ the classical
invariants of ${\GFF}(t)$. Notice that the potentially biased condition is
equivalent to the fact that all the roots of $\Delta$ are also roots of $c_4$. One can see that if $\Delta$ is
constant then either $\Delta=0$ or the family does not depend on $t$ (i.e. $a=b=d=e=u=0$).% One verifies this directly

We also notice that if $u^2-4a$ is a square, say $u^2-4a=r^2$ for $r\in\Q$, then the family doesn't change under the transformation
\es{\label{invariance}
a \leftrightarrow \tfrac12 (-4 a + u^2 - ur),\quad
u \leftrightarrow \tfrac12 (-u + 3 r),\quad
b \leftrightarrow b - uv + rv,\\
e \leftrightarrow \tfrac12 (2 e - cu + cr),\quad
d \leftrightarrow \tfrac12 (2 d - bu + br - 2 av + u^2 v - urv)
}
(and a suitable linear transformation in $x$).

\begin{theorem}\label{classification_1} Let $a_2(t), a_4(t)$ and $a_6(t)$ be polynomials in $\Q[t]$ with $\deg a_2(t) \leq 1$, $\deg a_4(t) \leq 2$, $\deg a_6(t) \leq 2$ and such that
the curve ${\GFF}(t) \colon y^2 = x^3 + a_2(t)x^2 + a_4(t)x +a_6(t)$ is non-isotrivial and potentially parity-biased. Then the family has rank $\leq 1$ over $\Q(t)$ and, up to some rational linear change of variables in the parameter $t$ and in the variables $x$ and $y$, the family is either
$$
{\PFF}_{\r} \colon y^2 = x^3 + 3 t x^2 + 3 \r x +\r t \\
$$
for some $\r \in \Z_{\neq0}$ and with rank $1$ if and only if $\r= - {12} k^4$ with $k\in\N$; or
$$
{\GG_w} \colon wy^2 = x^3 + 3tx^2 + 3tx + t^2
$$
for some $w \in \Z_{\neq0}$ and with rank $1$ if and only if $w$ is a square or $-2$ times a square.
\end{theorem}
\begin{proof}
Here we shall only show that all non-isotrivial and potentially parity-biased families satisfying the above conditions are of the form ${\PFF}_{\r}$ or ${\GG}_w$. We will compute their ranks in Section~\ref{ranks}, Propositions~\ref{rank_F} and~\ref{rank_G}.

We remind that $c_4= 16 (a_2^2 - 3 a_4)$, $c_6=32(-2 a_2^3 + 9 a_4 a_2 - 27 a_6)$ and $1728\Delta=c_4^3-c_6^2$. Thus, with our assumptions and the discussion above, we have $1 \leq \deg c_4\leq 2$\expl{ The root of $\Delta$ are the root of $c_4$, so if $\deg c_4=0$ (and $c_4\neq0$) then also $\deg\Delta=0$ and so the family is isotrivial, by the above, if $c_4=0$ then $j=0$ and the family is isotrivial}%
, $\deg c_6 \leq 3$ and one easily checks that $\deg \Delta$ can't be $1$\expl{It's impossible with polynomials of degree $2$ and $3$ and in any case it would be of multiplicative reduction.}
and thus $2 \leq  \deg \Delta \leq 6$. \smallskip
~\\
Now, we observe that $c_4$ has to be square-free. Indeed, if $c_4=\ell L^2$ for some $\ell \in \Q_{\neq0}$ (for $\ell=0$ the family is iso-trivial) and a degree one polynomial $L$, then since $\GFF$ is potentially parity-biased we can write $\Delta$ as $1728\Delta=k L^{m}$ for some $k\in\Q_{\neq0}$ and some $2\leq m\leq 6$. Then we have $c_6^2=L^m(\ell P^{6-m}-k)$, so $m\in\{0,2,4\}$ (if $m=6$, then the family is iso-trivial) and $\ell L^{6-m}-k$ is a square in $\Q(t)$ which is clearly not possible since $\ell L^{6-m}-k$ is square-free.
Also, we must have $\deg(c_4)\geq2$. Indeed, if $c_4=L$, $1728\Delta=k L^{m}$, for some linear polynomial $L\in\Q(t)$, some $k\in\Q_{\neq0}$ and some $m\in\N$, then we'd have $L^3-kL^m$ is a square in $\Q(t)$ which is clearly not possible.\expl{If $m>3$ then $v_L(L^3-kL^m)=3$, if $m\leq 3$, then $L^3-kL^m$ has odd degree unless it's zero.}
\smallskip
~\\
Now, suppose that $c_4$ is square-free of degree $2$. Then, $c_4=L_1L_2$ for some coprime linear polynomials $L_1,L_2\in\C[t]$. We can write $\Delta$ as $1728\Delta=k L_1^{m}L_2^n$ for some $k\in\Q_{\neq0}$ and some $m,n\in\N$ with $2\leq m+n\leq 6$, $m\leq n$. Thus $L_1^mL_2^n(L_1^{3-m}L_2^{3-n}-k)=c_6^2$ is a square in $\Q(t)$. In particular, it can't be $m=n=3$, nor $m=0,n=3$ (since $L_1^3-k$ is square-free). Moreover, it can't be $m=0,n=2$ as this would imply that $L_1^3L_2-k$ is a square in $\C(t)$ which is not possible (indeed, we can assume $L_1=t$ and $L_2=t-1$, then for $k\neq0$ the discriminant of $t^3(t-1)-k$ is zero only for $k=-\frac{27}{256}$ in which case $t^3(t-1)-k$ is not a square). Thus, we must have either $m=n=2$ or $m=2,n=3$. Thus, we have two cases:
\begin{enumerate}
\item $c_4=P$, $1728\Delta=k P^2$ for some $P\in\Q[t]$ of degree $2$ and some $k\in\Q_{\neq0}$;
\item $c_4=L_1L_2$, $1728\Delta=kL_1^2L_2^3$ for some coprime $L_1,L_2\in\Q[t]$ of degree $1$  and some $k\in\Q_{\neq0}$.
\end{enumerate}
First, let's consider the case where $c_4$ has degree 2 and $ 1728\Delta = kc_4^2$. Notice that we can not have $a=0$ and $u=0$ at the same time
(otherwise $c_4$ would be a degree $\leq 1$ polynomial). Since $\deg \Delta = 4$, we must have $a=d=0$ or $a=u^2/4$ and $d=(2ub-u^2v)/4$.\expl{Indeed $$1728\Delta/2^{10}=27 a^2  (u^2-4 a)t^6+54(-6 a^2 b+a^2 u v+a b u^2+9 a d u-2 d u^3)t^5+O(t^4).$$ For $a=0$ the $t^5$ term becomes $-2du^3$, whereas for $a=u^2/4$ it becomes $1/16 u^3 (u (u v-2 b)+4 d)$. %The coefficient of the $t^4$ term is $-27(12 a^2 c-a^2 v^2+12 a b^2-4 a b u v-2 a c u^2-18 a d v-18 a e u-b^2 u^2-18 b d u+27 d^2+12 d u^2 v+4 e u^3)$ which for $a=d=0$ becomes $u^2(4eu-b^2)$ and in the other case it becomes $u^2(4 b^2 +u (8 e-8 b v)+u^2 (4 v^2-4 c))$.
} By the transformations~\eqref{invariance} the two cases lead to the same families, so it suffices to consider the case $a=d=0$ only.

Now, we have $c_6^2=c_4^2(c_4-k)$ and so $c_4-k$ is a square in $\Q[t]$, a condition which univocally determines $k$ in terms of the other parameters.
With this choice for $k$ we have $c_4-k=16(ut+v-\frac{3b}{2u})^2=(4a_2-\frac{6b}{u})^2$. Thus, we have $c_6=\pm c_4(4a_2-\frac{6b}{u})$ and comparing the terms of degree $3$ in $t$ we see that we must take the minus sign. Expressing $c_4$ and $c_6$ in terms of the $a_i$ and simplifying this equality becomes
%$$2(-2 a_2^3 + 9 a_4 a_2 - 27 a_6)= -(a_2^2 - 3 a_4)(4a_2-\frac{6b}{u})$$
%$$6 a_4 a_2 - 54 a_6= a_2^2\frac{6b}{u}-3 a_4\frac{6b}{u}$$
$$6( a_4 -a_2 b/u) a_2 - 54 a_6= -18 a_4{b}/{u}$$
or, equivalently, $6(c -bv/u) a_2 +18 a_4{b}/{u}= 54 a_6.$ Comparing the terms of degree $1$ and $2$ in $t$ we obtain
\est{
\begin{cases}
6(c -bv/u) u +18 {b^2}/{u}=54 e,\\
6(c -bv/u) v +18 {bc}/{u}=54 f
\end{cases}
\Rightarrow
\begin{cases}
c  =(9 eu+bvu-3 {b^2})/{u^2},\\
 f=(3 ebu- {b^3}+ evu^2)u^3
\end{cases}
}
(remember we have $u\neq0$) and so we are led to the families
$$
y^2 =x^3 + (ut + v)x^2 + \left(bt + \frac{9 eu+bvu-3 {b^2}}{u^2}\right)x +et +\frac{3 ebu- {b^3}+ evu^2}{u^3}.
$$
We make the changes $b\leftrightarrow -bu$ and $e\leftrightarrow eu$ in order to kill the denominator  and we make the change of variables $ut+v \leftrightarrow t$. We arrive to%
\expl{
The associated invariants are
\begin{eqnarray*}
c_4 & = & 2^4(t^2 + 3bt + 9b^2 - 27e) \\
c_6 & = & - 2^5(t^2 + 3bt + 9b^2 - 27e) (2t+3b) \\
\Delta & = &2^4 (b^2-4e) (t^2 + 3bt + 9b^2 - 27e)^2\\
j & =& \frac{2^8}{b^2-4e} (t^2 + 3tb + 9b^2 - 27e).
\end{eqnarray*}
}
\es{\label{firstformF}
{\PFF} \colon y^2 = x^3 + t x^2 + (-bt  -3b^2 + 9e)x +  et  + b^3 - 3eb.
}
Finally, we make the changes of variables $t \lra 3 t - 3 b/2$, $x \lra x + b/2$ and write $e = \r/3 + b^2/4$ and obtain ${\PFF}_{\r}$, with associated invariants
\es{\label{parf}
c_4(t) & =  144(t^2-\r), \\
c_6(t) & =  - 1728 t(t^2-\r), \\
\Delta(t) & = -1728 \r (t^2-\r)^2.
}
Notice that if $d\neq0$, then the changes $t\lra t/d^2$, $x\lra d^2 x$, $y\lra d^3 y$ transform ${\PFF}_{\r}$ into $\PFF_{\r d^4}(t)$ and so in particular we can always reduce to the case where $\r\in\Z_{\neq0}$.

Now, let's consider the second case. Up to a linear change of variables in $t$ we can assume $c_4$ has the form $c_4= C t(t-1)$ and $1728\Delta=2^{12}kt^3(t-1)^2$ for some $k\neq0$. Comparing this expression of $c_4$ with its definition, we see that we must have $C=2^4(u^2-3a)$, $c=v^2/3$ and $b=\frac13 (u^2 + 2 u v-3a)$.
Since $\deg \Delta = 5$, we have $a=0$ or $4a-u^2=0$ and as before it suffices to consider the case $a=0$ (and hence $u \neq 0$). Now, $c_6^2=c_4^3-1728\Delta= 2^{12}u^6t^3(t-1)^2(t-1-k)$ and thus $k=-1$ and so $c_6=\pm 2^6u^3t^2(t-1)$ and again we need to take the minus sign. Comparing the coefficients of the polynomials in $t$ we find
\est{
\begin{cases}
-864 f + 32 v^3=0 \\
-864 e  + 96  u^2 v + 96 u v^2=0\\
-864 d + 96 u^3=64  u^3
\end{cases}
\Rightarrow
\begin{cases}
f=v^3/27 \\
e=(u^2 v + u v^2)/9\\
d=  u^3/27
\end{cases}
}
Making the change of variable $x\lra x-v/3$, we then obtain that the dependence of $v$ disappear and we obtain the families
$$
y^2=x^3 + t u x^2 + \tfrac 13 t u^2 x + \tfrac 1{27}u^3 t^2.
$$
Writing $u=3w$ and making the change of variables $x\lra w x$ and $y\lra w^2y$ we obtain $\GG_{w}(t)$ with
\es{\label{parg}
c_4(t) & =   12^2 w^2 t (t-1),\\
c_6(t) & =  -12^3 w^3 t^2 (t-1), \\
\Delta(t) & =  -12^3 w^6 r t^3(t-1)^2.%\\
%j & =& \frac{-1728}{r}(t-r).
}
\end{proof}

We can extend Theorem~\ref{classification_1} to the case where $\deg a_2(t)=2$. First we give the following Lemma which will allow us to exclude several cases. One could also rule out these cases by using Kodaira's classification of singular fibers~\cite{miranda}.

We remark that when performing the computations needed in the proofs of Lemma~\ref{lemmapol} and Theorem~\ref{classification_1} we used Mathematica and PARI/GP.

\begin{lemma}\label{lemmapol}
Let $R_1,R_2$, $S$ and $L$ be polynomials in $\C[t]$ of degree $2$, $2$, $3$ and $1$ respectively. Let $k\in\C\setminus\{0\}$. Then
\begin{itemize}
\item[a)]$R_1^3-k$ can't be divisible by the square of a degree $2$ polynomial  in $\C[t]$.
\item[b)]$R_1^3R_2-k$ can't be a square in $\C[t]$.
\item[c)]$LR_1^3-k$ can't be divisible by the square of a  degree $3$ polynomial in $\C[t]$.
\item[d)]$LS^3-k$ can't be a square in $\C[t]$.
\item[e)]$S^3-k$ can't be divisible by the square of a degree $4$ polynomial in $\C[t]$.
\item[f)]$L^4R_1-k$ can't be a square in $\C[t]$.
\end{itemize}
\end{lemma}
\begin{proof}
We only prove the first two statements, the proofs of the other ones being very similar.
\begin{itemize}
\item[a)] We can assume $R_1=t^2+1$ or $R_1=t^2$. In the second case the statement is obvious, thus assume $R_1=t^2+1$. The discriminant of $R_1^3-k$ is $6^6k^4(k-1)$ and thus it is zero only if $k=1$, but $R_1^3+1= t^2 (3 + 3 t^2 + t^4)$ and the second factor is not a square.
\item[b)] We can assume $R_2=t^2+bt+c$ and $R_1=t^2+1$ or $R_1=t^2$; we consider the first case only, the second one being a bit simpler. If  $C:=R_1^3R_2-k$ is the square of a degree $4$ polynomials, then $C$ and $C'$ have at least $4$ zeros in common and thus the first $4$ subresultants of $C$ and $C'$ are zero.  The fourth subresultant is a non-zero multiple of $(b^2-4c	)(c-1+ib)^2(c-i-ib)^2$. If $c=b^2/4$, then the third subresultant is a non-zero multiple of $(4+b^2)^2$ and thus we need $a=\pm 2i$, but with this choice the second subresultant is $2^{17}\cdot3\cdot 5\cdot k^5\neq0$. If $c=1\pm ib$, then the third subresultant is zero when $b= 2i$ or $b=0$ and in both cases the second subresultant is non-zero. Thus the first four subresultants of $C$ and $C'$ can't be all zeros and so $C$ can't be the square of a degree $4$ polynomial.
\end{itemize}
\end{proof}

\begin{theorem}\label{classification_2} Let $a_2(t), a_4(t)$ and $a_6(t)$ be polynomials in $\Q[t]$ with $\deg a_2(t) = 2$, $\deg a_4(t) \leq 2$, $\deg a_6(t) \leq 2$ and such that the curve ${\GFF}(t) \colon y^2 = x^3 + a_2(t)x^2 + a_4(t)x +a_6(t)$ is non-isotrivial and potentially biased. Then, up to some rational linear changes of variables in the parameter $t$ and in the variables $x$ and $y$, the family is one of the following
\es{\label{we_classification_2}
&\HH_{w}\colon w y^2=x^3 + (8 t^2-7t+3) x^2 - 3(2t-1) x + (t+1),\mbox{ with }  w\in\Z_{\neq0},\\
&\II_{w} \colon w y^2=x^3 + t(t-7) x^2 - 6t(t-6) x + 2t(5t-27),\mbox{ with }  w\in\Z_{\neq0},\\
&\JJ_{\n,w} \colon w y^2=x^3 + 3 t^2 x^2 - 3\n t x + \n^2,\mbox{ with }  \n,w\in\Z_{\neq0},\\
&{\LL}_{w,\r,v}\colon wy^2=x^3+3  (t^2+v) x^2+ 3 \r x +  \r (t^2+v), \mbox{ with }v\in\Z, r,w\in\Z_{\neq0}.\\
}
Moreover, the ranks of $\HH_{w}$, $\II_w$ and $\JJ_{\n,w}$ are $\leq1$. Also, $\II_w$ and $\JJ_{\n,w}$ have rank $1$ if and only if $w$ is a square, whereas $\rank_{\Q(t)} ( \HH_w)=1$ if and only if $w$ is $2$ times a square. Finally, the rank of ${\LL}_{w,\r,v}$ is always $\leq3$ and its value is given in equation~\eqref{formula-rank-L} below in terms of the number of irreducible factors of certain polynomials.
\end{theorem}
\begin{proof}
We will compute the ranks in Section~\ref{ranks}; here we only show that all the potentially biased are the ones given in~\eqref{we_classification_2}.

First, we observe that the $c_4$ and the $c_6$ invariants of ${\GFF}$ have degree $4$ and $6$ respectively since $c_4=16a_2(t)^2-48a_4(t)$ and $c_6=-64a_2(t)^3+288a_4(t)a_2(t)-864a_6(t)$. Also, $3\leq \deg(\Delta)\leq8$. Indeed all the terms of degree $\geq9$ trivially cancel, whereas imposing that the coefficients of degree $5,6,7,8$ cancel we can determine $d,e,f,c$; then, with this choices, the coefficients of degree $4$ and $3$ can be zero at the same time only if $bw-au=0$\expl{More precisely, denoting the coefficients by $\alpha_3$ and $\alpha_4$ we have $\alpha_3=\frac{2u}w\alpha_4+\frac{8(bw-au)}{w^3}$.} but in that case $\Delta=0$.

Now, we notice that we can assume $c_4$ is square-free. Indeed, if it is not, write $c_4=L_1^2L_2L_3$, with $L_i$ linear polynomials in $\C[t]$, and, since the family is potentially parity-biased% and non-isotrivial
,  $1728\Delta=kL_1^{i}L_2^{j}L_3^h$ with $i\in\{0,2,4\}$, $j, h\in\{0,2,3\}$, $3\leq 2i+j+h\leq 8$, and $k\neq0$. Thus,
$$c_6^2=c_4^3-1728\Delta=L_1^{2i}L_2^{j}L_3^h(L_1^{6-2i}L_2^{3-j}L_3^{3-h}-k)$$
and in particular $L_1^{6-2i}L_2^{3-j}L_3^{3-h}-k$ is $L_2^{\frac{1-(-1)^j}2}L_3^{\frac{1-(-1)^h}2}$ times a square in $\C[t]$ (if $L_2=L_3$ and $j=h=3$, and so $i\in\{0,2\}$, the condition would be that $L_1^{6-2i}-k$ is a square, which is clearly not possible). One can then easily rule out all the possibilities by Lemma~\ref{lemmapol}.\expl{We can assume $j\leq h$.\begin{itemize}
\item If $i=4$, then $i=j=0$ and $L_1^2L_2^3L_3^3-k=Q^2$ is a square and this is not possible for b).
\item If $i=2$, then either $j=0,h=3$, (so $L_1^4L_2^3-k=L_3Q^2$, not possible by c)) $j=0,h=2$ (so $L_1^4L_2^3L_3-k=Q^2$, not possible by b)), $j=0,h=0$ (so $L_1^4L_2^3L_3^3-k=Q^2$, not possible by d)), $j=2,h=2$ (so $L_1^4L_2L_3-k=Q^2$, not possible by f)).
\item If $i=0$, then either $j=0,h=3$ (so $L_1^6L_2^3-k=L_3Q^2$, not possible by e)), $j=2,h=2$ (so $L_1^6L_2L_3-k=Q^2$, not possible by b)), $j=2,h=3$ (so $L_1^6L_2-k=L_3Q^2$, not possible by b)), $j=h=3$ (so $L_1^6-k=L_2L_3Q^2$, obviously not possible).
\end{itemize}
}

Thus, we can assume $c_4=L_1L_2L_3L_4$ with $L_i\in\C[t]$ different linear polynomials. Since the family is potentially parity-biased and non-isotrivial, we have $1728\Delta=kL_1^{i}L_2^{j}L_3^hL_4^g$ with $k\neq0$ and $i,j,h,g\in\{0,2,3\}$. Clearly at least two among $i,j,h,g$ coincide and so we write $c_4=L_1L_2P$ with $P$ of degree $2$, and $1728\Delta=kL_1^{i}L_2^{j}P^h$ and we can assume $i\leq j$. Also $3\leq i+j+2h\leq8$. Then, the only possibilities are: $i=j=h=2$ and $i=0$ and either $j=h=2$ or $j=2,h=3$ or $j=3$ and $h=2$ (we excluded the cases $i=2,h=2,j\in\{2,3\}$ by b) and c) of Lemma~\ref{lemmapol}, and $i=j=3, h=0$ by a)).
It follows that  there are only the following possibilities for $c_4$ where $P_i$ and $P_{i,j}$ are polynomials in $\Q[t]$ of degree $i$, not necessary irreducible:
\begin{enumerate}
\item $c_4=P_4$ and $\Delta=kP_4^2$,
\item $c_4=P_{1,1}P_{1,2}P_2$ and $\Delta=kP_{1,2}^{3}P_2^{2}$,
\item $c_4=P_{1,1}P_{1,2}P_2$ and $\Delta=kP_{1,2}^{2}P_2^{3}$,
\item $c_4=P_3P_1$ and $\Delta=kP_3^2$,
\end{enumerate}
with $c_4$ square-free in all cases.

Let's consider the case 1). Comparing the coefficients in $t$ of the equation $c_6(t)^2=c_4(t)^3-kc_4(t)^2$ we obtain $9$ equations in the various parameters  (since the terms of degree $>8$ in $t$ are always equal). Imposing the equality of the coefficients coming from the degree $8,7$ and $6$ we can easily express $d,e$ and $f$ in terms of the other variables. The coefficient of degree $5$ factors and gives rise to two possibilities; one leads to $k=0$  (after eliminating other variables looking at the coefficients of lower degrees) and so $\Delta=0$,\expl{More directly, writing $\alpha_i$ for the coefficient of degree $i$, instead of considering $\alpha_5=0$ one can consider $\alpha_3-\frac{2u}w\alpha_4+\frac{3 a + 3 u^2 - 2 v w}{w^2}\alpha_5=0$ and this, with the previous choice for $d,e,f$, is equivalent to $a u - b w=0$, so that one can discard immediately the other possibility.} whereas the other one, eliminating $b$ and $c$, leads to the following families (after a linear change of variable in $k$)\expl{One arrives to
\est{
y^2=x^3+a_2(t)x^2+\pr{\frac{k}{48} - \frac{3 a^2}{4 w^2} +\frac{a_2(t)}w}x+\frac{(108 a^2 w + k w^3)a_2(t)-108 a^3 + 3 a k w^2}{432 w^3}
}
and then one put $432 kw^3\leftrightarrow 108 a^2 w + k w^3$.
}
\est{
y^2=x^3+a_2(t)x^2+\pr{\frac{a}wa_2(t)+9k-\frac{3a^2}{w^2}}x+k a_2(t)-\frac{a^3}{w^3}+\frac{3ak}{ w}
}
with $a_2(t)=w t^2+ut+v$. Up to some change of variables, this is of the form~\eqref{firstformF}, with $t$ replaced by $a_2(t)$, so killing $u$ with a change of variable in $t$, we see that we obtain families of the form $\PFF_{h}(p(t^2+q))$ for some parameters $h,p,q\in\Q$ with $h,p\neq0$. Writing $v=q$, $p=w$, $r=h/w^2$ and making the changes $x\lra w x$, $y\lra wy$ we obtain ${\LL}_{w,\r,v}(t)$ with invariants
\es{\label{parl}
c_{4}(t)&= 144 ( t^4 + 2 t^2 v + v^2-\r),\\
c_6(t)&=-1728 (t^2 + v) (\r - t^4 - 2 t^2 v - v^2),\\
\Delta(t)&=-1728 \r  ( t^4 + 2 t^2 v + v^2-\r)^2.
}

Now, consider the case 2). We can assume $c_4(t)$ has simple zeros at $0$ and $1$ and that $\Delta(t)$ has a triple zero at $0$. Thus, we can write $c_4(t)$ in the form $c_4(t)=-16 w^2 t (t - 1) ( t^2 + m t + n)$ and $\Delta(t)$ as  $1728\Delta(t)=k t^3 (t^2 + m t + n)^2$ for some $n,k\in\Z_{\neq0}$ and some $m\in\Z$. Then, we express $a,b,c,u$ in terms of $m,n$ and the other parameters and we impose
\es{\label{feq3}
c_6^2(t)=c_4(t)^3-k t^3 (t^2 + m t + n)^2
}
obtaining $9$ equations for the parameters. The equations corresponding to the degrees $0,2,3$ and $8$ in $t$ allow us to express $f,e,k,d$ in terms of the other parameters. Then, we use a suitable linear combination of the equations from the degrees $5,6$ and $7$ to obtain a linear equation in $n$, so that we can eliminate $n$ as well.\expl{The coefficient multiplying $n$ could be zero, but using the resultant in $m$ we can see that it is zero at the same time as the degree constant term (in $n$) only if $r=0$ which we excluded.} Then, (eliminating the denominator) the equations from the degrees $6$ and $7$ states that two polynomials in $m$ are equal to $0$. The common roots of these polynomials are $m=-1$ and $m=-\frac{11}3$. The former gives $k=0$, i.e. $\Delta=0$, whereas for $m=-\frac{11}3$ we see that~\eqref{feq3} is verified and so we have new families. After a change of variable in $x$ to reduce the degree in $t$ of the coefficient of $x$, the families are
\est{
y^2=x^3+\frac 13 (8 w - 7  t w + 3 t^2 w) x^2  +
 \frac{16}{27} (4 w^2 - 3 t w^2) x +\frac{64}{729} (8 w^3 + 3 t w^3).
}
We make the changes of variables $x\lra \frac{8}9 w x$, $t\lra 8 t/3$, $w\lra 2w$, $y\lra \frac{16}{27} w^2 y$ and we arrive to $\HH_{w}$
%\est{
%\HH_{w}\colon w y^2=x^3 + (8 t^2-7t+3) x^2 - 3(2t-1) x + (t+1)
%}
with
\es{\label{parh}
c_{4}(t)&=16 w^2 t (8t-3) (8 t^2 - 11 t +8), \\
c_6(t)&=-64w^3 t^2 (8 t^2 - 11 t +8) (64 t^2 - 80 t + 45) ,\\
\Delta(t)&=-512 w^6 t^3 (8 t^2 - 11 t +8)^2.
}

Now, consider the case 3). Again we write $c_4(t)$ as $c_4(t)=-16 w^2 t (t - 1) ( t^2 + m t + n)$ and this time $\Delta(t)$ as  $1728\Delta(t)=k t^2 (t^2 + m t + n)^3$ with $k,n\neq0$. We impose $c_6^2(t)=c_4(t)^3+k t^2 (t^2 + m t + n)^3
=0$ and proceed as above, expressing $f,k,e$ in terms of the other parameters, using the equations from the degrees $1,8$ and $7$ in $t$. A suitable linear combination of the $5th$ and $6th$ equations give an equation of the form $(1 + m)^2 (5 + 2 m)d=Q(m,n,v,w)$ for some polynomial $Q(m,n,v,w)$. We have $Q(-1,n,v,w)=-\frac{(nw)^3}{27}\neq0$ and thus we can assume $m\neq-1$. We now assume $5 + 2 m\neq0$, and we express $d$ in terms of the other variables and with this choice the remaining equations don't depend on $w$ and $v$ anymore. Thus we are left with $4$ independent equations equating polynomials in $m,n$ to zero, the resultants in $n$ of the $2$nd and $3$th polynomials and of the $4$th and $5$th polynomials have the only common zero $m=-\frac52$ which we had excluded. Finally, we consider the case $m=-\frac 52$. With this choice we can quickly determine also $n$ (whereas the dependence on $v$ disappear also in this case) and, after some changes of variables in $y,x,t$ we are led to the families $\II_{w}$
%\est{
%\II_{w} \colon w y^2=x^3 + t(t-7) x^2 - 6t(t-6) x + 2t(5t-27)
%}
with
\es{\label{pari}
c_{4}(t)&=16w^2 (t-4) t (t^2- 10 t + 27) , \\
c_6(t)&=-64 w^3 (t-1) t (t^2- 10 t + 27)^2 ,\\
\Delta(t)&=-64 w^6 t^2 (t^2- 10 t + 27)^3.
}

Finally let's consider case 4). With a change of variables we can write $c_4(t)$ as $c_4(t)=-16 w^2 t (t - \delta) ( t^3 + m t + n)$ with $k,n\neq 0$, $\delta\in\{0,1\}$ and $1728\Delta(t)=k (t^3 + m t + n)^2$. As above we expres $a,b,c,u$ in terms of the other variables and we impose $c_6^2(t)=c_4(t)^3+k (t^3 + m t + n)^3=0$, from which we can easily eliminate $d,e,k,f$. If $\delta=1$, then a linear combination of the remaining equations give $(3 + 4 m)^4 (2 m - 10 n-1)=0$ and in both cases one finds $k=0$. Thus, we can take $\delta=0$; this simplifies the remaining equations and we can eliminate $m$ and $f$ arriving to the families\expl{Notice that if $n$ is a cube this is isomorphic to the case $n=1$.} $\JJ_{{n},w}$
%\est{
%\JJ_{\n,w} \colon w y^2=x^3 + 3 t^2 x^2 - 3n t x + n^2
%}
with
\es{\label{parj}
c_{4}(t)&=144 w^2 t (t^3+n) , \\
c_6(t)&=-864 w^3(t^3+n)(2t^3+n) ,\\
\Delta(t)&=-432 w^6 n^2(t^3+n)^2.
}

Finally we observe that, up to rational linear changes of variables in $t,x$ and $y$, one can always reduce to the case where the parameters $w,r,v,n$ are in $\Z$.
\end{proof}

\begin{remark}\label{remark_squarefree}
Note that, with the exception of ${\LL}_{w,\r,v}$, all the families of Theorem~\ref{classification_1} and Theorem~\ref{classification_2} don't have primes of bad reduction of degree greater than $3$. In particular, the square-free sieve conjecture is known to hold for the associated polynomial $B$ defined in~\eqref{dfn_MB}. In the case of ${\LL}_{w,\r,v}$, one has that if $t^4 + 2 t^2 v + v^2-\r$ is irreducible (i.e. if $\r$ is not of the form $n^2$ nor $-4 n^2 (n^2 + v)$ for some $n\in\N$), then there is a prime of quite bad reduction of degree $4$ for which the square-free sieve conjecture is not proven.
\end{remark}

\begin{corol}\label{averageq}
Assume Chowla's conjecture and the square-free sieve conjecture for homogeneous square-free polynomials (Hypothesis $\mathscr{B}_1$ and $\mathscr{B}_2$ at page 5 in~\cite{helfgott}).
Then, all non-isotrivial families $\GFF$ of the form~\eqref{eq_family}, with $a_2(t), a_4(t),a_6(t)\in\Q[t]$  and $\deg(a_i)\leq2$ for $i=2,4,6$, for which we have $\Av_\Q(\eps_{\GFF})\neq0$ are of the form ${\PFF}_{\r}$ or $\GG_w$ up to some rational linear changes of variables in the parameter $t$ and in the variables $x$ and $y$.
\end{corol}
\begin{proof}
By Main Theorem~2 of~\cite{helfgott}, we have that if the family $\GFF$ is not potentially parity-biased or it has multiplicative reduction at infinity, then $\Av_\Q(\eps_{\GFF})=0$. Thus, we just need to check which of the families in Theorems~\ref{classification_1} and~\ref{classification_2} have multiplicative reduction at infinity. By~\eqref{parf}, \eqref{parg},~\eqref{parl} and~\eqref{parh}-\eqref{parj} one immediately sees that the only families with non-multiplicative reduction at infinity are ${\PFF}_{\r}$ and $\GG_w$.

\end{proof}

\section{Ranks over $\Q(t)$}\label{ranks}

In this section we compute the rank of the potentially parity-biased families given in Theorem~\ref{classification_1} and Theorem~\ref{classification_2}  following the same approach as in~\cite{arms_all}.
Let $\GFF$ be a family of elliptic curves as defined by \eqref{eq_family}, and
suppose that ${\GFF}$ is not isotrivial. The following conjecture %%(\cite{rosen-silverman})
gives a way to determine $r_\GFF$, the rank of ${\GFF}$
over $\Q(t)$, by considering averages of the traces of Frobenius at $p$ of the specializations $\GFF(t)$, when $t$ varies over $\F_p$.
More precisely, writing the number of points of the specialization $\GFF(t)$ over the finite field $\F_p$ as
$p+1 - a_{\GFF(t)}({p})$ (with $a_{\GFF(t)}({p}) =0$ for $p$ dividing
the discriminant of $\GFF(t)$), we define
$$
A_\GFF({p}) := \frac{1}{p}\sum_{t=0}^{p-1} a_{\GFF(t)}({p}).
$$

\begin{conj}[Nagao]\label{nagao} With the  notation above, the rank of ${\GFF}$ over $\Q(t)$ is given by
\es{\label{nagaof}
r_\GFF = \lim_{X\rightarrow \infty} \frac{1}{X} \sum_{p \leq X} -A_\GFF({p})\log({p})
}
where the sum runs through all prime numbers $p \leq X$. \end{conj}
As proved in \cite{rosen-silverman},  Tate's conjecture implies  Nagao's conjecture, and thus
Conjecture~\ref{nagao} holds for rational elliptic surfaces as Tate's conjecture holds in that case~\cite{rosen-silverman}. In particular, the conjecture holds if $\deg a_i(t) \leq 2$ for $i=2,4,6$, since in this case ${\GFF}$ is a rational elliptic surface (see \cite[section 8]{schutt_shioda}).

\medskip
If $\deg a_i(t) \leq 2$, we have that $\deg_t (x^3 + a_2(t)x^2 + a_4(t)x + a_6(t)) \leq 2$ and
we can rewrite \eqref{eq_family} as
$$
{\GFF} \colon y^2 = A(x) t^2 + B(x) t +C(x),
$$
where $A(x), B(x)$ and $C(x)$ are in $\Q(x)$.
Now, we have
\es{\label{defap}
\sum_{t \mod p} a_{\GFF(t)}({p}) &= - \sum_{t \mod p} \sum_{x \mod{p}} \left(\frac{x^3 + a_2(t)x^2 + a_4(t)x + a_6(t)}{p}\right)\\
& = - \sum_{x \mod p} \sum_{t \mod{p}} \left(\frac{A(x)t^2 + B(x) t + C(x)}{p}\right),\\
}
and we can evaluate the sum over $t$ for each fixed value of $x$.

We will need the following formulas (see for example \cite[Theorem~5.48]{finite-fields}):
\begin{equation}\label{quad_sum_2}
\sum_{t=0}^{p-1} \left(\frac{Bt+C}{p}\right) = \left\{ \begin{array}{ll} p\left(\frac{C}{p}\right) & \mbox{ if } p\mid B \\ 0 & \mbox{ otherwise.} \end{array} \right.
\end{equation}
If $A$ is non-zero mod $p$, then
\begin{equation}\label{quad_sum}
\sum_{t=0}^{p-1} \left(\frac{At^2+Bt+C}{p}\right) =-\left(\frac{A}{p}\right)+\left\{ \begin{array}{ll} p\left(\frac{A}{p}\right) & \mbox{ if } p\mid (B^2-4AC), \\0& \mbox{ otherwise.} \end{array} \right.
\end{equation}

We now use the above setting to compute the rank of the families of Theorem~\ref{classification_1} and~\ref{classification_2} over $\Q(t)$.
\footnote{The study of more general formula for the rank whenever $\deg a_i(t) \leq 2$ is a work on progress and will appear in a forthcoming paper.} In all cases
we shall also give explicitly non-torsion points. Note that to prove that an explicit point, say $G \in {\GFF}(\Q(t))$, is non-torsion, it is sufficient to prove that it
is neither a 3 nor a 4 torsion point (indeed, we only have to check a point is non-torsion when the rank of the family is $\rf \geq 1$ which implies the torsion subgroup has  cardinality at most $4$, see \cite{oguiso-shioda}).  In order to do this, it is sufficient to compute $2G$ and check that its $y$-coordinate is non-zero and that its $x$-coordinate is not the $x$-coordinate of $G$ (indeed if $3G$ is zero then $2G=-G$ and the $x$-coordinates of $G$ and $2G$ would coincide). We shall show this explicitly for Proposition~\ref{rank_F} only, the computation being completely analogous in all other cases.

\smallskip
~\\
We first compute the ranks of the family ${\PFF}_{\r}$, for which $A(x)=0$.
\begin{prop}\label{rank_F} Let $r \in \Z_{\neq 0}$ such and let ${\PFF}_{\r}$ be the family
$$
{\PFF}_{\r} \colon y^2=x^3+3tx^2+3\r x + \r t .
$$
Then, $\rank ({\PFF}_{\r}) \leq 1$, and $\rank ({\PFF}_{\r}) =1$ if and only if $\r = - {12} k^4,$ k in $\N$. Furthermore, if $\r = - {12} k^4,$ then
$(-2k^2,2^3 k^3)$ is a non-torsion point of ${\PFF}_{-12k^4}(\Q(t))$.

\end{prop}
%Note that if $b^2-4e=0$ then the equation $y^2=x^3+tx^2+(-bt-3b^2+9e)x + et +b^3-3eb$ is singular.
\begin{proof}
We have $B(x)=3x^2+\r$ and $C(x)=x^3+3\r x$. If $-\r/3$ is a square mod $p$, then the 2 roots of $B(x)$ are $\pm x_p$, where $x_p = \sqrt{-\r/3}$, and $C(\pm x_p) = \pm \frac{8}{3} \r x_p.$  Then, using~\eqref{nagaof},~\eqref{defap} and~\eqref{quad_sum_2}, we have that
\est{
\rank ({\PFF}_{\r})  %&=\lim_{X\rightarrow \infty} \frac{1}{X}\sum_{p\leq X} \frac{\log p}{p}\sum_{t=0}^{p-1} \sum_{x=0}^{p-1} \left(\frac{B(x)t+C(x)}{p}\right) \\
&= \lim_{X\rightarrow \infty} \frac{1}{X}\sum_{p\leq X} \frac{\log p}{p}\sum_{x=0}^{p-1} \sum_{t=0}^{p-1} \left(\frac{B(x)t+C(x)}{p}\right) \\
& =  \lim_{X \rightarrow \infty} \frac{1}{X} \sum_{\substack{p \leq X \\ \leg{-3\r}{p}=1}} \log{p} \left( \leg{6 \r x_p}{p} + \leg{-6 \r x_p}{p} \right)  \\
&= \lim_{X \rightarrow \infty} \frac{2}{X} \sum_{\substack{p \leq X \\ \leg{-3\r}{p}=\leg{-1}{p}=1}} \log{p} \; \leg{2x_p}{p},
}
since for $\leg{-3\r}{p}=\leg{-1}{p}=1$ one has $\leg{6 \r x_p}{p}=\leg{2 x_p}{p}$.
Now if $-3\r$ is neither a square nor minus a square in $\Z$, then the proportion of primes counted in the sum is $1/4$, and hence we obtain $\rank ({\PFF}_{\r}) <1$, which implies that $\rank ({\PFF}_{\r}) =0$.
If $-3\r= - n^2$ for some $n \in\Z_{\neq 0}$ then
$$
\rank (\PFF_\r) =  \lim_{X\rightarrow \infty} \frac{2}{X}  \sum_{\substack{p\leq X \\ \left(\frac{-1}{p}\right)=1} } {\log p} \left(\frac{6n}{p}\right) \leg{\delta_p}{p},
$$
where $\delta_p^2 \equiv -1 \mod p$.
Note that the sum does not depend on the choice of the sign of $\delta_p$; also, $\leg{\delta_p}{p}=1$ if and only if $p \equiv 1 \mod 8$. We claim that there is a positive proportion of the primes $p \equiv 1 \mod 4$ such that
\es{\label{formforf}
\leg{6n}{p} \leg{\delta_p}{p} = -1
}
so that in particular $\rank ({\PFF}_{\r})$ has to be $0$ also in this case.
Indeed, if the square-free part of $6n$ is $2q$ with $q$ odd, then take $p \equiv 5 \mod 8$ and $p \equiv b \mod q$, where $b$ is a quadratic non-residue modulo $q$ so that by quadratic reciprocity one obtains~\eqref{formforf}.
Similarly, if the square-free part of $6n$ is $q$ with $q$ odd, then take $p \equiv 5 \mod 8$ and $p \equiv b \mod q$, where $b$ is a non-zero quadratic residue modulo $p$.

Finally, if $-3\r = n^2$, then
$$
\rank (\PFF_\r)  = \lim_{X\rightarrow \infty} \frac{2}{X} \sum_{\substack{p\leq X \\ \left(\frac{-1}{p}\right)=1} } {\log p} \left(\frac{6n}{p}\right).
$$
If $6n \not= \pm k^2$, then the sum is 0; if $6n = \pm k^2$, then the sum is $1$, and in that case we have that $\r = - 12 k^4$. Finally, if $\r=-12k^4$, then the point
$G=(-2k^2,8k^3)$ is a point on ${\PFF}_{-12k^4}(\Q(t))$ and we have
$$2G=\left(\frac{9t^2-12k^2t+100k^4}{16k^2}, \frac{27t^3 + 18k^2t^2 + 324k^4t + 280k^6}{64k^3}\right)
$$
and so $G$ is neither a 3 nor a 4 torsion point.
\end{proof}

\begin{corol}\label{rank_W}
Let $\Wa: y^2 = x^3 + tx^2 - a(t+3a )x + a^3.$
Then $\rank (\Wa) \leq 1$, and the rank is 1 if and only if $a = \pm n^2$.
\end{corol}
\begin{proof} It follows from the fact that $\Wa(t)$ is isomorphic to $\PFF_{-3^54a^2} (12t + 18a)$. \end{proof}

\begin{corol}\label{rank_H} Let $\Hold : y^2 =x^3 + 3tx^2 + 3\hr tx + \hr^2 t$. Then, $\rank ({\Hold})=0$.
\end{corol}
\begin{proof} It follows from the fact that ${\Hold}$ is isomorphic to $\PFF_{4\hr^2} (4t-2\hr)$. \end{proof}

We now compute the rank of the remaining families of Theorem~\ref{classification_1} and Theorem~\ref{classification_2}.
First, we remark that for $\GFF$ as in~\eqref{eq_family} one has that the quadratic twist
$$\GFF^{(w)} \;:\; w y^2 = x^3 + a_4(t) x^2 + a_4(t) x + a_6(t), $$
satisfies
$a_{\GFF^{(w)}(t)}(p) = \leg{w}{p} a_{\GFF(t)}(p).$
Then, using~\eqref{quad_sum_2} and~\eqref{quad_sum} one has
\es{\label{formula-degree2}
\rank (\GFF^{(w)}) &= \lim_{X \rightarrow \infty} \frac{1}{X} \sum_{p \leq X} \log{p}  \bigg(
\sum_{\substack{ x \mod p \\ A(x) \equiv B(x) \equiv 0 \mod p}}  \leg{wC(x)}{p}  \\
& \hspace{9.3em}  + \hspace{-2em}
\sum_{\substack{ x \mod p \\ A(x) \not\equiv 0 \mod p \\ (B^2 - 4AC)(x) \equiv 0 \mod p}}  \leg{wA(x)}{p}
- \frac{1}{p} \sum_{\substack{ x \mod p \\ A(x) \not\equiv 0 \mod p}}  \leg{wA(x)}{p}
 \Bigg).
}
We note that the contribution from the first sum is zero unless $A$ and $B$ have common factors in $\Q[x]$, whereas the contribution from the last sum is
$$\begin{cases}
-1 & \text{if $wA(x) = P(x)^2$ for some } P(x) \in \Q[x], \deg{P(x)} \geq 1\\
0 & \mbox{otherwise}
\end{cases}$$
using Weil's bound (or also~\eqref{quad_sum} if $A$ has degree $2$).

We now use the above formula to compute the rank of the families $\GG_w$, $\II_{w}$, $\HH_w$, $\JJ_{\n, w}$ and $\LL_{w,\r,v}$, the most delicate being the last one.

\begin{prop}\label{rank_G} Let $w\in \Z_{\neq0}$ and $\GG_{w}$ be the family
$$
{\GG}_w \colon w y^2=x^3 + 3tx^2 + 3tx + t^2.
$$
Then, the rank of ${\GG}_w$ is $\leq 1$ and $\rank ( \GG_w)=1$ if and only if $w$ is a square or $-2$ times a square. Furthermore, if $w=1$
(resp. $w=-2$) then the point $(0,t)$ (resp. $(-3,2t)$) is a non-torsion point in ${\GG}_w(\Q(t))$.
\end{prop}
\begin{proof}
We use equation \eqref{formula-degree2} with $A(x)=1$, $B(x)=3x(x+1)$ and $C(x)=x^3$ so that
$(B^2-4AC)(x)=x^2(9x^2+14x+9)$. Notice that the discriminant of $9x^2+14x+9$ is $-2^5$ and thus for $p>3$ one has that $B^2-4AC$ has $3$ distinct roots in $\F_p$ if $\leg{-2}p=1$ and has $1$ root otherwise. Since $A(x_p)=1$ for any root $x_p$ of $(B^2-4AC)(x)$, we have
\est{
\rank (\GG_w) &= \lim_{X \rightarrow \infty} \frac{1}{X} \Bigg( \sum_{\substack{p \leq X \\ \leg{-2}p=1}}  3 \log p \leg{w}{p} +
\sum_{\substack{p \leq X \\ \leg{-2}p=-1}}   \log p \leg{w}{p} -\sum_{p \leq x} \log{p} \leg{w}{p}  \Bigg) \\
&= \lim_{X \rightarrow \infty} \frac{1}{X}  \sum_{\substack{p \leq X \\ \leg{-2}p=1}}  2 \log p \leg{w}{p}
= \lim_{X \rightarrow \infty} \frac{1}{X}  \sum_{\substack{p \leq X}}   \log p\pr{\leg{w}{p}+\leg{-2w}{p}}
}
and so the rank is $1$ if $w$ is a square or $-2$ times a square, and it is zero otherwise.
\end{proof}

\begin{prop}\label{rank_I} Let $w\in \Z_{\neq0}$ and $\II_{w}$ be the family
$$
{\II}_w \colon w y^2=x^3 + t(t-7) x^2 - 6t(t-6) x + 2t(5t-27)
$$
Then, the rank of ${\II}_w$ is $\leq 1$ and $\rank ( \II_w)=1$ if and only if $w$ is a square. Furthermore, if $w=1$ then the point $(9/4,5t/4-27/8)$ is a non-torsion point in ${\II_1}(\Q(t))$.
\end{prop}
\begin{proof}
We use equation \eqref{formula-degree2} with $A(x)=x^2-6x+10$, $B(x)=-7x^2+36x-54$ and $C(x)=x^3$ so that
$B^2-4AC=-(4x-9)(x^2-8x+18)^2$. Note that $A(9/4)=25/16=(5/4)^2$. Also, the discriminant of $x^2-8x+18$ is $-8$, and so if $-2$ is a square modulo $p$ then the roots of $x^2-8x+18$ are $x_{p,\pm}= 4\pm \delta_p$ where $\delta_p^2 = -2$ modulo $p$.  Moreover, we have $A(x_{p,\pm}) = \pm 2\delta_p$ and so
by equation \eqref{formula-degree2}, we obtain
$$
\rank (\II_w) =  \lim_{X \rightarrow \infty} \frac{1}{X} \Bigg(\sum_{p \leq X} \log p \leg{w}{p} + 2\sum_{\substack{p \leq X \\ \leg{-2}{p}=\leg{-1}{p}=1}}  \log{p} \leg{2\delta_pw}{p} \Bigg).
$$
The contribution coming from the first sum is 1 if and only if $w$ is a square (and $0$ otherwise), whereas the contribution coming from the last sum is 0
(the proportion of primes counted in this sum is $1/4$).
\end{proof}

\begin{prop}\label{rank_K} Let $w\in\Z_{\neq0}$ and $\HH_{w}$ be the family
$$
{\HH}_w \colon w y^2=x^3 + (8 t^2-7t+3) x^2 - 3(2t-1) x + (t+1).
$$
Then, the rank of ${\HH}_w$ is $\leq 1$ and $\rank ( \HH_w)=1$ if and only if $w$ is 2 times a square. Furthermore, if $w=2$ then
$(-1,2t)$ is a non-torsion point of $ {\HH_2}(\Q(t))$.
\end{prop}
%%%%END OF KOMMENTAR

\begin{proof} We have $A(x)=8x^2$, $B(x)=-(x+1)(7x-1)$ and $C(x)=(x+1)^3$. In particular, for $p \neq 2$ we have $A(x) =0 \mod{p}$ if and only if $x =0 \mod{p}$. Also, notice that $A(x)$ is always $2$ times a square.
Then, using
\eqref{formula-degree2}, we get
\est{%\label{sum2}
\rank (\HH_w) &= \lim_{X \rightarrow \infty} \frac{1}{X}\Bigg( \sum_{ p \leq X} \log{p}
\sum_{\substack{ x \mod p \\ (B^2 - 4AC)(x) \equiv 0 \mod p}} \leg{2w}{p}- \sum_{ p \leq X} \log{p} \leg{2w}{p}\Bigg)
}
We compute that $B^2-4AC= -(x+1)^2(32x^3 - 17x^2 + 14x - 1)$, and we let $Q(x) = 32x^3 - 17x^2 + 14x - 1$ with discriminant $-2\cdot(320)^2$.
Then, we have
\est{
\rank (\HH_w) =\lim_{X \rightarrow \infty} \frac{1}{X}\Bigg( \sum_{p\leq X \atop Q(x) \rm{~has~one~root~in~}  \F_p}  \log p \leg{2w}{p} +\sum_{p\leq X \atop Q(x) \rm{~has~three~root~in~}  \F_p} 3 \log p \leg{2w}{p} \Bigg).
}
Now, $Q(x)$ has exactly $1$ root in $\F_p$ if and only if $\leg{-2\cdot(320)^2}{p}=-1$. Thus, the first sum contributes
\est{
\lim_{X \rightarrow \infty} \frac{1}{X} \sum_{ p \leq X} \frac{\log{p}}2 \leg{2w}{p}\pr{1-\leg{-2}{p}}=\begin{cases}
\frac12&\text{if $2w$ is a rational square,}\\
-\frac12&\text{if $-w$ is a rational square.}\\
\end{cases}
}
Finally, since the Galois group of $Q(x)$ is $S_3$, we have that the primes such that $Q(x)$ splits completely have density $\frac16$, so that the contribution of the second sum is in $[-\frac12,\frac12]$. Also, such a contribution is positive if $2w$ or $-w$ are rational square (indeed in the second case $\leg{2w}{p}=\leg{-2}{p}=1$ if $Q$ splits completely in $\F_p$). Thus, since $\rank ({\HH}_w)$ is an integer, we must have that $\rank ({\HH}_w)=1$ if $2w$ is a rational square and $\rank ({\HH}_w)=0$ otherwise.
\end{proof}
\begin{prop}\label{rank_J}  Let $\n, w\in\Z_{\neq0}$ and $\JJ_{\n,w}$ be the family
$$
{\JJ}_{\n,w} \colon w y^2=x^3 + 3t^2x^2-3\n tx+\n^2.
$$
Then, $\rank({\JJ}_{\n,w}) \leq 1$ and  $\rank({\JJ}_{\n,w}) =1$  if and only if $w$ is a square. Furthermore, if $w=1$ then
$(0,\n)$ is a non-torsion point in ${\JJ}_{\n,1}(\Q(t))$.
\end{prop}
\begin{proof} In this case, we have $A(x)=3x^2$, $B(x)=-3\n x$, $C(x)=x^3 + \n^2$ and $B^2(x)-4A(x)C(x) = -3x^2 (4x^3+\n^2)$. Let $Q(x) = 4x^3 + \n^2$ and note that the discriminant of $Q(x)$ is $-3\cdot (2\n)^4$ and that $A(x)$ is always $3$ times a square. Also, $A(x)$ has a common root $x=0$ and one has $C(0)=\n^2$. Thus, using \eqref{formula-degree2}, we obtain
\est{
\rank ({\JJ}_{\n,w}) &= \lim_{X \rightarrow \infty} \frac{1}{X} \sum_{p \leq X} \log{p}  \Bigg(\leg{w}{p} +
\sum_{\substack{p\leq X \\ Q \rm{~splits}\\\rm{completely~in~}\F_p}} 3  \leg{3w}{p} +\sum_{\substack{p\leq X \\ Q \rm{~has~one}\\\rm{root~in~} \F_p}}  \leg{3w}{p} - \leg{3w}{p}\Bigg).\\
}
The first term contributes $1$ if $w$ is a rational square and $0$ otherwise and the last contributes $1$ if $3w$ is rational square and $0$ otherwise.
As in the proof of Proposition~\ref{rank_K} one then has that the third term contributes $\frac12$ if $3w$ is a square, $-\frac12$ if $-w$ is a square, and $0$ otherwise. Thus, we must have that the second summand contributes $\frac12$ if $3w$ or $-w$ are squares and $0$ otherwise.
\end{proof}
\begin{prop}\label{rank_L} Let $v\in \Z$, $\r, w\in\Z_{\neq0}$ and ${\LL}_{w,\r,v}$ be the family
\es{\label{eq-rank-L}
{\LL}_{w,\r,v}\colon wy^2=x^3+3  (t^2+v) x^2+ 3 \r x +  \r (t^2+v).
}
Then the rank of ${\LL}_{w,\r,v}$ is $\leq 3$ and all the cases can occur.
More precisely, let
$$C(x)=x^3+3vx^2+3\r x+\r v, \quad R(x)=x^6 + (15 \r w - 27 v^2 w) x^4 + 48 \r^2 w^2 x^2 - 64 (\r^3 w^3),$$
then
\es{\label{formula-rank-L}
\rank({\LL}_{w,\r,v})=\#\{\tn{Irr. factors of $R(x)$}\}-\#\{\tn{Irr. factors of $C(x)$}\}-\delta_1+\delta_2,
}
where $\delta_2\in\{0,1\}$ with $\delta_2=1$ iff $-4w^2 \r$ is $3$ times a $4$-th power and
\es{\label{delta1}
\delta_1:=\begin{cases}
2 &\text{if  $\r= v^2$ or if $v=0$ and $-2\r w$ is a square in $\Q$ whereas $-3\r$ and $\r w$ are not,}\\
1&\text{if $v=0$ and  $-3\r, rw$ and $-2\r w$ are not squares in $\Q$},\\
0 &\text{otherwise}.
\end{cases}
}
\end{prop}
\begin{proof}
In this case we have
$A(x)=3x^2+\r$, $C(x)=x^3+3vx^2+3\r x+\r v$ and $B(x)=0$. Using~\eqref{formula-degree2} we obtain
\est{
\rank (\EE_w) &= \lim_{X \rightarrow \infty} \frac{1}{X} \sum_{p \leq X} \log{p}  \bigg(
\sum_{\substack{ x \mod p \\ A(x) \equiv 0 \mod p}}  \leg{wC(x)}{p}  +
\sum_{\substack{ x \mod p \\ C(x) \equiv 0 \mod p}}  \leg{wA(x)}{p} \Bigg)\\
&= \lim_{X \rightarrow \infty} \frac{1}{X} \sum_{p \leq X} \log{p} (S_1(p)+S_2(p)),
}
say. We first consider $S_2(p)$.
The discriminant of $C(x)$ is $-108 \r (\r - v^2)^2$. If $\r=v^2$, then $C(x)=(v + x)^3$ and so $S_2(p)=  \leg{wA(-v)}{p}= \leg{w4v^2}{p}=\leg{w}{p}$ for $p$ large enough. In particular
\est{
\lim_{X\to\infty}\frac{1}{X}\sum_{p\leq X} S_2(p) \log p=\begin{cases}
1 &\text{$w$ is a square in $\Q$,}\\
0 &\text{otherwise.}
\end{cases}
}
Now, assume $\r\neq v^2$. We have
\begin{equation}\label{S1i}
S_2(p)=\sum_{\substack{x \mod p \\  wA(x)\equiv \square \mod p \\ C(x) \equiv 0 \mod
p}} 2-\sum_{\substack{x \mod p \\ C(x) \equiv 0 \mod
p}} 1=S'_{2}(p)-S''_{2}(p),
\end{equation}
say. Then, since  $C(x)$ is square-free, by Chebotarev's density theorem about the density of primes which splits, splits completely or are inert in a number field (\cite[VIII,\textsection 4, Th. 10]{lang} we have
\est{
\lim_{X\to\infty}\frac{1}{X}\sum_{p\leq X} S''_2(p) \log p=\#\{\text{Irreducible factors of $C(x)$ in $\Q[x]$}\}.
}
I%ndeed, suppose that $C(x)$ is irreducible over $\Q$ and that the number field $K$ defined by $C(x)$ is non-Galois, then by Chebotarev's density theorem the set of primes $p$ such that $C(x)$ has 3 roots in $\mod p$ ($p$ totally splits in $K$), 1 roots ($p$ splits but not completely) , 0 roots ($p$ is inert) %have density respectively equal to $1/6$, $1/2$ and $1/3$  
 
Next, we rewrite $S'_2(p)$ as
\est{
S'_2(p)=\sum_{0\not\equiv \ell\mod p}\sum_{\substack{x \mod p \\ wA(x)\equiv \ell^2 \mod p \\ C(x) \equiv 0 \mod
p}} 1
}
and we express the inner sum in terms of the resultant between $C(z)$ and $wA(z)-x^2$. To do this we notice that defining
$$R(x):=-\Res_z(C(z),wA(z)-x^2)=x^6 + (15 \r w - 27 v^2 w) x^4 + 48 \r^2 w^2 x^2 - 64 (\r^3 w^3)$$
we have that $R(\ell)\equiv 0 \mod p$ if and only if either $C(x)$ and $wA(x)-\ell^2$ have a common zero in $\mathbb F_p$ or if $C(x)$ and $wA(x)-\ell^2$ have a common irreducible factor of degree $>1$, i.e. if $(wA(x)-\ell^2) |C(x)$ with $wA(x)-\ell^2$ irreducible in $\mathbb F_p[x]$.
Moreover, for $p$ large enough $C$ doesn't have multiple zeros over $\mathbb F_p$ and we have that the resultant $R(\ell)$ has a zero $\ell'\neq0$ of multiplicity $m$ if and only if there are $m$ solutions $\mod p$ of  $wA(x)\equiv {\ell'}^2 \mod p, \ C(x) \equiv 0 \mod p$. Also, if $(wA(x)-{\ell'}^2)|C(x)$ with $wA(x)-{\ell'}^2$ irreducible  in $\mathbb F_p[x]$, then $\ell'$ is a double root of $R(\ell)$.
 Thus,
\est{
S'_2(p)=\sum_{\substack{ 0\neq \ell\mod p,\\ R(\ell)\equiv0\mod p}}m(\ell)-  \sum_{\substack{ 0\neq \ell\mod p,\\ (wA(x)-\ell^2) |C(x) \text{ in $\mathbb F_p[x]$},\\ wA(x)-\ell^2\text{ irreducible in $\mathbb F_p[x]$}}}2,
}
where $m(\ell)$ is the multiplicity of $\ell$ as a zero of $R(\ell)$. Since $R(0)=-64\r^3\not\equiv0\mod p$ for $p$ large enough, then
\est{
\lim_{X\to\infty}\frac{1}{X}\sum_{p\leq X}\log p\sum_{\substack{ 0\neq \ell\mod p,\\ R(\ell)\equiv0\mod p}}m(\ell)=\#\{\text{Irreducible factors of $R(x)$ in $\Q[x]$}\}.
}
Now we write $C(x)$ as $C(x)=L_1(A(x))x+L_2(A(x))$ with $L_1,L_2\in \Q[x]$ of degree $\leq1$. We have that $wA(x)-\ell^2$ divides $C(x)$ in $\mathbb F_p[x]$ if and only if $L_1(\ell^2/w)x+L_2(\ell^2/w)$ is identically zero in $\mathbb F_p[x]$ and so iff $L_1(\ell^2/w)=L_2(\ell^2/w)=0$. The linear polynomials $L_1$ and $L_2$ can have a common root $(\tn{mod}\ p)$ for infinitely many $p$ only if one is multiple of the other, i.e. if $C(x)=(A(x)-a)(bx+c)$ for some $a,b,c\in\Q$. In particular, if $C(x)$ cannot be written in such form then
\est{
\lim_{X\to\infty}\frac{1}{X}\sum_{p\leq X}\log p\sum_{\substack{ 0\neq \ell\mod p,\\ (wA(x)-\ell^2) |C(x) \text{ in $\mathbb F_p[x]$},\\ wA(x)-\ell^2\text{ irreducible in $\mathbb F_p[x]$}}}1=0.
}
Otherwise, since the discriminant of $wA(x)-\ell^2$ is $-12 w (-\ell^2 + \r w)$, one has
\est{
\lim_{X\to\infty}\frac{1}{X}\sum_{p\leq X}\log p\sum_{\substack{ 0\neq \ell\mod p,\\ \ell^2\equiv a\mod p,\\ -3w (\r w - a)\notin(\mathbb F_p)^2}}1=\lim_{X\to\infty}\frac{1}{X}\sum_{p\leq X}{\log p}\pr{1+\leg{a}{p} }\pr{1-\leg{3w(a-\r w)}{p} }
}
and this is equal to $2$ if $3w(a-\r w)$ is not a square in $\Q$ and $a$ is, equal to 1 if $a, 3w(a-\r w)$ and $3aw(a-\r w)$ are not squares in $\Q^*$,
and it is equal to $0$ otherwise. Also, we observe that looking at the first two subresultants of $wA(x)-a$ and $C(x)$, one has that $C(x)$ can have a factor of the form $A(x)-a$ only if $v=0$ in which case $a=-8\r w$. Thus,
\est{
\lim_{X\to\infty}\frac{1}{X}\sum_{p\leq X}S'_2(p) \log p =\#\{\text{Irreducible factors of $R(x)$ in $\Q[x]$}\}-\eta
}
where
\est{
\eta:=\begin{cases}
2 &\text{if $v=0$ and $-3\r, \r w$ are not a square in $\Q$, and $-2\r w$ is a square in $\Q$,}\\
1&\text{if $v=0$ and  $-3\r, \r w$ and $-2\r w$ are not a square in $\Q$,}\\
0 &\text{otherwise}.
\end{cases}
}
Thus, for $\r\neq v^2$ we have
\est{
\lim_{X\to\infty}\frac{1}{X}\sum_{p\leq X}S_2(p) \log p =\#\{\text{Irr. factors of $R(x)$}\}-\#\{\text{Irr. factors of $C(x)$}\}-\eta.
}
Since, for $\r=v^2$ (in particular $v\neq0$ and so $\eta=0$) we have that $R$ and $C$ have respectively $6$ and $3$ irreducible factors, then in general we have
\es{\label{s2f}
\lim_{X\to\infty}\frac{1}{X}\sum_{p\leq X}S_2(p) \log p =\#\{\text{Irr. factors of $R(x)$}\}-\#\{\text{Irr. factors of $C(x)$}\}-\delta_1.
}
where $\delta_1$ is as in~\eqref{delta1}.

We now consider $S_1(p)$. We could proceed as for $S_2(p)$, but instead we follow a more direct approach. For $p$ large enough, we have that $A(x)$ has two distinct zeros $\pm x_{p}$ mod $p$ if and only if $\leg {-3\r}p=1$. In particular, $S_1(p)=0$ for if $\leg {-3\r}p\neq 1$. Also, we notice that if $\leg {-3\r}p=1$, then $C(\pm x_{p})=\pm \frac{8\r}{3} x_p$ and so $S_1(p)=\leg{6w\r x_p}{p}(1+\leg{-1}{p})=\leg{2wx_p}{p}(1+\leg{-1}{p})$. Thus, proceeding as in the proof of Proposition~\ref{rank_F} we have that
\est{
\lim_{X\to\infty}\frac{1}{X}\sum_{p\leq X}S_1(p) \log p  =\lim_{X\to\infty}\frac{2}{X}\sum_{\substack{p\leq X,\\ \leg{-3\r}{p}=\leg{-1}{p}=1}}\leg{2wx_p}{p}=0
}
unless $-3\r$ is a square in $\Q$. If $\r=-3k^2$ with $k\in\Q$, then $x_p=k$ and we have
\est{
\lim_{X\to\infty}\frac{1}{X}\sum_{p\leq X}S_1(p) \log p  =\lim_{X\to\infty}\frac{1}{X}\sum_{\substack{p\leq X}}\leg{2wk}{p}\pr{ 1+\leg{-1}{p}}=\begin{cases}
1& \text{$2wk$ is $\pm$ a square,}\\
0& \text{otherwise.}
\end{cases}
}
Since $2wk\in\pm\Q^2$ if and only if  $\r$ is $-\frac{3}{4w^2}$ times a $4$-th power, then by the above computation and~\eqref{s2f} we obtain~\eqref{formula-rank-L}.

We will give an example of a family ${\LL}$ with rank $3$ in the paragraph just after the proof of the Proposition and one can easily find families with rank $0,1$ and $2$. Thus, it remains to show that $\rank ({\LL}_{w,\r,v})$ is always $\leq3$. To see this, we observe that the average value of $S_2(p) \log p$ given in~\eqref{s2f} is always $\leq2$. Indeed, by the definition of $S_2(p)$ we have
\es{\label{s2f_2}
\lim_{X\to\infty}\frac{1}{X}\sum_{p\leq X}S_2(p) \log p &\leq \lim_{X\to\infty}\frac{1}{X}\sum_{p\leq X}\sum_{\substack{x \mod p \\ C(x) \equiv 0 \mod
p}} 1\\
&=\#\{\text{square-free irreducible factors of $C(x)$}\}
}
and thus the average value of $S_2(p) \log p$ is $\leq2$ unless $C(x)$ is a product of three coprime linear factors (in which case $\delta_1=0$). Now, if $C(x)$ factors completely then its discriminant $-108 \r (\r - v^2)^2$ must be a square, i.e. $\r=-3k^2$ for some $k\in \N$. Now, if $y$ is a root of $C(x)$ then from $C(y)=0$ one obtains $v =\frac{ y(y^2-9 k^2 )}{3 (k^2 - y^2)}$ (we can take  $y\neq\pm k$ since $C(\pm k)=\mp8k^3$) and that the other roots of $C(x)$ are $\frac{k (y\pm 3 k )}{k \mp y}$. Then,
\est{
R(x)=(wA(y)-x^2)(wA(\tfrac{k (y+ 3 k )}{k - y})-x^2)(wA(\tfrac{k (y- 3 k )}{k + y})-x^2)
}
and these three quadratic polynomials factors if $wA(y)$ and $wA(\tfrac{k (y\pm 3 k )}{k \mp y})$ are squares, i.e. if $-3 w (k^2 - y^2)$, $6 k w (k + y)$ and $6 k w (k - y)$ are squares (and the average of $S_2(p) \log p$ is exactly the number of such integers which are rational squares). These are all squares if and only if $-3 w$, $ (k + y)/(k-y)$ and $6 k w (k - y)$ are squares in $\Q$ and, in particular, for this to happen we need $w<0$ and $|k|>|y|$. This implies that $6 k w (k - y)$ is negative and thus can't be a square. Thus $R(x)$ has at most $5$ irreducible factors and the proof of the proposition is complete.
\end{proof}
We remark that in the Propositions~\ref{rank_F} to \ref{rank_J}, the generic points appear naturally from the proofs. For example, in Proposition~\ref{rank_K} with $w=2$, one takes the root $1$ of $B^2-4AC$ and observe that $A(x) t^2 + B(x) t +C(x)=64t^2$ so that $(-2,8t)$ is a point in ${\HH}_w$. The same phenomena holds also for Proposition~\ref{rank_L} even if it's a bit less immediate. Let's illustrate this phenomena by looking at the two extremal cases: one where $C$ is the product of 3 coprime linear factors over $\Q$ and one where $C$ is irreducible of degree 3.
 \smallskip
~\\
First, suppose that we are in the case where ${C}$ has three roots. In the proof of Proposition~\ref{rank_L} we showed that for this to happen we need $\r=-3k^2$ for some $k\in\N$ and that in this case $\rank({\LL}_{w,\r,v})$ is equal to $\delta_2$ plus the number of squares among $-3 w (k - y) (k + y), 6 k w (k + y)$ and $6 k w (k - y)$. As said above it cannot happen that these three numbers are all squares, but it could be that two of them are and that at the same time $\delta_1=1$ so that $\rank({\LL})=3$. Indeed,  take
\begin{equation} \label{parameters1} \r=-3k^2, \qquad  v =\frac{ y(y^2-9 k^2 )}{3 (k^2 - y^2)}\end{equation} and
\begin{equation} \label{parameters2}
y=6(b^2-a^2),\qquad k=6(b^2+a^2),\qquad w=\frac{\ell^2}{12(b^2+a^2)},
\end{equation}
with $a,b,\ell\in\N$. Then, $-4w^2 \r/3=\ell^2$ is a square (and so $\delta_2=1$) and so are $6wk(k-y)= (6a\ell)^2$ and $6wk(k+y)=(6b\ell)^2$. These three conditions lead to generic points: with this choice for the parameters, the discriminant in $t$ of $P(x)=A(x)t^2+C(x)$ is a degree $5$ polynomials with roots
\est{
x_{1} & = k, \quad
x_{2}  = -k; \qquad
x_3  = y ,\quad
x_4 = - \frac{6 (a^2 + b^2) (2 a^2 + b^2)}{b^2},\quad
x_5 =  \frac{6 (a^2 + b^2) ( a^2 +2 b^2)}{a^2}.
}
Now we have $P(x_{2})=w (4k^2/\ell)^2$, $P(x_4)=w(2a/b^2\ell)^2$ and $P(x_5)=w(2b/a^2\ell)^2$. Thus, we
get the points $(x_2,4k^2/\ell)$ (the $\delta_2$ contribution), $(x_4,2a/b^2\ell)$ and $(x_5,2b/a^2\ell)$ in ${\LL}_{w,s,v}(\Q(t))$ (the $S_2(p)$ contribution) with the parameters $s,w$ and $v$ as in~\eqref{parameters1} and~\eqref{parameters2}.

Suppose now that $C$ is irreducible and that $R$ factorizes as 2 irreducible factors. This means that the contribution coming from $S_2(p)$ gives 1. Since $R(x)$ is the resultant in $y$ of $wA(y)-x^2$ and $C(y)$ with $C$ irreducible of degree $3$, 
we claim that $R(x)$ factors if and only if one (and so any) root $\rho$ of $C$ is such that $wA(\rho)$ is a square in $\Q(\rho)$. Indeed, let $\rho_1, \rho_2,\rho_3$ the 3 roots of $C$ so that we have  $R(x)=-(x^2-wA(\rho_1))(x^2-wA(\rho_2))
(x^2-wA(\rho_3))$. Let $G$ be the Galois group of the Galois closure of $\Q(\rho_1)$ in $\overline{\Q}$, note that $G$ contains a 3-cycle, $\sigma$, permuting the $\rho_i$. 
Let $H(x)$ be an irreducible factor of $R$ defined over $\Q$, $H(x)$ can be written $H(x)= \prod_{\rho \in S_1} (x^2-wA(\rho))  \prod_{\rho \in S_2} (x+\sqrt{wA(\rho)}) \prod_{\rho \in S_3} (x-\sqrt{wA(\rho)})$ 
where $S_1, S_2, S_3$ are disjoint subsets of $\{\rho_1, \rho_2,\rho_3\}$ and where the values $wA(\rho)$ are all distinct (by definition of $S_i$ and the fact that $H$ is squarefree).
 If $S_2$ (or $S_3$) is not empty then  for some $i$, $\pm \sqrt{wA(\rho_i)}$ is a root of $H(x)$ and if $wA(\rho_i)$ is not a square in $\Q(wA(\rho_i))$  then the minimal polynomial of $\pm \sqrt{wA(\rho_i)}$ over $\Q(\rho_i)$ is $x^2-wA(\rho_2))$ 
and divides $H$ which is not possible. So we have $wA(\rho_i)=M(\rho_i)^2$ for some polynomial $M\in \Q[x]$ (with degree $<3$). Since $M \in \Q[x]$ and $\sigma$ permutes transitively the $\rho_i$, we obtain that 
$wA(\rho) = M(\rho)^2$ for all $\rho \in \{\rho_1, \rho_2,\rho_3\}$. So $S_2$ or $S_3$ is not empty if and only if $wA(\rho) = M(\rho)^2$ for all $\rho \in \{\rho_1, \rho_2,\rho_3\}$ (the converse being clear). In this case, we have 
$R(x)=K(x)K(-x)$ with $K(x)=(x-M(\rho_1)(x-M(\rho_2))(x-M(\rho_3)$ which is fixed by the action of $G$ and so $K(x)\in \Q[x]$. Furthermore, the values $M(\rho_i)$ are conjugate and distinct 
(if $M(\rho_1)=M(\rho_2)$ then $M(\rho_3)=M(\rho_1)$ by the $\sigma$-action which would imply that $A(\rho_1)= A(\rho_2) = A(\rho_3)$: impossible with the degree 2 of $A$) so $K(x)$ is irreducible. Hence $R(x)$ factors into 2 irreducible factors of 
degree 3. Now, if $wA(\rho)$ is not a square in $\Q(wA(\rho))$ then $H(x)=\prod_{\rho \in S_1} (x^2-wA(\rho))$. By the same argument, the values $wA(\rho)$ are conjugate and distinct so $H(x)=R(x)$ is irreducible.

Thus, coming back to the main discussion, we get that $wA(\rho)t^2+C(\rho)$ is a square in $\Q(\rho)(t)$ and so
we obtain a generic point $G=(\rho,\ell t/w^2)$ where $\ell^2=wA(\rho) \in {\LL}_{w,\r,v}(\Q(\rho)(t))$. Thus, the trace $\tr_{\Q(\rho)/\Q}(G)$ gives a point
in ${\LL}_{w,\r,v}(\Q(t))$. For example, take $\r=1$,  $v=9$ and $w=1$ so that
$$
{\LL} \colon y^2 =x^3 + (3t^2 + 27)x^2 + 3x + (t^2 + 9)
$$
and the rank is 1. We have $R(x)=x^6 - 2172x^4 + 48x^2 - 64 = (x^3 - 46x^2 - 28x - 8)(x^3 + 46x^2 - 28x + 8)$ and if $\rho$ is a root of $C(x)$ then
$A(\rho)=(\rho^2/12+\rho/2-1/4)^2$ so that $G=(\rho,t(\rho^2/12+\rho/2-1/4)) \in {\LL}_{w,\r,v}(\Q(\rho)(t))$ and
$$
\tr_{\Q(\rho)/\Q}(G) = \left(\frac{15t^2+144}{t^2}, \frac{2(13t^4+216t^2+864)}{t^3}\right) \in {\LL}_{w,\r,v}(\Q(t)).
$$
Note that for all the other families we could find generic points with coordinates in $\Q[t]$ but that in this case we found a point with non-polynomial coordinates. We remark however that there is a point which is polynomial in $t$ over an algebraic extension of $\Q$.

\medskip
It could appear as a little bit disappointing not to be able to find potentially parity-biased families with higher ranks.
However, there are also geometric constraints on the rank due to the condition about the type of bad
reduction. Indeed,  let $E \rightarrow C$ be an elliptic surface defined over $\C$ with non-constant
$j$-invariant and let $R_\C$ be the rank  of the Mordell-Weil group over $\C$. Then, we have the following result due to Shioda
(see \cite{schwartz}).
\begin{theo}[Shioda]\label{formule_Shioda} With the above notation, we have
$R_\C \leq -4+4g + n_1 + 2n_2 - 2p_g$, where
\begin{itemize}
\item $g$ is the genus of $C$;
\item $n_1$ is the number of singular fibers of the N\'eron model of type $\mathbf{I}_b$, $b>0$;
\item $n_2$ is the number of singular fibers of the N\'eron model of other types;
\item $p_g$ is the geometric genus of $E$.
\end{itemize}
Furthermore if $p_g=0$ then $R_\C=-4+4g + n_1 + 2n_2$.
\end{theo}

For the families of Theorem~\ref{classification_1} and~\ref{classification_2}, we have $p_g=0$ and
$g=0$ ($C\simeq\PP^1$). The condition about the
type of bad reduction implies that the number of singular fibers is the degree of the square-free part of $\Delta$ plus $1$ coming from
the place at infinity. Since the families are potentially parity-biased, none of the finite fibers can be of type $\mathbf{I}_b$, $b>0$. Now
if $\deg a_2(t) \leq 1$ then one obtains $R_\C = 2$, whereas if $ \deg a_2(t)=2$ then $R_\C \leq 6$.

\section{The root numbers}\label{rootnumbers}

Let $\GFF$ be an elliptic surface given by \eqref{eq_family}, and let $\rootf(t)$ be the root number of the specialization  $\GFF(t)$. Then,
we can compute $\rootf(t)$ as a product
$$
\rootf(t) = - \prod_{p} w_{p}(t),
$$
where the $w_{p}(t)$ are the local root numbers of the elliptic curve $\GFF(t)$ and are defined in terms of representations on the Weil-Deligne group of $\Q_p$.
We remark that the  $-1$ appearing in the formula corresponds to the root number at $\infty$ which is always $-1$ for any elliptic curve defined over $\R$.
The local root numbers can be computed in terms of the reduction type of $\GFF(t)$ modulo $p$ using tables due to Halberstadt, Connell and Rohrlich (\cite{halberstadt}, \cite{connell}, \cite{rohrlich}). We use them in the version given by Rizzo~\cite{rizzo} where the assumption that $\GFF$ is in minimal Weierstrass form is dropped.\footnote{We remark that there are the following misprints in Rizzo's table: in Table II if $(a,b,c)=(\geq 5,6,9)$ then the special condition is $c_6'+2
\not\equiv 3c_{4,4} (9)$; in Table III, the second line should have $(a,b,c)=(0,0,\geq0)$, also if $(a,b,c)=(2,3,1)$ then the Kodaira type is $I^*_2$.}

After computing the root number of every specialization, we can compute the average root number. Notice that in general, it is false that
\begin{equation}  \label{fallacy}
\Avf = - \prod_{p} \int_{\Z_p} w_{p}(t) \; dt.
\end{equation}

However, it is easy to see that \eqref{fallacy} is true if $\rootf(t)$ is the product of {\it finitely many} functions which are $p$-adic locally constants almost everywhere
(i.e. $w_{p}(t)=1$ except for finitely may primes $p$), and this can be generalized if $\rootf(t)$ is well approximated by a finite product. In \cite{helfgott}, those are called {\it almost finite} products of $p$-adic locally constant functions, and we cite his result.

\begin{prop}[Helfgott, Proposition 7.7] \label{HH-averageZ} Let $S$ be a finite set of places of $\Q$, including $\infty$. For every $v \in S$, let $g_v: \Q_v \rightarrow \C$ be a bounded function that is locally constant almost everywhere. For every $p \not\in S$, let $h_p: \Q_p \rightarrow \C$ be a function that is locally constant almost everywhere and satisfies $|h_p(x)| \leq 1$ for all $x$.
Let $B(x) \in \Z[x]$ be a non-zero polynomial, and assume that $h_p(x)=1$ when $v_p(B(x)) \leq 1$. Let
$$W(n) = \prod_{v \in S} g_v(n) \prod_{v \not\in S} h_p(n).$$
If Conjecture~\ref{square-free-1} holds for $B$, then one has
$$
\Av_\Z W(n) = \frac{c_{-} + c_{+}}{2}  \prod_{p \in S} \int_{\Z_p} g_p(x) \;dx \cdot \prod_{p \not\in S} \int_{\Z_p} h_p(x) \;dx,$$
where $c_{\pm} = \lim_{t \rightarrow \pm \infty} \sgn (g_\infty(t))$.
\end{prop}

We remark that $c_\infty$ differs from the value of Proposition 7.7 of \cite{helfgott} as he considers averages over positive integers, and we are using \eqref{average-Z}.
Rizzo also proves a similar result in his paper \cite[Theorem 19]{rizzo} in the particular case that $B(x)=x$ (and then the result is unconditional).

Then, our strategy will be  rewrite the root number as
$$\varepsilon_\GFF(t) = - \prod_{p} w_{p}^*(t),$$
where the modified local root numbers $w_{p}^*(t)$ are such that the product is finite or almost finite, and then we can compute the average root number with Proposition~\ref{HH-averageZ}.

The average root numbers of Theorem \ref{average-periodic} and Theorem \ref{average-nonperiodic} illustrate those 2 phenomenons. For the family of Theorem \ref{average-periodic}, the root number is a finite product and it is periodic, and the average root number is a rational number. For the family of Theorem \ref{average-nonperiodic}, the root number is not given by an almost finite product in terms of the local root numbers $w_{p}(t)$, but can be written as an almost finite product in terms of the modified local root numbers $w^*_p(t)$, and the average root number is given by a convergent infinite Euler product, computed as in Proposition~\ref{HH-averageZ}. Our result are unconditional, as the degree of the polynomial $B$ is 2.

\subsection{The generalized Washington family and proof of Theorem \ref{average-periodic}} \label{formulas-for-j-c4}

As in the introduction, we fix $a\in \Z_{\neq0}$ and we consider the family of elliptic curves
$$
{\was}_a(t) \colon y^2 = x^3 +tx^2-a(3a+t)x +a^3
$$
with
\es{\label{invariant_w}
c_4(t) &= 16(t^2+3at+9a^2); \\
c_6(t) &= -32(t^2+3at+9a^2)(3a+2t); \\
\Delta(t) & =  16 a^2(t^2+3at+9a^2)^2; \\
j(t)&= \frac{256}{a^2} (t^2 +3at+9a^2).
}
Hence, ${\was}_a$ is a potentially parity-biased family. As explained after Theorem~\ref{formule_Shioda}, the rank of
${\was}_a(t)$ over $\C(t)$ is 2 and, as proved in Corollary~\ref{rank_W}, the rank over $\Q(t)$ is $\leq 1$ and it is equal to 1 if and only if $a$ is a square or minus a square.
\smallskip
In fact, the points $(0,a\sqrt{a})$ and $(a,a\sqrt{-a})$ are two points in ${\was}_a(t)(\C(t))$.  By the action of $\Gal(\overline{\Q}/\Q)$, one
can see that they are independent and the rank is thus $2$ over $\Q(i,\sqrt{a})(t)$. Also, if $a$ (resp. $-a$) is a square then $(0,a\sqrt{a})$ (resp.
$(a,a\sqrt{a})$) is an infinite order point defined over $\Q(t)$. As stated in~\cite{duquesne}, the point $(0,1)$ can always be part of the generators of
$\was_1(t)(\Q)$ for any $t \in \Z$ such that $t^2+3t+9$ is square-free.
\smallskip
~\\
Clearly, the family ${\was}_a(t)$ is a generalization of Washington's family (obtained with $a=1$) and it is closed under quadratic twists: if $w\in \Z_{\neq0}$ then
the quadratic twist of  ${\was}_a(t)$ by $w$ defined by ${\was}_{a,w}(t) \colon wy^2 = x^3 +tx^2 - a (3a+t)x + a^3$ is isomorphic to
${\PFF}_{aw}(wt)$. Furthermore, notice that ${\was}_{a/b}(r/s)$ is isomorphic to ${\was}_{abs^2}(b^2sr)$.

\subsubsection{The local root numbers of ${\was}_a(t)$}
In this section, we give formula for the local root numbers of ${\was_a}(t)$ for $t\in \Z$. In the following, we let
$
f_a(t) := (t^2+3at+9a^2).
$
Also, for convenience of notation, we indicate with $\varepsilon_a(t)$ the root number $\eps_{\Wa}(t)$ and we denote by $w_p(t)$ ($a$ will always be understood) the local root number at $p$ of ${\was_a}(t)$. The formula below can be directly computed from
Rizzo's tables (\cite{rizzo}) and can be deduced from the general formula of the root number of ${\PFF}_{\r}$ given in the appendix~\ref{appen_1}.

\subsubsection*{Case $p\geq 3$}

\begin{prop}\label{p5} For $p\geq 3$ we have
\est{
w_p(t) =\begin{cases}
\leg{-1}{p}^{v_p(a)+v_p(f_a(t))} & \text{if $0 \leq v_p(a) \leq v_p(t)$,}\\
-\leg{t_p}{p} & \text{if $0\leq v_p(t) < v_p(a)$ and $v_p(t)$ is even,}\\
\leg{-1}{p} & \text{if $0\leq v_p(t) < v_p(a)$ and $v_p(t)$ is odd.}
\end{cases}
}
\end{prop}
\begin{proof}

We check only the case $p\geq5$. The case of $p=3$ is analogous but involves a much more lengthy case by case analysis of all possibilities using Table II of~\cite{rizzo}). We remark that it's quite surprising that the final formula for $p=3$ turns out to be the same as for the case $p\geq5$.

Let $p \geq 5$, and we first suppose that $0 \leq v_p(a) < v_p(t)$. Then, $v_p(f_a(t)) = 2 v_p(a)$, and $$v_p\left(c_4, c_6, \Delta \right) = (2 v_p(a), 4v_p(a), 6 v_p(a)).$$
We have to find the smallest triple $(g,h,k)$ of nonnegative integers such that $g \equiv v_p(c_4) \mod 4$, $h \equiv v_p(c_6) \mod 6$ and $k \equiv v_p(\Delta) {\mod {12}}$, and then we can read the value of $w_p(t)$ in the Tables of \cite{rizzo} giving the root number of the minimal model.
For convenience, we use the following notation between triplets of non-negative integers: $$(g,h,k) \sim (g', h', k') \iff (g,h,k) = (g', h', k') - \lambda (4,6,12),$$ for an integer $\lambda$.
Writing $v_p(a) = 2\ell + \tau$, with $\tau \in\{0, 1\}$, we have that
$$v_p \left( c_4, c_6, \Delta \right) \sim (2 \tau, 3 \tau, 6 \tau)
$$
and using Table I of \cite{rizzo}, we get that $w_{p}(t) = \leg{-1}{p}^{v_p(a)}$ when $0 \leq v_p(a) < v_p(t)$, which agrees with the statement of the proposition, as $v_p(f_a(t)$ is even in that case.

Suppose now that $0 \leq v_p(t) < v_p(a)$. Similarly to the first case, we write $v_p(t) = 2\ell + \tau$, $\tau \in\{0, 1\}$, and then $v_p(a) = 2\ell+\tau+ (v_p(a)-v_p(t))$, with $v_p(a)-v_p(t) > 0$. This gives
\begin{eqnarray*}
v_p \left( c_4, c_6, \Delta \right) %%%&=& (2 v_p(t), 3v_p(t), 2 v_p(a) + 4v_p(t))\\
&=& \left(4\ell+2 \tau, 6\ell + 3 \tau, 12\ell + 6 \tau + 2 (v_p(a) - v_p(t)) \right)\\
&\sim& \begin{cases} (0,0, 2(v_p(a) - v_p(t)) & \mbox{if $v_p(t)$ is even,} \\
 (2, 3, 6+ 2(v_p(a) - v_p(t)) & \mbox{if $v_p(t)$ is odd,} \end{cases} \end{eqnarray*}
 and using Table II of \cite{rizzo}, we get
 $$w_p(t) = \begin{cases} - \leg{ - c_6(t)_p}{p} = - \leg{t_p}{p} & \mbox{if $v_p(t)$ is even,}\\
 \leg{ -1}{p} & \mbox{if $v_p(t)$ is odd.} \end{cases}$$

 Finally, suppose that $v_p(t)=v_p(a)$.  Then, $v_p(f_a(t)) = 2v_p(t) + v_p \left( t_p^2 + 3 a_p t_p + 9 a_p^2 \right) = 2 v_p(t) + v_p(f_{a_p}(t_p))$, and
 similarly, $v_p(2t+3a) = v_p(t) + v_p(2t_p + 3 a_p)$.
 We write $v_p(f_{a_p}(t_p))=6k+\tau$ with $0\leq \tau \leq 5$ and $v_p(t)=2\ell+\tau'$ with $0\leq \tau'\leq 1$. We have
\begin{eqnarray} \nonumber
v_p\left( c_4, c_6, \Delta \right)&=& (4\ell+2\tau'+6k+\tau,6\ell+3\tau'+6k+\tau+v_p(2t_p+3a_p),12\ell+6\tau'+12k+2\tau)\\
\nonumber
&\sim& (2\tau'+6k+\tau,3\tau'+6k+\tau+v_p(2t_p+3a_p),6\tau'+12k+2\tau)\\ \label{EQ-abc}
&\sim& (2k+2\tau'+\tau,3\tau'+\tau+v_p(2t_p+3a_p),6\tau'+2\tau).
\end{eqnarray}
We first suppose that  $v_p(t)=v_p(a)$ is even.
We note that if $v_p(2t_p+3a_p)>0$ then $v_p(f_{a_p}(t_p) = 0$, i.e. $\tau=k=0$. Replacing $\tau' = 0$ in \eqref{EQ-abc}, and using Table I of \cite{rizzo}, it is easy to see that
\begin{eqnarray*}
 w_p(t) = \begin{cases} 1 & \tau=0 \\
\leg{-1}{p} & \tau=1,3,5 \\
\leg{-3}{p} & \tau = 2,4 \end{cases} \end{eqnarray*}
To see that this agrees with the statement of the proposition, we remark that if $v_p(f_{a_p}(t_p))>0$, then $t_p^2+3a_pt_p+9a_p^2$ has a root modulo $p$ and its discriminant $-3 a_p^2$ is a square modulo
$p$, hence, $(-3/p)=1$.

We now suppose that $v_p(t)=v_p(a)$ is odd, i.e. $\tau'=1$. Replacing in \eqref{EQ-abc} and using Table I of \cite{rizzo}, we have
$$(v_p(c_4), v_p(c_6), v_p(\Delta) )\sim (2+2k+2\tau,3+v_p(3a_p+2t_p)+\tau,6+2\tau),$$ and it is easy to see that
\begin{eqnarray*}
w_p(t) = \begin{cases} \leg{-1}{p} & \tau=0,2,4 \\
\leg{-3}{p} & \tau = 1,5 \\
1 & \tau = 3. \end{cases} \end{eqnarray*}
Again, using the fact that $\leg{-3}{p}=1$ when $\tau > 0$, this agrees with the statement of the proposition.
%%%$$w_p(t) = (-1/p)^{v_p(a)+v_p(f_{a_p}(t_p))}$$
\end{proof}

\subsubsection*{Case $p=2$}
For $a\in \Z_{\neq0}$ and $t \in \Z$, we set $s_a(t) \in \{\pm 1\}$ such that $w_2(t)\equiv s_a(t) f_a(t)_2 \pmod{4}$.
\begin{prop}\label{sat}  The values $s_a(t)$ are given by the following cases.
\begin{itemize}
\item For $0 \leq v_2(a) \leq v_2(t)$ and $v_2(a)$ even then $s_a(t)=1$ if and only if
$$
\left\{
\begin{array}{ll}
a_2 = \pm 1 \pmod{8} &   \\
\mbox{ or } &\\
a_2 = 3 \pmod{8}& \mbox{ and } 2^{-v_2(a)}t  \equiv 1, 2, 3 \pmod{4} \\
\mbox{ or } &\\
a_2 = 5 \pmod{8}& \mbox{ and } 2^{-v_2(a)}t \equiv 0, 2, 3 \pmod{4}
\end{array} \right. .
$$
\item For $0 \leq v_2(a) \leq v_2(t)$ and $v_2(a)$ odd then $s_a(t)=1$ if and only if
$$
\left\{
\begin{array}{ll}
a_2 = 1 \pmod{4}& \mbox{ and } 2^{-v_2(a)}t \equiv 1,2 \pmod{4} \\
\mbox{ or } &\\
a_2 = 3 \pmod{4}& \mbox{ and } 2^{-v_2(a)}t  \equiv 0,1 \pmod{4}
\end{array} \right. .
$$
\item For $v_2(a)=v_2(t)+1$ and $v_2(t)$ even then $s_a(t)=1$ if and only if
$$
\left\{
\begin{array}{ll}
a_2 = 1 \pmod{4}& \mbox{ and } t_2 \equiv 1,3 \pmod{8} \\
\mbox{ or } &\\
a_2 = 3 \pmod{4}& \mbox{ and } t_2  \equiv 1, 7 \pmod{8}
\end{array} \right. .
$$
\item For $v_2(a)=v_2(t)+1$ and $v_2(t)$ odd then $s_a(t)=1$ if and only if $t_2\equiv a_2 \pmod{4}$.
\item For $v_2(a)=v_2(t)+2$ and $v_2(t)$ even then $s_a(t)=1$ if and only if
$$
\left\{
\begin{array}{ll}
a_2 = 1 \pmod{4}& \mbox{ and } t_2 \equiv 3, 5, 7 \pmod{8} \\
\mbox{ or } &\\
a_2 = 3 \pmod{4}& \mbox{ and } t_2  \equiv 1, 3, 7 \pmod{8}
\end{array} \right. .
$$
\item For $v_2(a)=v_2(t)+2$ and $v_2(t)$ odd then $s_a(t)=1$ if and only if $t_2 \equiv 1 \pmod{4}$.
\item For $v_2(a)\geq v_2(t)+3$ and $v_2(t)$ even then $s_a(t)=1$ if and only if
$$
\left\{
\begin{array}{ll}
v_2(a) =v_2(t)+3 & \mbox{ and } t_2 \equiv 3, 5, 7 \pmod{8} \\
\mbox{ or } &\\
v_2(a) =v_2(t)+4 & \mbox{ and } t_2  \equiv 1 \pmod{4} \\
\mbox{ or } & \\
v_2(a) \geq v_2(t)+5 & \mbox{ and } t_2  \equiv 5 \pmod{8}
\end{array} \right. .
$$
\item For $v_2(a)\geq v_2(t)+3$ and $v_2(t)$ odd then $s_a(t)=1$ if and only if $t_2\equiv 1 \pmod{4}$.
\end{itemize}
\end{prop}
\begin{proof}
As for Proposition~\ref{p5}, one performs a rather lengthy case by case analysis of all possibilities using Table III of~\cite{rizzo}).
\end{proof}

\begin{remark}\label{sat_particular_case}
Note that if $v_2(a) = v_2(t)+4$ then in any case $s_a(t) \equiv t_2 \pmod{4}$ and that if $v_2(a)\geq v_2(t)+3$ and $v_2(t)$ odd, then $s_a(t) \equiv t_2 \pmod{4}$. These facts will be important in the proofs of Theorems~\ref{rarn}, \ref{rarn3} and~\ref{rarn2}.
\end{remark}

\subsubsection*{The root number of ${\was}_a$ and proof of the first part of Theorem~\ref{average-periodic}}

By the previous section we have $\varepsilon_a(t) = -\prod_{p} w_p(t)$ we now show how to transform this product into
\es{\label{root_for_W}
\varepsilon_a(t) \equiv -s_a(t) \gcd(a_2,t) \prod_{p \mid \frac{a_2}{\gcd(a_2,t)}} (-1)^{1+v_p(t)} \left(\frac{t_p}{p}\right)^{1+v_p(t)} \pmod{4}.
}
as given in Theorem~\ref{average-periodic}. We recall that $t_p:=p^{-v_p(t)}t$ and thus $a_2:=2^{-v_2(a)}a$.
\begin{proof} Let $p\geq 3$. From the definition $f_a(t):= (t^2+3at+9a^2)$ one has that if $0\leq v_p(t) < v_p(a)$, then $v_p(f_a(t)) = 2 v_p(t)$. Hence, by Proposition~\ref{p5}, if $0\leq v_p(t) < v_p(a)$ then
$$
w_p(t) = \left(\frac{-1}{p}\right)^{v_p(f_a(t)) + v_p(a)} \left(-\left(\frac{t_p}{p}\right)\right)^{1+v_p(t)} \left(\frac{-1}{p}\right)^{v_p(t) + v_p(a)}.
$$
Thus,
$$
\prod_{p\geq 3} w_p(t) = \prod_{p\geq 3} \left(\frac{-1}{p}\right)^{v_p(f_a(t)) + v_p(a)} \cdot \prod_{\substack{p\geq 3 \\ 0 \leq v_p(t) < v_p(a)}} \left(-\left(\frac{t_p}{p}\right)\right)^{1+v_p(t)} \left(\frac{-1}{p}\right)^{v_p(t) + v_p(a)}
$$
Using the fact that $\leg{-1}{p} \equiv p \pmod{4}$, we have
$$
\prod_{p\geq 3} \left(\frac{-1}{p}\right)^{v_p(f_a(t)) + v_p(a)} \equiv |(af_a(t))_2|\equiv |a_2|f_a(t)_2 \pmod{4}.
$$
since $f_a(t)>0$ for all $t$.
Now, for a prime $p$, we have $0\leq v_p(t) <v_p(a)$ if and only if $p \mid \frac{a}{\gcd(a,t)}$. In this case we also have
$v_p(t)+v_p(a) \equiv v_p\left( {a}/{\gcd(a,t)}\right) \pmod{2}$. Furthermore, the odd prime factors of $\frac{a}{\gcd(a,t)}$ are
the prime factors of $\left({a}/{\gcd(a,t)}\right)_2$ and $\left({a}/{\gcd(a,t)}\right)_2 = \frac{a_2}{\gcd(a_2,t)}$. So,
 \begin{eqnarray*}
\prod_{\substack{p\geq 3 \\ 0 \leq v_p(t) < v_p(a)}} \left(-\left(\frac{t_p}{p}\right)\right)^{1+v_p(t)} \left(\frac{-1}{p}\right)^{v_p(t) + v_p(a)}
&=& \prod_{p \mid \frac{a_2}{\gcd(a_2,t)}} \left(-\left(\frac{t_p}{p}\right)\right)^{1+v_p(t)} \left(\frac{-1}{p}\right)^{v_p(a_2/\gcd(a_2,t))} \\[-0.5em]
&\equiv&  \frac{|a_2|}{\gcd(a_2,t)} \prod_{p \mid \frac{a_2}{\gcd(a_2,t)}} \left(-\left(\frac{t_p}{p}\right)\right)^{1+v_p(t)} \pmod{4}
\end{eqnarray*}
Finally, recalling that by definition $w_2(t)\equiv s_a(t) f_a(t)_2 \pmod{4}$, we have
\begin{eqnarray*}
\varepsilon_a(t) &=& - \prod_{p\geq2} w_p(t) \\[-0.5em]
                          &\equiv& - s_a(t)   \frac{ (f_a(t)_2)^2|a_2|^2}{\gcd(a_2,t)} \prod_{p \mid \frac{a_2}{\gcd(a_2,t)}} \left(-\left(\frac{t_p}{p}\right)\right)^{1+v_p(t)} \pmod{4} \\
                          &\equiv& - s_a(t) \gcd(a_2,t) \prod_{p \mid \frac{a_2}{\gcd(a_2,t)}} (-1)^{1+v_p(t)} \left(\frac{t_p}{p}\right)^{1+v_p(t)} \pmod{4}
                          \end{eqnarray*}
as claimed.
\end{proof}
\begin{corol}[O.~Rizzo]
Let ${\was}_1 \colon y^2=x^3+tx^2-(t+3)x+1$ be the Washington's family. Then the root number of $E_t$ is -1 for every $t\in \Z$.
\end{corol}
\subsubsection{The average root number for ${\was}_a(t)$ and the proof of the second part of Theorem~\ref{average-periodic}}

In this section, we give a closed formula for the average root number of ${\PFF}_a(t)$.
\begin{prop}\label{average_root_W}
The average root number of the family ${\was}_a$ is
$$
\Av_\Z(\eps_{{\was}_a}) = - \prod_{p \mid 2a} \Mw(p),
$$
where for $p$ odd we have
\begin{eqnarray*}
\Mw(p) = \begin{cases} \displaystyle \frac{1-p^{-2\lfloor v_p(a)/2\rfloor}}{p+1} + {p^{-v_p(a)}}  & \mbox{ if } p \equiv 1 \mod 4,\\[1em]
\displaystyle  - \frac{p-1}{p^2+1} \left( 1 - (-p^{-2})^{\lfloor v_p(a)/2\rfloor} \right) + (-1)^{v_p(a)} {p^{-v_p(a)}} & \mbox{ if } p \equiv 3 \mod 4 \end{cases}
\end{eqnarray*}
and for $p=2$,
\begin{eqnarray*}
\Mw(2) = \begin{cases}
1 & \mbox{ if } a \equiv \pm 1 \pmod{8}, \\
1/2 & \mbox{ if } a \equiv  \pm 3 \pmod{8}, \\
0 & \mbox{ if } v_2(a)=1, \\
1/2 & \mbox{ if }  v_2(a)=2 \mbox{ and } a_2 \equiv \pm 1 \pmod{8}, \\
3/8 & \mbox{ if } v_2(a)=2 \mbox{ and } a_2 \equiv \pm 3 \pmod{8}, \\
1/4 & \mbox{ if } v_2(a) = 3, \\
2^{1-v_2(a)} - \frac{2^{v_2(a)-4}-1}{3 \times 2^{v_2(a)-4}}   & \mbox{ if }   v_2(a) \geq 4 \mbox{ and } v_2(a) \mbox{ even and } a_2 \equiv \pm 1 \pmod{8} ,\\[0.3em]
3/2^{v_2(a)+1} - \frac{2^{v_2(a)-4}-1}{3 \times 2^{v_2(a)-4}}  & \mbox{ if } v_2(a) \geq 4 \mbox{ and } v_2(a) \mbox{ even and } a_2 \equiv \pm 3 \pmod{8}, \\[0.3em]
2^{1-v_2(a)} + \frac{1-2^{v_2(a)-3}}{3 \times 2^{v_2(a)-3}} & \mbox{ if } v_2(a) \geq 5 \mbox{ and } v_2(a) \mbox{ odd}.
 \end{cases}
\end{eqnarray*}
\end{prop}
\begin{proof}
Since $\leg{-1}p\equiv 1\mod 4$ for $p$ odd, then we can rewrite~\eqref{root_for_W} as
\est{
\varepsilon_a(t) &= -s_a(t) \prod_{p|a_2}\leg{-1}{p}^{\min(v_p(a_2),v_p(t))} \prod_{p \mid \frac{a_2}{\gcd(a_2,t)}} (-1)^{1+v_p(t)} \left(\frac{t_p}{p}\right)^{1+v_p(t)}=-\prod_{p \mid 2a} w^*_{p}(t),
\\
}
where $w^*_2(t):=s_a(t)$ for $p$ odd
\begin{eqnarray*}
 w^*_{p}(t) &:=& \begin{cases} \leg{-1}{p}^{v_p(t)} (-1)^{1+v_p(t)} \left( \frac{t_p}{p} \right)^{1+v_p(t)} & v_p(t) < v_p(a), \\
\leg{-1}{p}^{v_p(a)}  & v_p(t) \geq v_p(a). \end{cases}
\end{eqnarray*}
Then, the average root number is given by
$$
\mbox{Av}_{\Z} (\varepsilon_a) = -\prod_{p \mid 2a} \int_{\Z_p} w^*_{p}(t) d \mu(t).
$$
For $p \mid a$ odd, we have

\begin{eqnarray*}
\int_{\Z_p} w_{p}(t) d \mu(t) = \sum_{e=0}^{v_p(a)-1} \leg{-1}{p}^e  \frac{(-1)^{e+1}}{p^{e+1}} \sum_{d \in (\Z/p\Z)^*}
 \left( \frac{d}{p} \right)^{e+1}+ \leg{-1}{p}^{v_p(a)} \sum_{e=v_p(a)}^{\infty} \sum_{d \in (\Z/p\Z)^*}
\frac{1}{p^{e+1}}.
\end{eqnarray*}
Let $N_a = \lfloor \frac{v_p(a)-2}{2} \rfloor + 1.$
The first sum is
\begin{eqnarray*}
\sum_{e=0,\, e \; \rm{odd}}^{v_p(a)-1} \leg{-1}{p}^e \frac{p-1}{p} \frac{1}{p^e} = \begin{cases}
\frac{1-p^{-2 N_a}}{p+1} & p \equiv 1 \mod 4, \\
- \frac{p-1}{p^2+1} \left( 1 - (-p^{-2})^{N_a} \right) & p \equiv 3 \mod 4, \end{cases}
\end{eqnarray*}
and the second sum is
\begin{eqnarray*}
\leg{-1}{p}^{v_p(a)} \frac{p-1}{p^{v_p(a)+1}} \sum_{e=0}^{\infty} \frac{1}{p^e} = \leg{-1}{p}^{v_p(a)} {p^{-v_p(a)}}.
\end{eqnarray*}
Thus,
\es{\label{average_root_W-p}
\int_{\Z_p} w_{p}(t) d \mu(t) = \begin{cases} \displaystyle \frac{1-p^{-2 N_a}}{p+1} + {p^{-v_p(a)}}  & p \equiv 1 \mod 4\\[1em]
\displaystyle  - \frac{p-1}{p^2+1} \left( 1 - (-p^{-2})^{N_a} \right) + (-1)^{v_p(a)} {p^{-v_p(a)}} & p \equiv 3 \mod 4 \end{cases}
}

For $p=2$ we consider several cases depending on $v_2(a)$ and $a_2 \pmod{8}$.
\subsubsection*{The case $v_2(a)=0$.}
First, looking at the values of $s_a(t)$, we note that if $a \equiv \pm 1 \pmod{8}$ then $s_a(t)=1$ for all $t$ and $\int_{\Z_2} s_a(t)dt = 1$. Otherwise,
we write
$$
\int_{\Z_2} s_a(t)dt = \int_{v_2(t)=0} s_a(t)dt  +  \int_{v_2(t)=1} s_a(t)dt + \int_{v_2(t) \geq 2} s_a(t)dt
$$
where in any case, if $d$ is odd then $s_a(d2^e)$  depend on $d\pmod{4}$. If $a \equiv 3 \pmod{8}$ then
$$
\int_{v_2(t)=0} s_a(t)dt = \frac{1}{2^{0+2}} \sum_{d\in(\Z/4\Z)^\times} s_a(d) = \frac{1}{2^2} (s_a(1) + s_a(3)) = \frac{1}{2}
$$
and
$$
\int_{v_2(t)=1} s_a(t)dt = \frac{1}{2^{1+2}} \sum_{d\in(\Z/4\Z)^\times} s_a(2d) = \frac{1}{2^3} (s_a(2) + s_a(6)) = \frac{1}{2^2}
$$
and
$$
\int_{v_2(t)\geq 2} s_a(t)dt = \sum_{e\geq 2} \frac{1}{2^{e+2}} \sum_{d\in(\Z/4\Z)^\times} s_a(2^ed) =\sum_{e\geq 2} \left(s_a(2^e) + s_2(2^e3)\right)  = \sum_{e\geq 2} \frac{-1}{2^{e+1}} = -\frac{1}{2^2}.
$$
Summing up the various contributions we obtain $\int_{\Z_2} s_a(t)dt = \frac{1}{2}$. If $a\equiv 5 \pmod{8}$, then the same method leads to the same result:
$$
\int_{\Z_2} s_a(t)dt = \int_{v_2(t)=0} s_a(t)dt  +  \int_{v_2(t)=1} s_a(t)dt + \int_{v_2(t) \geq 2} s_a(t)dt = 0 + \frac{1}{2^2} + \frac{1}{2^2} = \frac{1}{2}.
$$
Hence, summarizing the above computations, we have
\es{\label{v=0}
\int_{\Z_2} s_a(t)dt = \left\{ \begin{array}{ll} 1 & \mbox{ if } a \equiv \pm 1 \pmod{8}, \\ \frac{1}{2} & \mbox{ if } a \equiv \pm 1 \pmod{8}. \end{array} \right.
}
\subsubsection*{The case $v_2(a)=1$.}
We have
$$
\int_{\Z_2} s_a(t)dt = \int_{v_2(t)=0} s_a(t)dt  +  \int_{v_2(t)=1} s_a(t)dt + \int_{v_2(t) \geq 2} s_a(t)dt
$$
where, from the table for $s_a(t)$ we have  $\int_{v_2(t)=0} s_a(t)dt  =  \int_{v_2(t)=1} s_a(t)dt =0$ and
$$
 \int_{v_2(t) \geq 2} s_a(t)dt = \sum_{e\geq 2} \frac{1}{2^{e+2}} \left(s_a(2^e) + s_a(2^e 3)\right) = (-1)^{(a_0-1)/2} \left(\frac{2}{2^{2+2}} + \sum_{e\geq 3} \frac{-2}{2^{e+2}}\right)  = 0.
$$
Thus if $v_2(a)=1$, then
\es{\label{v=1}
\int_{\Z_2} s_a(t)dt = 0 .
}

\subsubsection*{The case $v_2(a)=2$.}
We have
$$
\int_{\Z_2} s_a(t)dt = \int_{v_2(t)=0} s_a(t)dt  +  \int_{v_2(t)=1} s_a(t)dt  +  \int_{v_2(t)=2} s_a(t)dt + \int_{v_2(t)=3} s_a(t)dt + \int_{v_2(t) \geq 4} s_a(t)dt
$$
with $\int_{v_2(t)=0} s_a(t)dt = \frac{1}{2^{0+3}} (s_a(1) + s_a(3) + s_a(5) + s_a(7)) = \frac{1}{2^2}$, $\int_{v_2(t)=1} s_a(t)dt=0$. Furthermore,
we obtain
$$
\int_{v_2(t)=2} s_a(t)dt = \frac{1}{2^{4}}\times  \left\{\begin{array}{ll} 2 & \mbox{ if } a_2 \equiv \pm 1 \pmod{8} \\ 2 & \mbox{ if } a_2 \equiv  3 \pmod{8} \\ 0 & \mbox{ if } a_2 \equiv  5 \pmod{8} \end{array} \right.
$$
and $\int_{v_2(t)=3} s_a(t)dt = \frac{2}{2^{5}}$. Finally,
$$
\int_{v_2(t)\geq 4} s_a(t)dt = \frac{1}{2^{6}}\times  \left\{\begin{array}{ll} 4 & \mbox{ if } a_2 \equiv \pm 1 \pmod{8} \\ -4 &  \mbox{ if } a_2 \equiv  3 \pmod{8} \\ 4 & \mbox{ if } a_2 \equiv 5 \pmod{8} \end{array} \right.
$$
and so
\es{\label{v=2}
\int_{\Z_2} s_a(t)dt = \begin{cases} \frac{1}{2} & \mbox{ if } v_2(a)=2 \mbox{ and }a_2 \equiv \pm 1 \pmod{8}, \\
\frac{3}{8} & \mbox{ if } v_2(a)=2 \mbox{ and }a_2 \equiv \pm 3 \pmod{8}. \end{cases}
}
\subsubsection*{The case $v_2(a)=3$.}
In this case, we have
\est{
&\int_{v_2(t)=0} s_a(t)dt = \frac{1}{2^2}, \qquad \int_{v_2(t)=1} s_a(t)dt =\int_{v_2(t)=2} s_a(t)dt=0,\\
&\int_{v_2(t)=3} s_a(t)dt= - \int_{v_2(t)\geq 4} s_a(t)dt= (-1)^{(a_0-1)/2} \frac{1}{2^5}.
}
Thus, if $v_2(a)=3$ then
\es{\label{v=3}
\int_{\Z_2} s_a(t)dt = \frac{1}{4} .
}

\subsubsection*{The case $v_2(a)\geq4$ with $v_2(a)$ even.}
In this case we have
$$
\int_{\Z_2} s_a(t)dt = \int_{v_2(t)<v_2(a)} s_a(t)dt  +  \int_{v_2(t)=v_2(a)} s_a(t)dt  +  \int_{v_2(t)=v_2(a)+1} s_a(t)dt + \int_{v_2(t)\geq v_2(a)+2} s_a(t)dt.
$$
With the same techniques as before,
\est{
& \int_{v_2(t)=v_2(a)} s_a(t)dt = \frac{1}{2^{v_2(a)+1}} \left\{ \begin{array}{ll} 1 & \mbox{ if } a_2 \equiv \pm 1 \pmod{8}, \\ 1 & \mbox{ if } a_2 \equiv 3 \pmod{8}, \\
0 & \mbox{ if } a_2 \equiv 5 \pmod{8}, \end{array} \right.\\[0.8em]
& \int_{v_2(t)=v_2(a)+1} s_a(t)dt = \frac{1}{2^{v_2(a)+2}}, \\
& \int_{v_2(t)\geq v_2(a)+2} s_a(t)dt = \frac{1}{2^{v_2(a)+2}} \left\{ \begin{array}{ll} 1 & \mbox{ if } a_2 \equiv \pm 1 \pmod{8}, \\ -1 & \mbox{ if } a_2 \equiv 3 \pmod{8} ,\\
1 & \mbox{ if } a_2 \equiv 5 \pmod{8}. \end{array} \right.
}
Now,
$$
\int_{v_2(t)<v_2(a)} s_a(t) dt = \int_{v_2(t)=v_2(a)-2} s_a(t) dt + \int_{v_2(t)=v_2(a)-4} s_a(t) dt + \sum_{j=4 \atop j {\rm ~even}} \int_{v_2(t)=v_2(a)-2-j} s_a(t) dt
$$
The first integral of the right hand side is $\frac{1}{2^{v_2(a)}}$, the second one is $0$ and
\begin{eqnarray*}
 \sum_{j=4 \atop j {\rm ~even}}^{v_2(a)-2} \int_{v_2(t)=v_2(a)-2-j} s_a(t) dt &= & \sum_{j=0 \atop j {\rm ~even}}^{v_2(a)-6} \int_{v_2(t)=v_2(a)-6-j} s_a(t) dt \\
 &=&  \sum_{j=0 \atop j {\rm ~even}}^{v_2(a)-6} \frac{1}{2^{v_2(a)-6-j+3}} (-2) = -\frac{2^{v_2(a)-4}-1}{3 \times 2^{v_2(a)-4}}.
\end{eqnarray*}
Thus, collecting the above results, we have that if $a=2^{v_2(a)}a_2$ with $a_2$ odd and $v_2(a)\geq 4$ even, then
\es{\label{v>3e}
\int_{\Z_2} s_a(t)dt = \frac{1}{2^{v_2(a)}} - \frac{2^{v_2(a)-4}-1}{3 \times 2^{v_2(a)-4}} +
\begin{cases}
 \frac{1}{2^{v_2(a)}} & \mbox{ if } a_2 \equiv \pm 1 \pmod{ 8}, \\
\frac{1}{2^{v_2(a)+1}}&  \mbox{ if } a_2 \equiv \pm 3 \pmod{ 8}.
\end{cases}
}
\subsubsection*{The case $v_2(a)\geq5$ with $v_2(a)$ odd.}
In this case we have
$$
\int_{\Z_2} s_a(t)dt = \int_{v_2(t)<v_2(a)} s_a(t)dt  +  \int_{v_2(t)=v_2(a)} s_a(t)dt  +  \int_{v_2(t)=v_2(a)+1} s_a(t)dt + \int_{v_2(t)\geq v_2(a)+2} s_a(t)dt.
$$
The integral $\int_{v_2(t)=v_2(a)} s_a(t)dt$ is zero and $\int_{v_2(t)=v_2(a)+1} s_a(t)dt =- \int_{v_2(t)\geq v_2(a)+2} s_a(t)dt$. Then as above,
$$
 \int_{v_2(t)<v_2(a)} s_a(t)dt = \sum_{j=0 \atop j {\rm ~even}}^{v_2(a)-1}  \int_{v_2(t)=j} s_a(t)dt = \sum_{j=0 \atop j {\rm ~even}}^{v_2(a)-1}  \int_{v_2(t)=v_2(a)-1-j} s_a(t)dt.
$$
Now, $\int_{v_2(t)=v_2(a)-1} s_a(t)dt = 0$, $\int_{v_2(t)=v_2(a)-3} s_a(t)dt = \frac{1}{2^{v_2(a)-1}}$ and
\est{
 \sum_{j=4 \atop j {\rm ~even}}^{v_2(a)-1}  \int_{v_2(t)=v_2(a)-1-j} s_a(t)dt &=  \sum_{j=0 \atop j {\rm ~even}}^{v_2(a)-5}  \int_{v_2(t)=v_2(a)-5-j} s_a(t)dt =  \sum_{j=0 \atop j {\rm ~even}}^{v_2(a)-5} \frac{-1}{2^{v_2(a)-3-j}} \\
&= -\frac{2^{v_2(a)-3}-1}{3 \times 2^{v_2(a)-3}}.
}
Thus, if $v_2(a)\geq 5$ with $v_2(a)$ odd, then
\es{\label{v>4o}
\int_{\Z_2} s_a(t)dt = \frac{1}{2^{v_2(a)-1}}+\frac{1-2^{v_2(a)-3}}{3 \times 2^{v_2(a)-3}}.
}
\medskip

Thus, by~\eqref{average_root_W-p} and~\eqref{v=0}-\eqref{v>4o} the proof of Proposition~\ref{average_root_W} is complete.
\end{proof}

\subsubsection{Families with elevated rank}\label{ex_W}

We can use the family ${\was}_a(t)$ in order to find families with elevated rank over $\Z$. First, we notice the following corollary
\begin{corol}\label{ex_1} Let $a,b \in \Z$ with $a \equiv \pm 1 \pmod{8}$ and $\gcd(a,b) =1$. Then for all $t\in \Z$, we have
$$
\varepsilon_a(at+b) = - \prod_{p\mid a} - \leg{b}{p}=(-1)^{1+\lambda(\kappa(a))} \leg{b}{\kappa(a)}.
$$
where $\kappa(a):=\prod_{p|a} p$ is the kernel of $a$ and $\lambda$ is the Louville function.
\end{corol}
\begin{proof}
This a direct application of Theorem~\ref{average-periodic}.
\end{proof}
In particular, Corollary~\ref{ex_1} gives that  the root number of $\was_{p^2}(pt+a)$ is 1 for all $t\in \Z$ when $a$ is a quadratic residue mod $p$ and the
root number of $\was_p(pt+b)$ is $-1$ for all $t$ when $b$ is a quadratic non-residue mod $p$. This proves Corollary~\ref{exess_rank_W_1}.
\smallskip
~\\

We can also give examples of families of rank 2 and 3 with elevated rank by considering families of the form $\Wa(p(t))$ where $p(t)$ is a degree 2 polynomial. Indeed,
for $a\in\N$ consider the family %%$\NF_p$ by
\est{
\NFa(t):=
\was_{a^2}(2t^2-2at-a^2).
}
We then have that
$$\NFa(t)=\was_{a^2}(2t^2-2at-a^2)\simeq \PFF_{- 12(3a)^4}(6(2t-a)^2)\simeq{\LL}_{6,-3^3a^4,0}(2t-a).$$
Now, writing $\r =-3^3a^4$, $w=6$ and $v=0$ one has that $-4 w^2r=3(12a)^4$, $-3r=(3a)^4$ and the polynomials $C(x)$ and $R(x)$ of Proposition~\ref{rank_L} factor into $3$ and $5$ irreducible polynomials respectively. Thus, by Proposition~\ref{rank_L}, we have that $\NF_a$ has rank $3$ over $\Q(t)$ for all $a$. %%% and we obtain the following Corollary.

\begin{corol}
Let $p$ be a prime number
 $\equiv \pm 1 \pmod{8}$, and let
 %%% then the family $\mathcal{W}_p^*(t)$ is a rank $3$ potentially parity-biased family.
 $\mathcal{W}_{p,\ell}^{\dagger}(t):=\mathcal{W}_p^{\dagger}(pt+\ell)$ for $\ell\in\Z$. Then for $(p,\ell)=1$, $\mathcal{W}_{p,\ell}^{\dagger}(t)$ is a rank $3$ family with elevated rank over $\Z$.
\end{corol}
\begin{proof}
For any odd prime $p$,
 an easy application of Theorem~\ref{average-periodic} gives that the root number of
$\mathcal{W}_p^\dagger(t)$ is
$$
\varepsilon_{\mathcal{W}_p^\dagger}(t) = \left\{ \begin{array}{cc} \left(\frac{2}{p}\right) & \mbox{ if } p \nmid t \\ -1 &\mbox{ if } p \mid t\end{array} \right.
$$
for any $t\in\Z$. Replacing $t$ by $pt+\ell$, and using the fact that $p \equiv \pm 1 \mod 8$, we get the result.
\kommentar{Moreover, we have
$$\NF_a(t)=\was_{a^2}(2t^2-2at-a^2)\simeq \PFF_{- 12(3a)^4}(6(2t-a)^2)\simeq{\LL}_{6,-3^3a^4,0}(2t-a)$$
Now, writing $\r =-3^3a^4$, $w=6$ and $v=0$ one has that $-4 w^2r=3(12a)^4$, $-3r=(3a)^4$ and the polynomials $C(x)$ and $R(x)$ of Proposition~\ref{rank_L} factor into $3$ and $5$ irreducible polynomials respectively. Thus, by Proposition~\ref{rank_L} we have that $\NF_a$ has rank $3$ for all $a$ and we obtain the following Corollary.}
\end{proof}

\kommentar{\begin{corol}
Let $p$ be a prime number $\equiv \pm 1 \pmod{8}$, then the family $\NF_p(t)$ is a rank $3$ potentially parity-biased family. Moreover, if we let $\NF_{p,\ell}^*(t):=\NF_p(pt+\ell)$ for $\ell\in\Z$, then for $(p,\ell)=1$ one has that $\NF_{p,\ell}^*(t)$ is a rank $3$ family with elevated rank.
\end{corol}}

It's not difficult to construct in a similar way rank $2$ families with elevated rank. For example, one such family is given by ${\was}_4(-3t^2 - 4t - 21)$.

\subsubsection{Twists of Washington's family}\label{twist_W}
In this section, we consider quadratic twists of the original Washington's family (see \cite{kawachi-nakano}, \cite{byeon} for some studies about
Washington's twists).
Let $d \in \Z_{\neq0}$, the twist by $w$ of Washington's family is given by
$$
\was^{(d)}_1(t) \colon y^2 = x^3 +dtx^2 -(t+3)d^2 x + d^3.
$$
We easily see that the family $\was^{(d)}_1(t)$ is in fact the family ${\was}_{d}(dt)$. So in the formula of Theorem~\ref{average-periodic}, the product is empty
and equal to one, hence $\varepsilon_{\was^{(d)}_1}(t)\equiv -|d_2| s_d(dt)\mod 4$. The value of $s_d(dt)$ is given by the first two cases of Proposition~\ref{sat}, furthermore,
we have $2^{-v_2(d)}dt = d_2 t$.

\begin{prop}\label{twist} The root number, $\varepsilon_{\was^{(d)}_1}(t)$ of $\was^{(d)}_1(t)$ is given by the following cases.
~\\
If $v_2(d)$ is even, then
\begin{itemize}
\item if $d_2 \equiv \pm 1 \pmod{8}$ then $\varepsilon_{\was^{(d)}_1}(t)\equiv - |d_2| \pmod{4}$;
\item if $d_2 \equiv 3 \pmod{8}$ then $\varepsilon_{\was^{(d)}_1}(t)=\sgn(d_2)$ if and only if $t\equiv 1,2,3 \pmod{4}$;
\item if $d_2 \equiv 5 \pmod{8}$ then $\varepsilon_{\was^{(d)}_1}(t)=\sgn(d_2)$ if and only if $t\equiv 1 \pmod{4}$.
\end{itemize}
If  $v_2(d)$ is odd then $\varepsilon_{\was^{(d)}_1}(t)=\sgn(d_2)$ if and only if $t\equiv 0, 3 \pmod{4}$.
\end{prop}

One can also consider the twist by
$$d_u(t)=u^3+tu^2-(t+3)u+1 = u(u-1)t + u^3-3u+1$$
for any $u\in\Z$ (one could also take $u$ to be a polynomial in $t$).
In this case, the generic point $(ud_u(t),d_u(t)^2)$ is a non-torsion point of $\was^{(d_u(t))}_1$. So the rank of $\was^{(d_u(t))}_1$ over $\Q(t)$ is at least $1$.
Moreover from Proposition~\ref{twist} we can deduce the following result.

\begin{corol} Let $u\in \Z$.
\begin{itemize}
\item If $u \equiv 1 \pmod{4}$ then $\varepsilon_{\was^{(d_u(t))}_1}(t)= 1$ if and only if $d_u(t)>0$.
\item If $u \equiv 0 \pmod{4}$ then $\varepsilon_{\was^{(d_u(t))}_1}(t)= 1$ if and only if $d_u(t)<0$.
\end{itemize}
\end{corol}
\begin{proof} Assume that $u \equiv 1 \pmod{4}$ then $d_t(u) \equiv -1 \pmod{8}$ for all $t$ and we apply Proposition~\ref{twist}:
$\varepsilon(E_{d_t(u)}(t)) \equiv - |d_t(u)| \equiv \sgn(d_t(u))$. If $u \equiv 0 \pmod{4}$ then $d_t(u) \equiv 1 \pmod{8}$ for all $t$ and
we apply the same method.
\end{proof}

As an example, let's consider the case $u=5$, so that $d_t(5)=20t+111$.  In this case $\varepsilon_{\was^{(20t+111)}_1}(t) >0$ if and only if $t\geq -5$ and there are at least 2 independent points of
$\was^{(20t+111)}_1$:
$$
(5(20t+111),(20t+111)^2) \quad \mbox{ and }  \quad \left(-\frac{20t+111}{4},\frac{(20t+111)^2}{8}\right).
$$

%%%%%%%%%%%%%%%%%%%
%
%
%%%%%%%%%%%%%%%%%%%%

\subsection{The family $\Hold$ and the proof of Theorem \ref{average-nonperiodic}}

First we give the root number for the family
$$
\Hold \colon y^2=x^3 + 3tx^2+ 3\hr tx + \hr^2 t
$$
for which we have
\begin{eqnarray*}
c_4(t) & = & 2^4 3^2 t (t-\hr), \\
c_6(t) & = &-2^5 3^3 t (t-\hr)(2t-\hr), \\
\Delta(t) & = & -2^4 3^3 \hr^2 t^2 (t-\hr)^2, \\
j(t) & = & -\frac{2^8 3^3}{\hr^2} t(t-\hr) .
\end{eqnarray*}
In the following we shall always assume $t\neq0,\hr$, so that $\Delta(t) \neq0$. Also, for convenience of notation, we will use in his section $\varepsilon_{\hr}(t)$  for the root number of $\Hold(t)$, and $w_p(t)$ for the local root number at $p$ of $\Hold(t)$. Then,
$$
\varepsilon_\hr(t) = - \prod_{p} w_{p}(t).
$$

\kommentar{In order to compute
$$\Av_\Z \left( \Hold \right) :=
\Av_\Z \left( \varepsilon_\hr(t) \right)  = - \Av_\Z \left( \prod_{p} w_{p}(t) \right),
$$
we first compute the local root numbers $w_{p}(t)$ at $p \geq 5$, using the Tables \ref{***}. The formulas for $w_{2}(t)$ and $w_{3}(t)$ which are computed similarly but are much more technical and involve a lot of cases are given in Section \ref{***}.}
\subsubsection{The local root numbers of $\Hold$}
The local root numbers for $\Hold$ can be obtained by performing a simple but quite lengthy case by case analysis from
Rizzo's table~\cite{rizzo} as in the proof of Proposition~\ref{p5}. We give the final results only, here for $p\geq5$ and in Appendix~\ref{appen_2} for $p=2,3$.

\begin{lemma} \label{H-p5}
Let $p \geq 5$. Then, for $0 \leq v_p(\hr) \leq v_p(t)$ one has
$$
w_{p}(t)= \begin{cases}\leg{-3}{p} \leg{3}{p}^{v_p(t)+v_p(t-\hr)+v_p(\hr)} & \mbox{ if } 6 \nmid v_p(t-\hr)-v_p(t)+3v_p(\hr), \\[0.6em]
\ 1 & \mbox{ if } 6 \mid  v_p(t-\hr)-v_p(t)+3v_p(\hr),
\end{cases}$$
whereas if $0 \leq v_p(t)  < v_p(\hr)$ then
$$
w_{p}(t) = \left\{ \begin{array}{cc} -\leg{3t_p}{p} & \mbox{ if $v_p(t)$ is even}, \\[0.8em]
\leg{-1}{p} & \mbox{ if $v_p(t)$ is odd.}
\end{array} \right.	
$$ 
\end{lemma}

We now modify the local root numbers $w_{p}(t)$ in order to apply Proposition~\ref{HH-averageZ}. We will write $\varepsilon_{\hr}(t)$ as
$$
\varepsilon_{\hr}(t)=-\prod_{p} w_{p}(t) =- \prod_{p} w_{p}^*(t),
$$
for some $w_p^*(t)$ satisfying $w_{p}^*(t) = 1$ for $v_p(t(t-s)) \leq 1$ for all primes $p \nmid 6\hr$.

Let $p \geq 5$, and suppose that $v_p(\hr)=0$ and $p \mid \Delta(t) = - 2^4 3^3 \hr^2 t^2 (t-\hr)^2$ (if not, $w_{p}(t)=1$).
Then, we have 2 cases: either $v_p(t)=0$ and $v_p(t-\hr)>0$, or $v_p(t)>0$ and $v_p(t-\hr)=0$. Thus, Lemma \ref{H-p5} gives in this case
\begin{eqnarray*} w_{p}(t) = \begin{cases} \leg{-1}{p} \leg{3}{p}^{v_p(t-\hr)+1} & v_p(t)=0, v_p(t-\hr)>0 \text{ and }6 \nmid v_p(t-\hr), \\
\leg{-1}{p} \leg{3}{p}^{v_p(t)+1} & v_p(t)>0, v_p(t-\hr)=0 \text{ and } 6 \nmid v_p(t),\\
\ 1 &\text{if } v_p(t)=0, v_p(t-\hr)>0 \text{ and }6 \mid v_p(t-\hr)\text{ or if } v_p(t)>0, 6 \mid v_p(t).
\end{cases} \end{eqnarray*}
Then,  for all $p \neq 2,3$, we define
\begin{eqnarray} \label{def-star}
w_{p}^*(t) := w_{p}(t) \left( \frac{-1}{p} \right)^{v_p(t-\hr)} \left( \frac{-1}{p} \right)^{v_p(t)}, \end{eqnarray}
so that for $p \nmid 6\hr$ we have $w_{p}^*(t) = 1$ for $v_p(t(t-\hr)) \leq 1$.

\begin{lemma} \label{lemma-modified-RN}
For $p \geq 5$, let $w_{p}^*(t)$ be defined by \eqref{def-star}. Let $w^*_{2}(t),w^*_{3}(t) ,w^*_{\infty}(t)\in\{\pm1\} $ be defined by
\begin{eqnarray*}
w^*_{3}(t) &=& (-1)^{v_3(t)} (-1)^{v_3(t-\hr)} w_{3}(t) \\
w^*_{2}(t) &\equiv&  t_2 (t-\hr)_2  w_{2}(p) \mod 4 \\
w^*_{\infty}(t) &=&   \sgn(t(t-\hr))
\end{eqnarray*}
 Then,
$$\varepsilon_{\hr}(t)  =   - w^*_{\infty}(t)\prod_p w_{p}^*(t).$$ \end{lemma}

\begin{proof}
Using \eqref{def-star}, we have
\begin{eqnarray*}
\prod_{p \neq 2,3} w_{p}^*(t) &=& %%% w_{2}(t) w_{3}(t) %%%\left( \frac{-2}{p} \right)^{v_p(t)} \left( \frac{-1}{p} \right)^{v_p(t-c)}
\prod_{p \neq 2,3} \left( \frac{-1}{p} \right)^{v_p(t)} \left( \frac{-1}{p} \right)^{v_p(t-\hr)}
\prod_{p \neq 2,3} w_{p}(t).
%%\\ &=&  \left( \frac{-1}{t} \right) \left( \frac{-1}{t-r} \right)   \prod_{p \mid 6r} w_{p}(t)
%%\prod_{p \nmid 6r} w_{p}(t),
\end{eqnarray*}
Since $\leg{-1}p\equiv p\mod 4$, then
\begin{eqnarray*}
\prod_{p \neq 2,3} \leg{-1}{p}^{v_p(t)} \leg{-1}{p}^{v_p(t-\hr)} &\equiv& (-1)^{v_3(t)} (-1)^{v_3(t-\hr)} \prod_{p \neq 2} p^{v_p(t)}
p^{v_p(t-\hr)} \mod 4 \\
&\equiv& (-1)^{v_3(t)} (-1)^{v_3(t-\hr)}  |t_2 \, (t-\hr)_2| \mod 4,\\
%&\equiv& (-1)^{v_3(t)} (-1)^{v_3(t-\hr)}  \sgn(t(t-\hr)) \, t_2 \; (t-\hr)_2 \mod 4,
\end{eqnarray*}
which proves the result.
\end{proof}

\subsubsection{The average root number for $\Hold(t)$}
Using Lemma \ref{lemma-modified-RN}
and  Proposition~\ref{HH-averageZ}, we then have
\begin{equation} \label{convergent-product}
\Av_\Z (\varepsilon_{\Hold}) = - \prod_{p} \int_{\Z_p} w_{p}^*(t) \; dt,
\end{equation}
since  $t(t-\hr)$ is positive except for finitely many values of $t$.
Computing the $p$-adic integrals we will obtain the following proposition, thus completing the proof of Theorem \ref{average-nonperiodic}.
\begin{prop}\label{average_root_H}
The average root number of the family $\Hold$ is given by the Euler product
$$
\Av_\Z(\eps_{\Hold}) = - \prod_{p} \Mh(p),
$$
where the Euler factors for $p=2$ and $p=3$ are given by
\est{
\Mh(2) &= \begin{cases} -1/2 &\text{if }  v_2(\hr) = 0 \\
0 &\text{if }  v_2(\hr) =  1 \\
1/8 &\text{if }  v_2(\hr)= 2 \\
% 1/4 &\text{if }  v_2(\hr)= 3 \\
%1/32 &\text{if }  v_2(\hr)=4 \\
{2^{1-v_2(\hr)}}+ \frac{1}{3} \left( 4^{-(v_2(\hr)-3)/2} - 1 \right) &\text{if }  v_2(\hr) \geq 3 \text{ with $v_2(\hr)$ odd} \\
 {2^{-v_2(\hr)-1}}+ \frac{1}{3} \left( 4^{-(v_2(\hr)-4)/2} - 1 \right) &\text{if }  v_2(\hr) \geq 4 \text{ with $v_2(\hr)$ even,}
 \end{cases}\\
 \Mh(3) &=  \begin{cases} \displaystyle
 %\frac{6}{7} \frac{1}{3^{v_3(\hr)+2}} &\text{if } v_3(\hr) \in\{ 0,1\} \\[0.7em]
\displaystyle \frac{6}{7} \frac{1}{3^{v_3(\hr)+2}} +  \frac{3}{4} \left( 3^{-v_3(\hr)} - 1 \right) &\text{if }  %v_3(\hr) \geq 2,
v_3(\hr) \equiv 0 \mod 2 \\[0.7em]
\displaystyle \frac{6}{7} \frac{1}{3^{v_3(\hr)+2}} +  \frac{3}{4} \left( 3^{-v_3(\hr)+1} - 1 \right) &\text{if } % v_3(\hr) \geq 2,
v_3(\hr) \equiv 1 \mod 2,
\end{cases}
}
whereas for $p\geq5$, the Euler factors are
$$%%\int_{\Z_p} w^*_{p}(t) \; dt
\Mh(p) = \leg{-1}{p}  \frac{ 1 - p^{j-v_p(\hr)}}{p+1} + \leg{-1}{p}^j \frac{1}{p^{v_p(\hr)}} \begin{cases} 1 & p \equiv 1 \mod 3, \\
 \left( 1 - \frac{4(p-1)(p^{1-j}+p^{3+j})}{p^6-1} \right) & p \equiv 2 \mod 3, \end{cases}$$
where $v_p(\hr) \equiv j \mod 2$ with $j\in\{0,1\}$. In particular, for $v_p(\hr)=0$ and $p\geq5$, we have
$$\Mh(p) = \begin{cases} 1 & p \equiv 1 \mod 3, \\
  \left( 1 - \frac{4(p-1)(p+p^3)}{p^6-1} \right) & p \equiv 2 \mod 3. \end{cases}$$

\end{prop}

When computing the $p$-adic integrals we shall need the following Lemma.

\begin{lemma} \label{measure-vpt-c} \label{measureSK}
For $k\in\Z_{\geq0}$, let $S_k := \left\{ t \in \Z_p \,:\, v_p(t)=v_p(\hr), v_p(t-\hr) = v_p(\hr)+k \right\}$, then $S_k$ has measure
\begin{eqnarray*}
\mu \left( S_k \right) = \begin{cases} \displaystyle \frac{p-2}{p^{v_p(\hr)+1}} &\text{if } k=0 \\
\displaystyle  \frac{p-1}{p^{v_p(\hr)+k+1}} & \text{if } k \geq 1 \end{cases} \end{eqnarray*}
\end{lemma}
\begin{proof} Let $\chi_{k}$ be the characteristic function of $S_k$.
If $k = 0$, then
$$\mu \left( S_k\right) = \frac{1}{p^{v_p(\hr)+1}} \sum_{d \in (\Z/p \Z)^*}
\chi_{0}(p^{v_p(\hr)} d),$$
and $\chi_{0}(p^{v_p(\hr)} d) = 1$ iff $p^{v_p(\hr)} d - p^{v_p(\hr)} \hr_p \not\equiv 0 \mod {p^{v_p(\hr)+1}}$ and thus iff $d \not\equiv \hr_p \mod p$, which gives the result. In the same way, considering the contribution of $d \equiv \hr_p \mod p$ only, we have
 $$\mu \left( t \in \Z_p \,:\, v_p(t)=v_p(\hr), v_p(t-\hr) \geq v_p(\hr)+1 \right)  = \frac{1}{p^{v_p(\hr)+1}}$$
and similarly, for any $k \geq 1$,
 $$\mu \left( t \in \Z_p \,:\, v_p(t)=v_p(\hr), v_p(t-\hr) \geq v_p(\hr)+k \right)  = \frac{1}{p^{v_p(\hr)+k}}.$$
Then,
$$\mu \left( S_k \right) = \frac{1}{p^{v_p(\hr)+k}} - \frac{1}{p^{v_p(\hr)+k+1}} = \frac{p-1}{p^{v_p(\hr)+k+1}}$$
completing the proof of the Lemma.
\end{proof}

First, we shall compute $\int_{\Z_p} w_{p}^*(t) \; dt$ for $p\geq5$.

\begin{prop} \label{valuesWp}
Let $p \geq 5$. Then $\int_{\Z_p} w^*_{p}(t) \; dt = \Mh(p) $.
 \end{prop}

\begin{proof}
We shall consider three cases, according to whether $v_p(t)$ is smaller, equal, or larger than $v_p(\hr)$.
\subsubsection*{The case $0 \leq v_p(t) < v_p(\hr)$.}
We have that $v_p(\hr) > 0$,  and $v_p(t-\hr) = v_p(t)$ and so by~\eqref{def-star} in this case we have $w^*_{p}(t)=w_{p}(t)$. Using Lemma \ref{H-p5}, we have
\begin{eqnarray*}
\int_{0 \leq v_p(t) < v_p(\hr)} w^*_{p}(t) \; dt &=& - \int_{{0 \leq v_p(t) < v_p(\hr)} \atop {2 \mid v_p(t)}}
 \leg{3t_p}{p}  \; dt +
\int_{{0 \leq v_p(t) < v_p(\hr)} \atop {2 \nmid v_p(t)}}   \leg{-1}{p} \; dt .
\end{eqnarray*}
It is easy to see that the first integral is 0 and that so is the second if $v_p(\hr)=1$ (the domain of integration is empty).
Thus, suppose $v_p(\hr) \geq 2$. Then, letting $v_p(\hr) \equiv j \mod 2$ with $j\in\{0,1\}$
\begin{eqnarray} \label{case-i}  \int_{0 \leq v_p(t) < v_p(\hr)} w^*_{p}(t) \; dt =
\leg{-1}{p}  \sum_{0 \leq 2k+1 < v_p(\hr)} \frac{p-1}{p^{2k+2}} =\leg{-1}{p} \frac{1}{p+1} \left( 1 - p^{j-v_p(\hr)} \right).
\end{eqnarray}
Notice that the expression on the right is $0$ if $v_p(\hr)=0$.

\subsubsection*{The case ${v_p(t) = v_p(\hr)}$.}

We let $v_p(t-\hr) = v_p(\hr) + k$, with $k \geq 0$. We have
\es{\label{eq_for_w_p}
 w_{p}^*(t) = \begin{cases}
\leg{-3}{p}^{k+1} \leg{3}{p}^{v_p(\hr)} &  \text{if }3v_p(\hr) + k \not\equiv 0 \mod 6, \\
\leg{-1}{p}^{k} &  \text{if } 3 v_p(\hr) + k \equiv 0 \mod 6,
\end{cases}
}
and
\est{
\int_{v_p(t)=v_p(\hr)} w^*_{p}(t) \; dt &= \sum_{k = 0}^\infty \int_{\substack{v_p(t)=v_p(\hr)\\v_p(t-\hr)=v_p(\hr)+k}}
w^*_{p}(t) \;dt \\ \label{measure-trick}
&=  \sum_{k = 0}^\infty \mu \left( t \in \Z_p \;:\; v_p(t)=v_p(\hr), v_p(t-\hr) = v_p(\hr)+k \right) \; w^*_{p}(t),
}
with $w^*_{p}(t)$ as in~\eqref{eq_for_w_p}. Thus, by Lemma~\ref{measureSK} we have
\est{
\int_{v_p(t)= v_p(\hr)} w^*_{\hr}(t) \; dt & = \frac{p-2}{p^{v_p(\hr)+1}} \leg{-1}{p}^{j}+\sum_{{k=1}\atop{k \not\equiv 3j \mod 6}}^{\infty} \frac{p-1}{p^{v_p(\hr)+1+k}}  \leg{-3}{p}^{k+1} \leg{3}{p}^{j} \\
&\quad+
\sum_{{k=1}\atop{k \equiv 3j \mod 6}}^{\infty} \frac{p-1}{p^{v_p(\hr)+1+k}}  \leg{-1}{p}^j.
}
The sum over $k$ in the first line gives
\est{
&\frac{p-1}{p^{v_p(\hr)+1}}\leg{-3}{p} \leg{3}{p}^{j}   \sum_{{k=1}}^{\infty} p^{-k}\leg{-3}{p}^{k} -
\frac{p-1}{p^{v_p(\hr)+1}} \leg{-1}{p}^{j} \leg{-3}{p} \sum_{{k=1}\atop{k \equiv 3j \mod 6}}^{\infty}p^{-k} \\
&\hspace{15em}=\frac{p-1}{p^{v_p(\hr)+1}} \leg{3}{p}^{j}   \frac1{p -\leg{-3}{p}}-
\frac{p-1}{p^{v_p(\hr)+1}} \leg{-1}{p}^{j} \leg{-3}{p}\frac{ p^{3j}}{p^6-1}
}
whereas the sum over $k$ in the second line adds up to
\est{
\frac{(p-1)p^{3j}}{p^{v_p(\hr)+1}(p^6-1)}  \leg{-1}{p}^j
}
Then, since $\leg{-3}p=1$ if $p\equiv 1\mod 3$ and $\leg{-3}p=-1$ if $p\equiv 2\mod 3$, we have
\est{
\int_{v_p(t)=v_p(\hr)} w^*_{p}(t) \; dt=\leg{-1}{p}^{j}\begin{cases} \frac{p-1}{p^{v_p(\hr)+1}} & p \equiv 1 \mod 3, \\
\frac{1}{p^{v_p(\hr)+1}} \left(p-2+(-1)^j \frac{p-1}{p+1} + \frac{2 p^{3j}(p-1)}{p^6-1} \right) & p \equiv 2 \mod 3 . \end{cases}
}
\subsubsection*{The case $0 \leq v_p(\hr) < v_p(t)$.}

\kommentar{We now compute $$\int_{0 \leq v_p(\hr) < v_p(t)} w_{p}^*(p) \, dt.$$}
 In this case, $v_p(t-\hr) = v_p(\hr)$, and by Lemma~\ref{H-p5} and~\eqref{def-star} we get
$$
\leg{-1}{p}^{v_p(\hr)} w_{p}^*(t) = \begin{cases} \leg{-3}{p}^{v_p(t)+1} & v_p(t)-4v_p(\hr) \not\equiv 0 \mod 6 \\
1 & v_p(t) -4v_p(\hr)\equiv 0 \mod 6, \end{cases}$$
which gives
\begin{eqnarray*}
&&\leg{-1}{p}^{v_p(\hr)} \int_{0 \leq v_p(\hr) < v_p(t)} w_{p}^*(p) \, dt =
\sum_{{e>  v_p(\hr) } \atop {e - 4v_p(\hr) \not\equiv 0 \mod 6}} \frac{p-1}{p^{e+1}} \leg{-3}{p}^{e+1} + \sum_{{e> v_p(\hr) } \atop {e - 4v_p(\hr) \equiv 0 \mod 6}} \frac{p-1}{p^{e+1}}. \\
 \end{eqnarray*}
Thus, for $v_p(\hr) \equiv j \mod 2$ with $j\in\{0,1\}$ the first sum is
 \est{
%& \sum_{{e> v_p(\hr)} } \frac{p-1}{p^{e+1}} \leg{-3}{p}^{e+1}  - \sum_{{e> v_p(\hr) } \atop {e \equiv 4v_p(\hr)  \mod 6}} \frac{p-1}{p^{e+1}} \leg{-3}{p}^{e+1}\\
%&=
&
 \sum_{{e> v_p(\hr)} } \frac{p-1}{p^{e+1}} \leg{-3}{p}^{e+1}
- \leg{-3}{p}\sum_{{e> -\frac12 v_p(\hr) } } \frac{p-1}{p^{4v_p(\hr) +6e+1}} \\
&\hspace{15em}= \leg{-3}{p}^{v_p(s)}\frac{p-1}{p^{v_p(s)+1}} \frac{1 }{p- \leg{-3}{p} }
- \leg{-3}{p}\frac{(p-1)p^{3j}}{p^{v_p(\hr) +1}(p^6-1)} \\
}
whereas the second gives
\est{
 \sum_{{e> v_p(\hr) } \atop {e - 4v_p(\hr) \equiv 0 \mod 6}} \frac{p-1}{p^{e+1}}=\frac{(p-1)p^{3j}}{p^{v_p(\hr)+1 }(p^6-1)}.
}
Thus,
\begin{eqnarray}
\int_{v_p(t)>v_p(\hr)} w^*_{p}(t) \; dt=\leg{-1}{p}^{j}  \begin{cases}
\frac{1}{p^{v_p(s)+1}}  & p \equiv 1 \mod 3, \\
\frac{(-1)^{j} }{p^{v_p(s)+1}} \frac{p-1 }{p+1 } + \frac{2(p-1)p^{3j}}{p^{v_p(\hr) +1}(p^6-1)}& p \equiv 2 \mod 3 . \end{cases}
\end{eqnarray}
Finally, summing
$$\int_{\Z_p} w^*_{p}(t) \, dt = \int_{0 \leq v_p(t) < v_p(\hr)} w^*_{p}(t) \, dt +
\int_{v_p(t)=v_p(\hr)} w^*_{p}(t) \, dt +
\int_{0 \leq v_p(\hr) < v_p(t)} w^*_{p}(t) \, dt,$$ we get the result.
\end{proof}

 \begin{prop} \label{valuesW3}
We have $\int_{\Z_3} w^*_{3}(t) \; dt = \Mh(3) $.
\end{prop}

\begin{proof}

We recall that $$w^*_{3}(t) = (-1)^{v_3(t)} (-1)^{v_3(t-\hr)} w_{3}(t),$$
and that the values of $w_{3}(t)$ are given in Proposition \ref{app-RN3} of Appendix \ref{appen_2}.

\subsubsection*{The case $0 \leq v_3(\hr) < v_3(t)$.} In this case we have $v_3(t-\hr) = v_3(\hr)$. Also, from Appendix \ref{appen_2}, we have that $w_{3}(t)$ depends only on $v_3(t)$ and $(t-\hr)_3 \mod 9$ (and possibly $\hr_3$ and $v_3(\hr)$). Thus, if $v_3(t) \equiv v_3(\hr) \mod 3$, we have that
$$\int_{v_3(t)=e} w^*_{3}(t) \; dt = \frac{(-1)^{v_3(\hr)+e}}{3^{e+2}} \sum_{d \in (\Z/9\Z)^*} w_{3, \hr}(dp^e) = \frac{2 (-1)^{v_3(\hr)+e}}{3^{e+2}}.$$
If $v_3(t) - v_3(\hr) \not\equiv 0 \mod 3$, then the integral is easily seen to be 0.
This gives that
\begin{eqnarray} \nonumber
\int_{0 \leq v_3(\hr) < v_3(t)} w^*_{3}(t) \; dt &=& \frac{2(-1)^{v_3(\hr)}}{9} \sum_{\substack{e > v_3(\hr)\\ e \equiv v_3(\hr) \mod 3}}  \frac{(-1)^e}{3^e} \\ \notag%\label{case:vpt>vpr}
&=&
\frac{2}{3^{v_3(\hr)+2}} \sum_{n=1}^\infty \left( {-1}/{3} \right)^{3n} = \frac{-1}{14 \cdot 3^{v_3(\hr)+2}}.
%%%\frac{2}{3^{v_p(\hr)+5}} \frac{p^3}{p^3 -1}.
\end{eqnarray}

\subsubsection*{The case  $0 \leq v_3(t) = v_3(\hr)$.}

Let $e=v_3(t) = v_3(\hr)$ and $v_3(t-\hr) = e+k$ with $k \geq 1$ so that $w^*_{3}(t) = (-1)^{k} w_{3}(t)$. % and $w_{3}(t)$ depends only on $(t-\hr)_3 \mod 9$ (and possibly $\hr_3$ and $e$).
First, we consider the case $k\geq1$.
If $k \equiv 0 \mod 3$ (and then $k \geq 3$),  then $w_{3}(t)$ is determined by a congruence modulo 9 on $(t-\hr)_3$, and we compute
\begin{eqnarray*}
\int_{{v_3(t)=v_3(\hr)=e,} \atop {v_3(t-\hr)=e+k}} w^*_{3}(t) \; dt &=& \frac{(-1)^k}{3^{e+k+2}}
\sum_{{d \in (\Z/3^{k+2}\Z)^*} \atop  {{d \equiv \hr_3 \mod {3^k}} \atop {d \not\equiv \hr_3 \mod {3^{k+1}}}}}w_{3}( (d-\hr_3)_3)%\\
%&=&  \frac{(-1)^k}{3^{e+k+2}} \sum_{a \in (\Z/9\Z)^*} w_3 \left( (a 3^k)_3 \right)
= 2 \frac{(-1)^k}{3^{e+k+2}} .
 \end{eqnarray*}
 For $k \not\equiv 0 \mod 3$ and $k\geq1$ we easily get
$$
\int_{\substack{v_3(t)=v_3(\hr)=e \\ v_3(t-\hr)=e+k}} w^*_{3}(t) \; dt = 0
$$
and thus
\begin{eqnarray} \label{case:vpt=vpr:first}
\int_{\substack{v_3(t)=v_3(\hr)=e \\v_3(t-\hr)> e}} w^*_3(t) \; dt &=& \frac{2}{3^{e+2}} \sum_{{k=1} \atop {k \equiv 0 \mod 3}}^\infty (-1)^k/3^k = - \frac{1}{14} \frac{1}{3^{2+v_3(\hr)}}.
  \end{eqnarray}
Next, $k=0$, that is $e=v_3(t)=v_3(\hr) =v_3(t-\hr).$ Then, we must have $v_3(2t-\hr)=e+\ell$ with $\ell \geq 1$.  Also, in this case
$w_{3}^*(t) = w_{3}(t).$
If $\ell \geq 2$, then $w_{3}(t) = 1$ and
\est{
\int_{{v_3(t)=v_3(\hr)=e,} \atop {v_3(2t-\hr)=e+\ell }}  w^*_3(t) \; dt
 &= \frac{2}{3^{e+\ell+1}} ,
}
whereas the integral is quickly seen to be $0$ if $\ell=1$. Thus,
\begin{eqnarray} \label{case:vpt=vpr:second}
\int_{v_3(t)=v_3(\hr)=e, v_3(t-\hr)=e} w^*_3(t) \; dt &=& \frac{2}{3^{e+1}} \sum_{{k=2}}^\infty 3^{-\ell} = \frac{1}{3^{v_3(\hr)+2}}.
  \end{eqnarray}
Summing the contributions \eqref{case:vpt=vpr:first} and \eqref{case:vpt=vpr:second}, we get
\begin{eqnarray}\notag %\label{case:vpt=vpr}
\int_{v_3(t)=v_3(\hr)} w^*_{3}(t)  \; dt = \frac{13}{14} \frac{1}{3^{v_3(\hr)+2}}.
\end{eqnarray}

\subsubsection*{The case  $0 \leq v_3(t) < v_3(\hr)$.}
We let $v_3(t)=e$. In this case we have $v_3(t-\hr) = v_3(t)=e$, and so $w^*_{3}(t) = w_{3}(t).$ If $e\equiv 0\mod 2$ and $e<v_3(s)-1$ then $w-3(t)=-1$ and so we find
$$
\int_{v_3(t)=e} w^*_3(t) \; dt   = \frac{-2}{3^{e+1}},
$$
whereas in all other cases the above integral is $0$. Thus,
\begin{eqnarray*}
\int_{0 \leq v_3(t) < v_3(\hr)} w^*_{3}(t) \; dt &=& -\frac{2}{3} \sum_{\substack{0 \leq e \leq v_3(\hr)-2, \\ e \equiv 0 \mod 2}} \frac{1}{3^e} =\frac{-2}{3} \sum_{0 \leq n \leq N} \frac{1}{9^n}= \frac{3}{4} \left( 9^{-(N+1)} - 1 \right),
\end{eqnarray*}
where $N = \lfloor \frac{v_3(\hr) - 2}{2} \rfloor.$

Then, summing the contribution of the three cases we obtain the proposition.
%Then, summing the contribution of the three cases, \eqref{case:vpt>vpr},  \eqref{case:vpt=vpr} and \eqref{case:vpt<vpr}, this proves the theorem.
\end{proof}

\begin{prop} \label{valuesW2}
We have $\int_{\Z_2} w^*_{2}(t) \; dt = \Mh(2)$.
\end{prop}

\begin{proof}
The proof of the proposition is given by a series of lemmas, which compute the contribution to $\int_{\Z_2} w^*_{2}(t) \; dt$ in $4$ cases depending on the relative valuations of $t$ and $\hr$. To obtain Proposition \ref{valuesW2}, it is then enough to sum the $4$ contributions.

Before proceeding with the lemmas we recall that for $t \not\in [0,\hr]$ (note that excluding a finite number of values of $t$ does not influence the various averages) we have
$$w^*_{2}(t) \equiv t_2 (t-\hr)_2 w_{2}(t) \mod 4.$$
and that the values of $w_{2}(t)$ are given in Proposition~\ref{app-RN2} in Appendix~\ref{appen_2}.
\end{proof}
\begin{lemma}\label{v-smaller}
Let $\chi_4$ be non-principal character modulo $4$.
If $v_2(\hr)$ is even, then
$$\int_{0 \leq v_2(\hr) < v_2(t)} w^*_{2}(t) \; dt = -\frac{1}{2^{v_2(\hr)+2}} + \chi_4(\hr_2) \frac{1}{2^{v_2(\hr)+4}} - \chi_4(\hr_2) \frac{29}{63} \frac{1}{2^{2+v_2(\hr)}}.$$
If $v_2(\hr)$ is odd,
$$\int_{0 \leq v_2(\hr) < v_2(t)} w^*_{2}(t) \; dt =
 \chi_4(\hr_2) \frac{1}{2^{v_p(\hr)+5}}
-\chi_4(\hr_2) \frac{46}{63} \frac{1}{2^{v_2(\hr)+4}} .
$$
\end{lemma}

\begin{proof} We first remark that if $0 \leq v_2(\hr) < v_2(t)=e,$ then
\begin{eqnarray*}
t_2 (t-\hr)_2 &\equiv& \begin{cases} - t_2 \hr_2 \mod 4 &\text{ if } v_2(t) - v_2(\hr) \geq 2,\\
 t_2 \hr_2 \mod 4 &\text{ if } v_2(t) - v_2(\hr) = 1. \end{cases}
\end{eqnarray*}
We first suppose that $v_2(\hr)$ is even.
If $v_2(t) - v_2(\hr) \equiv 0,2 \mod 6$, and $v_2(t) \neq  v_2(\hr)+2$, then since $w^*_{2}(t) \equiv  t_2 \hr_2 \mod 4$, it is clear that
$$
\int_{v_2(t) = e} w^*_{2}(t) \; dt = 0.$$
If $v_2(t) - v_2(\hr) \equiv 1,3,4,5 \mod 6$, and $v_2(t) - v_2(\hr) > 1$, then
$w^*_{2}(t) \equiv - \hr_2 \mod 4,$ and
$$
\int_{v_2(t) = e} w^*_{2}(t) \; dt = \frac{1}{2^{e+1}} \sum_{d \in (\Z/2\Z)^*} w^*_{2}(dp^e) =- \chi_4(\hr_2) \frac{1}{2^{e+1}}.
$$
Then, if $v_2(\hr)$ is even, we then have that
\begin{eqnarray} \label{sum-even}  \sum_{e \geq v_2(\hr)+3} \int_{v_2(t)=e} w^*_{2}(t) \; dt
&=& - \chi_4(\hr_2)  \sum_{{e \geq v_2(\hr)+3}\atop {e - v_2(\hr) \not\equiv 0,2 \mod 6}} \frac{1}{2^{e+1}}
%%&=& \sum_{e \geq v_2(\hr)+3} \frac{1}{2^{e+1}} - \sum_{6n + v_2(\hr) \geq v_2(\hr)+3} %%\frac{1}{2^{6n+v_2(\hr)+1}} - \sum_{6n + v_2(\hr)+2 \geq v_2(\hr)+3} \frac{1}{2^{6n+v_2(\hr)+3}}\\
%%&=& \frac{1}{2^{v_2(\hr)+3}} - \sum_{n=1}^\infty \frac{1}{2^{6n+v_2(\hr)+1}} - \sum_{n=1}^\infty %%\frac{1}{2^{6n+v_2(\hr)+3}} \\
= -  \frac{\chi_4(\hr_2) }{2^{v_2(\hr)+2}} \, \frac{29}{63}.
\end{eqnarray}
If $v_2(t) - v_2(\hr) =1$, then $w_{2}^*(t) \equiv - t_2 \hr_2^{-1} t_2 \hr_2 \mod 4 \equiv -1 \mod 4$, so $w_{2}^*(t)=-1$ and we have
$$
\int_{v_2(t) = v_2(\hr) + 1} w^*_{2}(t) \; dt = \frac{-1}{2^{v_2(\hr)+2}}.$$
Finally, if $v_2(t) - v_2(\hr) =2$, then we compute
$$w^*_{2}(d2^e) \equiv -d \hr_2 w_{2}(d2^e ) \mod 4 = \begin{cases} 1 & d \equiv 1 \mod 8,\\
1 & d \equiv 3,7 \mod 8, \hr_2 \equiv 1 \mod 4,\\
-1 & d \equiv 3,7 \mod 8,\hr_2 \equiv 3 \mod 4,\\
-1 & d \equiv 5 \mod 8, \end{cases}
$$
and so
$$
\int_{v_2(t) = v_2(\hr) + 2} w^*_{2}(t) \; dt = \frac{1}{2^{e+3}} \sum_{d \in (\Z/8\Z)^*} w^*_{2}(dp^e) = \frac{ \chi_4(\hr_2)}{2^{v_2(\hr)+4}}.$$
Adding the contributions for $v_p(t)=v_p(\hr)+1$ and $v_p(t)=v_p(\hr)+2$ to \eqref{sum-even}, we get the result for $v_2(\hr)$ even.

We now suppose that $v_2(\hr)$ is odd. If $v_2(t) - v_2(\hr) \equiv 0,2,4 \mod 6$, or $v_2(t) - v_2(\hr) \equiv 1 \mod 6$ and $v_2(t) \neq v_2(\hr)+1$, then $w_{2}^*(t) \equiv - t_2^2 \hr_2 \equiv - \hr_2 \mod 4$  and so, as before,
$$
\int_{v_2(t) = e} w^*_{2}(t) \; dt  = %%\begin{cases} - \frac{1}{2^{e+1}} & r_2 \equiv 1 \mod 4 \\ \frac{1}{2^{e+1}} & r_2 \equiv 3 \mod 4 . \end{cases}
- \chi_4(\hr_2) \frac{1}{2^{e+1}}.
$$
If $v_2(t) - v_2(\hr) \equiv 3,5 \mod 6$ and $v_2(t)-v_2(\hr) \neq 3$, then $w_{2}^*(t) \equiv  t_2 \hr_2 \mod 4$ and so
$$
\int_{v_2(t) = e} w^*_{2}(t) \; dt = 0.$$
Thus, if $v_2(\hr)$ is odd, we then have that
\begin{eqnarray} \label{sum-odd}
 \sum_{e \geq v_2(\hr)+4} \int_{v_2(t)=e} w^*_{2}(t) \; dt = - \chi_4(\hr_2) \sum_{{e \geq v_2(\hr)+4} \atop {e - v_2(\hr) \not\equiv 3,5 \mod 6}} \frac{1}{2^{e+1}}
=-\frac{ \chi_4(\hr_2) }{2^{v_2(\hr)+4}}  \frac{46}{63} .
\end{eqnarray}
We then have to treat the 2 remaining cases $v_2(t)=v_2(\hr)+1$ and $v_2(t)=v_2(\hr)+3$. In the latter case,
we have that
$$w_{2}^*(t) \equiv \begin{cases} t_2 \hr_2 & t_2 \equiv 5 \mod 8, \\ - t_2 \hr_2 & t_2 \not\equiv 5 \mod 8, \end{cases}$$
and then
$$
\int_{v_2(t) = v_2(\hr) + 3} w^*_{2}(t) \; dt = \frac{1}{2^{e+3}} \sum_{d \in (\Z/8\Z)^*} w^*_{2}(dp^e) = %%%\begin{cases} \frac{2}{2^{e+3}} & r_2 \equiv 1 \mod 4 \\ -\frac{2}{2^{e+3}} & r_2 \equiv -1 \mod 4 \end{cases}
 \frac{\chi_4(\hr_2)}{2^{e+2}}.
$$
Finally, if $v_2(t) - v_2(\hr) =1$, then we have
$$w^*_{2}(t) = \chi_4(\hr_2 )\begin{cases} 1 & t_2 \equiv 1 \mod 8\\
-1 & t_2 \equiv 3 \mod 8, \hr_2 \equiv 1 \mod 4\\
1 & t_2 \equiv 3 \mod 8, \hr_2 \equiv 3 \mod 4\\
-1 & t_2 \equiv 5 \mod 8\\
1 & t_2 \equiv 7 \mod 8, \hr_2 \equiv 1 \mod 4 \\
-1 & t_2 \equiv 7 \mod 8, \hr_2 \equiv 3 \mod 4 \end{cases}
$$
and so
$$
\int_{v_2(t) = v_2(\hr) + 1} w^*_{2}(t) \; dt = \frac{1}{2^{e+3}} \sum_{d \in (\Z/8\Z)^*} w^*_{2}(dp^e) = %%%%\begin{cases} \frac{2}{2^{e+3}} &  r_2 \equiv 1 \mod 4 \\
%%%- \frac{2}{2^{e+3}} &  r_2 \equiv 3 \mod 4 \end{cases}
0.
$$
Adding the contributions of $v_p(t)=v_p(\hr)+1$ and $v_p(t)=v_p(\hr)+3$ to \eqref{sum-odd}, we get the result.
\end{proof}

\kommentar{
\begin{theorem}
If $v_2(\hr)$ is even, let  $N_0 = \lceil (v_2(\hr) + 3)/6 \rceil$ and $N_2 = \lceil (v_2(\hr) + 1)/6 \rceil$. Then,
$$\int_{0 \leq v_2(\hr) < v_2(t)} w^*_{2}(t) \; dt = - 2^{-v_2(\hr)-2} - \leg{\hr_2}{4} \frac{1}{2^{v_p(\hr)+4}} + \leg{\hr_2}{4}  \frac{64}{63} \left( \frac{1}{2^{6N_0+1}} + \frac{1}{2^{6N_2+3}} \right).
$$
If $v_2(\hr)$ is odd, then
$$\int_{0 \leq v_2(\hr) < v_2(t)} w^*_{2}(t) \; dt = ...$$
\end{theorem}
\begin{proof} If $v_2(\hr)$ is even, then
\begin{eqnarray*} - \leg{\hr_2}{4} \sum_{e \geq v_2(\hr)+3} w^*_{2}(t) &=& \sum_{e \geq v_2(\hr)+3} \frac{1}{2^{e+1}} - \sum_{6n \geq v_2(\hr)+3} \frac{1}{2^{6n+1}} - \sum_{6n + 2 \geq v_2(\hr)+3} \frac{1}{2^{6n+3}}\\
&=& \frac{1}{2^{v_2(\hr)+3}} - \frac{1}{2^{6N_0+1}} \sum_{n=0}^\infty 2^{-6n} - \frac{1}{2^{6N_2+3}} \sum_{n=0}^\infty 2^{-6n}\\
&=& \frac{1}{2^{v_2(\hr)+3}} - \frac{64}{63} \left( \frac{1}{2^{6N_0+1}} + \frac{1}{2^{6N_2+3}}  \right).
\end{eqnarray*}
Adding the contributions for $v_p(t)=v_p(\hr)+1$ and $v_p(t)=v_p(\hr)+2$, we get the result.
 \end{proof}
}

\begin{lemma} \label{lemma-vgeq2}
Suppose that $v_2(\hr) \geq 2$. Then,
$$\int_{0 \leq v_2(t) \leq v_2(\hr) - 2} w^*_{2}(t) \; dt = \begin{cases} \frac{1}{4} & v_2(\hr)=2 \\
%\frac{1}{4} & v_2(\hr)=3 \\
%\frac{1}{16} & v_2(\hr)=4\\
{2^{1-v_2(\hr)}} + \frac{1}{3} \left( 4^{-(v_2(\hr)-3)/2} - 1 \right) & \mbox{$v_2(\hr) \geq 3$ odd} \\
 {2^{-v_2(\hr)}}+ \frac{1}{3} \left( 4^{-(v_2(\hr)-4)/2} - 1 \right) & \mbox{$v_2(\hr) \geq 4$ even}
  \end{cases}
$$
\end{lemma}
\begin{proof} Since $e=v_p(t) \leq v_p(\hr)-2$, we have that $(t-\hr)_2 = t_2 -2^k \hr_2$ for $k=v_2(\hr)-v_2(t) \geq 2$, and $(t-\hr)_2 \equiv t_2 \mod 4$, which gives $w_{2}^*(t) = w_{2}(t)$.
First, suppose that $v_2(t)$ is even. Then, it is easy to see from Proposition \ref{app-RN2} of Appendix \ref{appen_2} that $$\int_{v_2(t) = e} w^*_{2,r}(t) \; dt =
\begin{cases}
\frac{1}{2^{v_2(\hr)}} & e=v_2(\hr)-2, \\
\frac{1}{2^{v_2(\hr)-1}} & e=v_2(\hr)-3, \\
0 &  e=v_2(\hr)-4,\\
\frac{-1}{2^{e+2}} & e \geq v_2(\hr)-5. \end{cases}$$
We now suppose that  $v_2(t)$ is odd. Then, it is clear that
$$
\int_{v_2(t) = e} w^*_{2}(t) \; dt = 0.
$$
Thus, summing all contributions
$$\int_{0 \leq v_2(t) \leq v_2(\hr) - 2} w^*_{2}(t) \; dt = \sum_{{0 \leq e \leq v_2(\hr) - 2} \atop {e \text{ even}}}
\int_{v_2(t) = e} w^*_{2}(t) \; dt ,
$$
we get the result.
\end{proof}

\begin{lemma} \label{lemma-vequalminus}We have
$$\int_{v_2(t) = v_2(\hr) - 1} w^*_{2}(t) \; dt = 0.$$
\end{lemma}
\begin{proof} If $v_2(t) = v_2(\hr)-1$, then $(t-\hr)_2 = t_2 - 2 \hr_2,$ and one check that $t_2-2\hr_2 \equiv 1 \mod 4 \iff t_2 \equiv -1 \mod 4,$ which gives
$$w_{2}^*(t) \equiv t_2 (t-\hr)_2 w_{2}(t) \equiv - w_{2}(t) \mod 4.$$ From Proposition~\ref{app-RN2} of Appendix \ref{appen_2}, we then easily deduce that
$$\int_{v_2(t)=v_2(\hr)-1} w_{2}^*(t) \; dt = 0$$ for all cases.\end{proof}

\begin{lemma} \label{v-equal}
If $v_2(\hr)$ is even, then
$$\int_{v_2(t) = v_2(\hr)} w^*_{2}(t) \; dt = -\chi_4(\hr_2) \frac{1}{2^{v_2(\hr)+4}} - \frac{1}{2^{v_2(\hr)+2}} + \chi_4(\hr_2) \frac{1}{2^{v_2(\hr)+2}} \frac{29}{63}.$$
If $v_2(\hr)$ is odd, then
$$\int_{v_2(t) = v_2(\hr)} w^*_{2}(t) \; dt = - \chi_4(\hr_2) \frac{1}{2^{v_2(\hr)+5}} + \chi_4(\hr_2) \frac{1}{2^{v_2(\hr)+4}} \frac{46}{63}.$$
\end{lemma}

\begin{proof} Let $e=v_2(t)=v_2(\hr)$ and $v_2(t-\hr) = e + k.$ Notice that $k \geq 1$ since $t_2 - \hr_2 \equiv 0 \mod 2$.
%In this case, we have $$w_{2}^*(t) \equiv t_2 (t-\hr)_2 w_{2}(t) \mod 4.$$
We first suppose that $v_2(\hr)$ is even.
If $k \equiv 0,2 \mod 6$, $k \geq 3$, then $w_{2}^*(t) \equiv - t_2 (t-\hr)_2\equiv - s_2 (t-\hr)_2 \mod 4$
and so
\begin{eqnarray*}
\int_{\substack{v_2(t)=v_2(\hr)=e \\ v_2(t-\hr)=e+k}} \; dt
&=& \frac{-1}{2^{e+k+2}} \sum_{a \in (\Z/4\Z)^*}w_2^*(2^{e}(s_2+a2^k))=0.
\end{eqnarray*}
If $k \equiv 1,3,4,5 \mod 6$, $k \geq 3$, then $w_{2}^*(t) \equiv s_2\mod 4$,
and
\begin{eqnarray*}
\int_{\substack{v_2(t)=v_2(\hr)=e \\v_2(t-\hr)=e+k}} w^*_{2}(t) \; dt &=&  \chi_4(\hr_2) \frac{1}{2^{e+k+1}}.
\end{eqnarray*}
We then have that
\begin{eqnarray}
 \int_{\substack{e=v_2(t)=v_2(\hr)\\ v_2(t-\hr) \geq e+3}} w_{2}^*(t) \; dt
\label{kgeq3} &&= \chi_4(\hr_2)\sum_{\substack{k \geq 3 \\k \not\equiv 0,2 \mod 6}} \frac{1}{2^{k+v_2(\hr)+1}}
=\chi_4(\hr_2) \frac{1}{2^{v_2(\hr)+2}}\frac{29}{63}.
\end{eqnarray}
%since this is the same sum as \eqref{sum-even} with the change of variable $k=e-v_2(\hr)$.
%
%%%and clearly $$\int_{v_p(t)=e, v_2(t-r)=e+k} \; dt =0.$$
We now have to compute the contribution for $v_2(t-\hr)=v_2(\hr)+1$ and $v_2(t-\hr)=v_2(\hr)+2$.
For the first case $v_2(t-\hr)=v_2(\hr)+1$, by Proposition~\ref{app-RN2} we have
\es{ \label{kequal1}
\int_{\substack{e=v_2(t)=v_2(\hr)\\ v_2(t-\hr)=e+1}} w_{2}^*(t) \; dt %&= \frac{1}{2^{e+3}} \sum_{{{d \in (\Z/2^{3}\Z)^*} \atop {d \equiv \hr_2 \mod {2}}} \atop  {d \not\equiv \hr_2 \mod {2^{2}}}} \chi_4(d (d-\hr_2)_2 )w_{2}(d2^e)  \\
&% \hspace{-3cm}
 =  \frac{1}{2^{e+3}}
\sum_{a \in (\Z/4\Z)^*} \chi_4(a(\hr_2+2a))\, w_{2}(2^e(\hr_2+2a))
= -2^{-e-2}.
}
For the second case $v_2(t-\hr)=v_2(\hr)+2$, we have $w_{2}^*(t)\equiv s_2(t-s_2)w_2(t)\mod 4$ and so
\es{ \label{kequal2}
\int_{\substack{e=v_2(t)=v_2(\hr)\\v_2(t-\hr)=e+2}} w_{2}^*(t) \; dt %&= \frac{1}{2^{e+5}} \sum_{{{d \in (\Z/2^{5}\Z)^*} \atop {d \equiv \hr_2 \mod {4}}} \atop  {d \not\equiv \hr_2 \mod {2^{3}}}} d (d-\hr_2)_2 w_{2}(d) \mod 4 \\
 &= \frac{\chi_4(s_2)}{2^{e+5}} \sum_{a \in (\Z/8\Z)^*} \chi_4(a) \, w_{2}(2^e(\hr_2+4a)) =-\chi_4(s_2)2^{-e-4},
}
since $w_{2}(2^e(\hr_2+4a)) = 1$ only in the cases $a \equiv 1,3,7 \mod 8$ if $\hr_2 \equiv 1 \mod 4$, or $a \equiv 1,3, 5 \mod 8$ if $\hr_2 \equiv 3 \mod 4$.
Summing \eqref{kequal1}, \eqref{kequal2} and \eqref{kgeq3}, we get the result when $v_2(\hr)$ is even.

Suppose now that $e=v_2(\hr)=v_2(t)$ is odd. If $k = v_2(t-\hr)-v_2(\hr) \equiv 0,1,2,4 \mod 6$, and $k \geq 4$, then
$w^*_{2}(t) \equiv t_2 (t-\hr)_2 w_{2}(t) \equiv t_2 \equiv s_2\mod 4$, and
\begin{eqnarray*}
\int_{\substack{e=v_2(t)=v_2(\hr)\\ v_2(t-\hr)=e+k}} w^*_{2}(t) \; dt %&=& \frac{1}{2^{e+k+2}} \sum_{{{d \in (\Z/2^{k+2}\Z)^*} \atop {d \equiv \hr_2 \mod {2^k}}} \atop  {d \not\equiv \hr_2 \mod {{2^{k+1}}}}} d \mod 4 \\
&=& %\frac{1}{2^{e+k+2}} \sum_{a \in (\Z/4\Z)^*} \hr_2 \mod 4 =
 \chi_4(s_2)\frac{1}{2^{e+k+1}}.
\end{eqnarray*}
If $k = v_2(t-\hr)-v_2(\hr) \equiv 3,5 \mod 6$, and $k \geq 4$, then
$w^*_{2}(t) \equiv t_2 (t-\hr)_2 w_{2}(t) \equiv  - s_2 (t-\hr)_2 \mod 4$ and, as before,
\begin{eqnarray*}
\int_{\substack{e=v_2(t)=v_2(\hr)\\ v_2(t-\hr)=e+k}} w^*_{2}(t) \; dt % &=& \frac{-1}{2^{e+k+2}} \sum_{{{d \in (\Z/2^{k+2}\Z)^*} \atop {d \equiv \hr_2 \mod {2^k}}} \atop  {d \not\equiv \hr_2 \mod {{2^{k+1}}}}} d (d-\hr_2)_2 \mod 4 \\
%&=& \frac{- 1}{2^{e+k+2}} \sum_{a \in (\Z/4\Z)^*} \hr_2 \, a \mod 4
= 0.
\end{eqnarray*}
Summing the contributions above, we get that
\begin{eqnarray}\label{a1}
\int_{\substack{v_2(t)=v_2(\hr)\\v_2(t-\hr)\geq v_2(\hr) + 4}} w^*_{2}(t) \, dt =\chi_4(\hr_2)  \sum_{{k \geq 4} \atop {k \not\equiv 3,5 \mod 6}} \frac{1}{2^{v_2(\hr)+k+1}} = \chi_4(\hr_2) \frac{1}{2^{v_2(\hr)+4}} \frac{46}{63} . \end{eqnarray}
%since this is the same sum as \eqref{sum-odd} with the change of variable $e = v_2(\hr) + k.$
%
We now have to compute the contributions for $v_2(t-\hr)=v_2(\hr)+1$ and $v_2(t-\hr)=v_2(\hr)+3$.
For the first case $v_2(t-\hr)=v_2(\hr)+1$, we have
\es{\label{a2}
\int_{\substack{e=v_2(t)=v_2(\hr)\\ v_2(t-\hr)=e+1}} w_{2}^*(t) \; dt% &= \frac{1}{2^{e+4}} \sum_{{{d \in (\Z/2^{4}\Z)^*} \atop {d \equiv \hr_2 \mod {2}}} \atop  {d \not\equiv \hr_2 \mod {2^{2}}}} d (d-\hr_2)_2 w_{2}(d) \mod 4 \\ &=& \frac{1}{2^{e+4}}
\sum_{a \in (\Z/8\Z)^*} \chi_4(a(\hr_2+2a) ) w_{2}(2^e(\hr_2+2a))=0
}
since $w_{2}(2^e(\hr_2+2a)) = 1$ only in the cases $a \equiv 1,7 \mod 8$ if $\hr_2 \equiv 1 \mod 4$ or $a \equiv 1, 3 \mod 8$ if $\hr_2 \equiv 3 \mod 4$.
Finally, if $v_2(t-\hr)=v_2(\hr)+3$, we find
\begin{eqnarray} \label{a3}
\int_{\substack{e=v_2(t)=v_2(\hr)\\ v_2(t-\hr)=e+3}} w_{2}^*(t) \; dt &=& - \chi_4(\hr_2) \frac{1}{2^{v_2(\hr)+5}}. \end{eqnarray}
\kommentar{
\begin{eqnarray*}
\int_{e=v_2(t)=v_2(\hr), v_2(t-\hr)=e+2} w_{2}^*(t) \; dt &=& \frac{1}{2^{e+4}}
\sum_{{{d \in (\Z/2^{4}\Z)^*} \atop {d \equiv \hr_2 \mod {2}}} \atop  {d \not\equiv \hr_2 \mod {2^{2}}}} d (d-\hr_2)_2 w_{2}(d) \mod 4 \\ &=& \frac{1}{2^{e+4}}
\sum_{a \in (\Z/8\Z)^*} (\hr_2+2a) \, a \, w_{2}(\hr_2+2a) \mod 4
\end{eqnarray*}}
Summing \eqref{a1}, \eqref{a2} and \eqref{a3}, we get the result when $v_2(\hr)$ is odd.
\end{proof}

\kommentar{\begin{lemma}
If $v_2(\hr)$ is even, then
$$\int_{v_2(t) \geq v_2(\hr)} w^*_{2}(t) \; dt = \frac{-1}{2^{v_2(\hr)+1}}.$$
If $v_2(\hr)$ is odd, then
$$\int_{v_2(t) \geq v_2(\hr)} w^*_{2}(t) \; dt = 0.$$
\end{lemma}
\begin{proof} We are just summing the formulas of Lemma \ref{v-smaller} and \ref{v-equal}.
\end{proof}}

\section{The density of average root numbers}\label{densityaverage}

We shall prove Theorem~\ref{rarn},~\ref{rarn3} and~\ref{rarn2} by considering subfamilies of $\was_a(t)$ of the form $\was_{a(t)}(Q(t))$ where  $a(t)$ and $Q(t)$ are polynomials in $\Z[t]$. Thanks to Theorem~\ref{average-periodic}, we know exactly the root number for all the elliptic curves in these families, and so we just need to choose $a(t)$ and $Q(t)$ so that we obtain the desired averages.
In the case of averages over $\Q$, we can reduce to the case where the $\infty$-factor of the (modified) root number essentially determines the root number, whereas in the the case of averages over $\Z$, we reduce to the case where the root number is determined by its $p$-factor for a suitably chosen $p$. The proof of Theorem~\ref{rarn3} is more elaborate and requires working with all prime divisors of $k$.

\begin{proof}[Proof of Theorem ~\ref{rarn}]

We first prove that $\Av_\Z(\mathfrak F_{\Z}') \supseteq \Q \cap [-1,1]$.

For any $h/k\in \Q$ with $(h,k)=1$, $k>0$ we need to show that there exists a non-isotrivial family $\EE$ such that $\Av_\Z(\EE)=h/k$. First notice that we can assume $0\neq |h/k|<1$, since by Theorem~\ref{average-periodic} we have that $\was_{2}(1+4t)$, $\was_{1}(t)$ and $\was_{3}(1+12t)$ are non-isotrivial families with root numbers constantly equal to $(-1)^t$, $-1$ and $1$ respectively. Also, let $h=\pm |h|$.

Let $p$ be a prime such that $p+1=2r k$ for $r\geq1$. By Dirichlet's Theorem on  primes in  arithmetic progressions we can always find such a prime. Let $m=p+1-2r|h|$ so that $0<m<p$ and $m$ is even. Let
\est{
P(t)=\mp p\prod_{i=1}^m(t-i),\qquad a(t)=2^4 p P(t),\qquad Q(t)=(4pt^2+1)P(t),
}
so that by~\eqref{invariant_w} one easily sees that $\was_{a(t)}(Q(t))$ is a  potentially parity-biased non-isotrivial family. %%potentially strange over $\Z$ but not over $\Q$?

We shall assume $t\neq 1,\dots, m$ so that $P(t)\neq0$ and we let $\eps(t)$ be the root number of $\was_{a(t)}(Q(t)$. First, notice that
\est{
\gcd(a(t)_2,Q(t))= |P(t)_2|=\frac{|a(t)_2|}{p}.
}
Then, by~\eqref{mf_average-periodic} for $t\neq 1,\dots, m$ we have
\est{
\eps(t)&\equiv -s_{a(t)}(Q(t))\gcd(a(t)_2,Q(t))\prod_{\substack{q|\frac{a(t)_2}{\gcd(a(t)_2,Q(t))},\\ q\text{ prime}}}(-1)^{1+v_q(Q(t))}\pr{\frac{Q(t)_q}{q}}^{1+v_q(Q(t))}\mod 4\\
&\equiv -s_{a(t)}(Q(t))|P(t)_2|(-1)^{1+v_p(Q(t))}\pr{\frac{Q(t)_p}{p}}^{1+v_p(Q(t))}\mod 4\\
&\equiv -Q(t)_2|Q(t)_2|(-1)^{1+v_p(Q(t))}\pr{\frac{Q(t)_p}{p}}^{1+v_p(Q(t))}\mod 4\\
}
since $v_2(a(t))=4+v_2(Q(t))$ and thus by Remark~\ref{sat_particular_case} after Proposition~\ref{sat} we have $s_{a(t)}\equiv Q(t)_2\mod 4$.
Now, for $t\neq 1,\dots,m$ an integer we have that $|Q(t)|=\mp Q(t)$, so that
\est{
\eps(t)=\pm(-1)^{1+v_p(Q(t))}\pr{\frac{Q(t)_p}{p}}^{1+v_p(Q(t))}.
}
Now notice that $v_p(Q(t))=1$ unless $t\equiv 1,\dots,m\mod p$, whereas if $t\equiv i\mod p$ with $i\in\{1,\dots,m\}$ then $v_p(Q(t))=1+v_p(t-i)$ and $Q(t)_p=(4pt^2 + 1) \kappa_i(t-i)_p$ where $\kappa_i=\mp \prod_{j\neq i}(j-i)$.
It follows that
\est{
\frac1{2X}\sum_{|t|\leq X}\eps(t)&=\pm\int_{\Z_p}(-1)^{1+v_p(Q(t))}\pr{\frac{Q(t)_p}{p}}^{1+v_p(Q(t))}\,d\mu_p\\
&=\pm\frac{p-m}{p}\pm\sum_{i=1}^m\int_{\substack{t\in i+p\Z_p,\\}}(-1)^{v_p(t-i)}\pr{\frac{\kappa_i(t-i)_p}{p}}^{v_p(t-i)}\,dt\\
&=\pm\frac{p-m}{p}\pm\sum_{i=1}^m\sum_{\ell=1}^\infty(-1)^\ell\pr{\frac{\kappa_i}{p}}^\ell\int_{\substack{t\in i+p^\ell\Z_p^*\\}}\pr{\frac{(t-i)_p}{p}}^{\ell}\,dt\\
&=\pm\frac{p-m}{p}\pm\sum_{i=1}^m\sum_{\substack{\ell=1}}^\infty (-1)^\ell\pr{\frac{\kappa_i}{p}}^\ell \frac1{p^\ell}\int_{\substack{x\in \Z_p^*\\}}\pr{\frac{x}{p}}^\ell\,dt\\
&=\pm\frac{p-m}{p}\pm\sum_{i=1}^m\sum_{\substack{\ell=1,\\ \ell \text{ even}}}^\infty\frac1{p^\ell}\frac{p-1}{p}=\pm\frac{p-m}{p}\pm m\sum_{\substack{\ell=0}}^\infty\frac{p-1}{p^{3+2\ell}}\\
&=\pm\frac{p-m}{p}\pm m\frac{p-1}{p(p^2-1)}=\pm\Big(1- \frac{m}{p+1}\Big)=\pm\frac{2r|h|}{2rk}=\frac{h}{k},\\
}
and $\Av_\Z(\mathfrak F_{\Z}')\supseteq \Q\cap[0,1]$ as desired.

To show that we also have that
 $\Av_\Z(\mathfrak F_{i, \Z})\supseteq \Q\cap[0,1]$, we proceed as above taking $Q(t)=P(t)$ instead of $Q(t)=(4pt^2+1)P(t)$.
\end{proof}

We now move to the proof of Theorem \ref{rarn2}, about the density of average of the root number over the rational. Given a family of elliptic curve $\GFF$, we  first state a result to compute $\Av_{\Q} (\varepsilon_\GFF)$ as defined by \eqref{averageQ}. As for the averages over the integers, we write the root number as an {\it almost finite} product of local root number, and we use the following result.

\begin{prop}[Helfgott, Proposition 7.8] \label{HH-averageQ} Let $S$ be a finite set of places of $\Q$, including $\infty$. For every $v \in S$, let $g_v: \Q_v \times \Q_v \rightarrow \C$ be a bounded function that is locally constant outside a finite set of lines through the origin. For every $p \not\in S$, let $h_p: \Q_p \times \Q_p \rightarrow \C$ be a function that is locally constant outside a finite set of lines thought the origin, and satisfying $|h_p(x,y)| \leq 1$ for all $x, y \in \Q_p$.
Let $B(x,y) \in \Z[x,y]$ be a non-zero homogeneous polynomial of degree at most 6, and assume that $h_p(x,y)=1$ when $v_p(B(x,y)) \leq 1$. Let
$$W(x,y) = \prod_{v \in S} g_v(x,y) \prod_{p \not\in S} h_p(x,y).$$
Then,
\est{
\Av_{\Z^2, \rm{coprime}} W(x,y) &:=\lim_{N\to\infty} \frac{\sum_{(x,y)\in [-N,N]^2, (x,y)=1} W(x,y)}{\#\{(x,y)\in [-N,N]^2\mid (x,y)=1\}}\\
&=c_\infty \prod_{p \in S}  \frac{1}{1-p^{-2}} \int_{O_p} g_p(x,y) \;dx dy \cdot \prod_{p \not\in S} \frac{1}{1-p^{-2}} \int_{O_p} h_p(x,y) \;dx dy
}
where $O_p = (\Z_p \times \Z_p) \setminus (p \Z_p \times p \Z_p)$, and
$$c_\infty = \lim_{N \rightarrow \infty}\frac{1}{2N^2} \int_{-N}^N \int_{-N}^N g_\infty(x,y) \; dx dy.
$$
\end{prop}

We remark that our version of \cite[Proposition 7.8]{helfgott} is unconditional, as we are assuming that $\deg{B} \leq 6$.

\begin{proof}[Proof of Theorem~\ref{rarn2}]

We first prove that $\Av_\Q(\mathfrak F_{\Q}')$ is dense in $[0,1]$.

For $X\geq2$, let $$m_X:=\prod_{2\leq p\leq X}p \;\;\;\mbox{and} \;\;\;  n_X:=m_X^{2\lfloor f(X) \rfloor +1},$$ where $f(x)$ is any positive function such that $f(X)$ and $\log{X}/f(X)$  tend to infinity. Let $P(t)$ be a polynomial with integer coefficients of even degree $2d>0$ and define
\est{
Q_X(t):=-P(t)(1+n_Xt^2),\qquad a_X(t):=-2^4P(t)(1+n_Xt^2)^2.
}
Notice that by the equations of the invariants~\eqref{invariant_w} for the family $\was_a(t)$,
 $\was_{a_X(t)}(Q_X(t))$ is a non-isotrivial potentially parity-biased family. %over $\Q$.

Now, let $r/s\in\Q$ with $(r,s)=1$, $s>0$. We have the isomorphism of elliptic curves $$\was_{a_X(r/s)}(Q_X(r/s)) \simeq \was_{s^{2d+4}a_X(r/s)}(s^{2d+4}Q_X(r/s)) =\was_{a_X(r,s)}(Q_X(r,s)),$$ where
\est{
P(r,s)&:=s^{2d}P(r/s),\\
Q_X(r,s)&:=s^{2d+4}Q_X(r/s)=-s^2P(r,s)(s^2+n_Xr^2),\\
a_X(r,s)&:=s^{2d+4}a_X(r/s)=-2^4P(r,s)(s^2+n_Xr^2)^2
}
are homogeneous polynomial in $r$ and $s$ which are non-zero for all but finitely many $r/s$. In the following we shall ignore such values as they give a negligible contribution to the average.
Also, we let $\varepsilon(r/s)$ be the root number of the elliptic curve given by $\was_{a_X(r,s)}(Q_X(r,s))$ (or, i.e., by $\was_{a_X(r/s)}(Q_X(r/s))$). Using Theorem \ref{average-periodic}, we have
\begin{eqnarray} \nonumber
\varepsilon(r/s)&\equiv& -  s_{a_X(r,s)}(Q_X(r,s)) \; {\gcd(a_X(r,s)_2,Q_X(r,s))} \times \\
&& \label{rootnumber-formula} \times \prod_{p|g_X(r,s)}(-1)^{1+v_p(Q_X(r,s))}\pr{\frac{Q_X(r,s)_p}{p}}^{1+v_p(Q_X(r,s))}
 \mod 4. \end{eqnarray}
where \est{
g_X(r,s):= \frac{|a_X(r,s)_2|}{\gcd(a_X(r,s)_2,Q_X(r,s))} %%%= (s^2+n_Xr^2)_o/(s^2+n_Xr^2,s_o^2).
}
Now, we have
\est{
\gcd(a_X(r,s)_2,Q_X(r,s)) &=  |P(r,s)_2(s^2+n_Xr^2)_2|(s^2+n_Xr^2,s_2^2)\\
}
and thus
\est{
g_X(r,s):=\frac{|a_X(r,s)_2|}{\gcd(a_X(r,s)_2,Q_X(r,s))}= \frac{(s^2+n_Xr^2)_2}{(s^2+n_Xr^2,s_2^2)}.
}
%since $(r,s)=1$.
Notice that, for $3\leq p\leq X$, we have $v_p(n_X)$ and $v_p(s^2)$ have opposite parities and so $v_p(s^2+n_Xr^2)=\min(v_p(s^2),v_p(n_X))$, since $(r,s)=1$. In particular, we have $p\nmid g_X(r,s)$ for $3\leq p\leq X$. Moreover, for $p>X$ then $p|g_X(r,s)$ if and only if $p \mid s^2 + n_X r^2$ since $p \nmid s$ then.
%%%$-n_X$ is a square modulo $p$. %Also, we assume $X$ is large enough so that $P(r,s)$ and $s^2+n_Xr^2$ do not have common solutions $(r,s)$ mod $p$.
%%%Thus, if $-n_X$ is a square modulo $p$ then
In that case, we have
\est{
v_p(Q_X(r,s))=v_p(s^2+n_Xr^2) + v_p(P(r,s)).
}
Let $E_X$ be the set of primes  $p > X$ such that there exist $r,s$ such thta $p \mid s^2 + n_X r^2$ and $p \mid P(r,s)$. Since  $p\nmid s$ this implies
$P(t)$ and $1+n_xt^2$ has a common solution modulo $p$ and thus $p$ divides the resultant $R_X$ of these two polynomials. We notice that $R_X = O (n_X).$
It follows that
\est{
&\prod_{p|g_X(r,s)}(-1)^{1+v_p(Q_X(r,s))}\pr{\frac{Q_X(r,s)_p}{p}}^{1+v_p(Q_X(r,s))}\\
&\hspace{5em}=\prod_{\substack{p>X, \,p \not \in E_X \\ v_p(s^2+n_Xr^2)\geq2}}(-1)^{1+v_p(s^2+n_Xt^2)}\pr{\frac {Q_X(r,s)_p}p}^{1+v_p(s^2+n_Xr^2)}\\
&\hspace{5em}\quad \times \prod_{\substack{p>X, \, p\in E_X,\\v_p(s^2+n_Xr^2)\geq1}} (-1)^{1+v_p(s^2+n_Xr^2)+v_p(P(r,s))}\pr{\frac {Q_X(r,s)_p}p}^{1+v_p(s^2+n_Xr^2)+v_p(P(r,s))}.
}

We now simplify the first part of \eqref{rootnumber-formula}.
For all but finitely many values of $r,s$ we have
\est{
v_2(Q_X(r,s))&=2v_2(s)+v_2(P(r,s))+v_2(s^2+n_Xt^2)=2v_2(s)+v_2(P(r,s))+\min(2v_2(s),2 \lfloor f(X) \rfloor+1),\\
v_2(a_X(r,s))&=4+v_2(P(r,s))+2v_2(s^2+n_Xt^2))=4+v_2(P(r,s))+2\min(2v_2(s),2 \lfloor f(X) \rfloor+1),\\
}
and thus if $v_2(s)\leq \lfloor f(X) \rfloor$, then
$v_2(a_X(r,s)) =  v_2(Q_X(r,s))+4,$
and thus by Remark \ref{sat_particular_case} after Proposition \ref{sat},
we have $s_{a_X(r,s)}\equiv Q_X(r,s)_2 \mod 4$. Then, for $v_2(s)\leq \lfloor f(X) \rfloor$,
\est{
&s_{a_X(r,s)}(Q_X(r,s)) \, \gcd(a_X(r,s)_2,Q_X(r,s))\\
&\qquad\qquad\equiv -s_2^2P(r,s)_2(s^2+n_Xr^2)_2 \; |P(r,s)_2(s^2+n_Xr^2)_2|(s^2+n_Xr^2,s_2^2)\mod 4\\
&\qquad\qquad\equiv -\sgn (P(r,s))(s^2+n_Xr^2,s_2^2)\mod 4.}
Notice that since $(r,s)=1$ we have
\est{
(s^2+n_Xr^2,s_2^2)=\prod_{3\leq p\leq X}p^{\min (2v_p(s),2\lfloor f(X) \rfloor +1)}\equiv\prod_{\substack{3\leq p\leq X,\\v_p(s)>\lfloor f(X) \rfloor}}p\mod 4.
}
If $v_2(s) \geq \lfloor f(X) \rfloor$, then it also follows from Proposition~\ref{sat} that
\est{
&s_{a_X(r,s)}(Q_X(r,s)) \, \gcd(a_X(r,s)_2,Q_X(r,s))\equiv -\sgn (P(r,s)) \; f_2(r,s) \prod_{\substack{3\leq p\leq X,\\v_p(s)>\lfloor f(X) \rfloor}}p\mod 4,}
where $f_2(r,s)$ is a $2$-locally constant function.

Replacing the above in \eqref{rootnumber-formula}, we have proven that
\est{
\eps(r/s)&= \sgn (P(r,s)) \; f_{2, X} (r,s) \; \prod_{\substack{3\leq p\leq X,\\ v_p(s)>\lfloor f(X) \rfloor }} \leg{-1}{p}\\
&\quad\times \prod_{\substack{p>X, \,p \not \in E_X \\ v_p(s^2+n_Xr^2)\geq2}}(-1)^{1+v_p(s^2+n_Xt^2)}\pr{\frac {Q_X(r,s)_p}p}^{1+v_p(s^2+n_Xr^2)}\\
&\quad\times \prod_{\substack{p>X, \, p\in E_X,\\v_p(s^2+n_Xr^2)\geq1}} (-1)^{1+v_p(s^2+n_Xr^2)+v_p(P(r,s))}\pr{\frac {Q_X(r,s)_p}p}^{1+v_p(s^2+n_Xr^2)+v_p(P(r,s))},
%%%\prod_{\substack{p>X,\\ v_p(s^2+n_Xt^2)\geq2}}(-1)^{1+v_p(s^2+n_Xt^2)}\pr{\frac {Q(r,s)_p}p}^{1+v_p(s^2+n_Xt^2)}\mod 4.
}
where $f_{2, X}(r,s)=1$ if $v_2(s) \leq \lfloor f(X) \rfloor.$ Then, for $3\leq p\leq X$, we write
$$f_{p, X}(r,s)=\begin{cases} \leg{-1}p & \mbox{if $v_p(s)>\lfloor f(X) \rfloor$} \\ 1 & \mbox{otherwise.} \end{cases}$$
For $p>X$, $p\notin E_X$ we write
$$h_{p, X}(r,s)=
\begin{cases} 1 & \mbox{if $v_p(s^2+n_Xr^2)<2$} \\
(-1)^{1+v_p(s^2+n_Xt^2)}\pr{\frac {Q_X(r,s)_p}p}^{1+v_p(s^2+n_Xr^2)} & \mbox{otherwise.} \end{cases}$$
Finally, for $p>X,$ $p\in E_X$ we write
$$g_{p, X}(r,s)= \begin{cases} 1 &  \mbox{if $v_p(s^2+n_Xr^2)=0$} \\
 (-1)^{1+v_p(s^2+n_Xr^2)+v_p(P(r,s))}\pr{\frac {Q_X(r,s)_p}p}^{1+v_p(s^2+n_Xr^2)+v_p(P(r,s))} & \mbox{otherwise.} \end{cases}$$
Thus, we have
\est{
\eps(r/s)&= \sgn (P(r,s)) \prod_{2\leq p\leq X} f_{p,X}(r,s)\; \prod_{\substack{p > X \\ p \in E_X}} g_{p,X}(r,s)\; \prod_{\substack{p > X \\ p \not\in E_X}} h_{p,X}(r,s).
}
The conditions for Proposition~\ref{HH-averageQ}  with $S=X\cup E_X$ (which is a finite set) and $B(x,y)=x^2+n_Xy^2$ are satisfied and thus
\begin{eqnarray} \nonumber \hspace{-0.7cm}
\Av_\Q (\eps) &=&
\Av_{\Z^2, \rm{coprime} }(\epsilon_{a_X(r,s)}(Q_X(r,s))) \\ \label{formula-averageX}
&=& c_\infty(P)  \prod_{\substack{2 \leq p \leq X}} \frac{1}{1-p^{-2}} \int_{O_p} f_{p,X}(r,s) dr ds \times \\
\nonumber
&&\times \prod_{\substack{p > X \\ p \in E_X}} \frac{1}{1-p^{-2}} \int_{O_p} g_{p,X}(r,s) dr ds \;
\prod_{\substack{p > X \\ p \not\in E_X}}\frac{1}{1-p^{-2}} \int_{O_p} h_{p,X}(r,s) dr ds ,
\end{eqnarray}
where \est{
c_{\infty}(P) &=\lim_{N\rightarrow\infty}\frac1{4N^2}\int_{-N}^N\int_{-N}^N\sgn(P(x,y))\,dxdy.
}
We will now show that the three products over $p$ all contributes $1+o(1)$ as $X\to\infty$.
First, we notice that for $p \leq X$, we have $f_{p,X}(r,s) = 1$ if $v_p(s) \leq \lfloor f(X) \rfloor$ and thus
\est{
\int_{O_p} f_{p, X}(r,s) \;dr ds %%%\int_{p^{X+1} \Z_p} \int_{\Z_p^*} f_{p, X}(r,s) dr ds +
%%%\mu \left( \left\{ (r,s) \in O_p : v_p(s) \leq X \right\} \right) \\
&= \int_{O_p} 1 \; dr d s + O \left( \mu \left( \left\{ (r,s) \in O_p : v_p(s) \geq \lfloor f(X) \rfloor +1 \right\} \right)  \right) \\
&= (1-p^{-2}) + O \left( p^{-\lfloor f(X) \rfloor-1} \right).\\
}
Also, for $p\notin E_X$, we have $h_{p,X}(r,s) = 1$ if $v_p(s^2+n_Xr^2)<2$ and so
\est{
\int_{O_p} h_{p, X}(r,s) \;dr ds &= \int_{O_p} 1 \; dr d s + O \left( \mu \left( \left\{ (r,s) \in O_p : v_p(s^2 + N_X r^2) \geq 2 \right\} \right)  \right) \\
&= (1-p^{-2}) + O \left( \frac1{p^2}\right),\\
}
and in the same way for $p\in E_X$ we obtain
\es{\label{int_gp}
\int_{O_p} g_{p, X}(r,s) \;dr ds &= (1-p^{-2}) + O \left( \frac1{p}\right).\\
}
With the above formulas it's then clear that the first two products over $p$ in~\eqref{formula-averageX} are $1+o(1)$ as $X\to\infty$. We now show that the same holds for the last product. We let $e(X)$ be the cardinality of $E_X$, which is the set of the prime divisors $p > X$ of an integer $R_X\ll n_X$. Then we have
\begin{eqnarray*}
X^{e(X)} <\prod_{p>X,\, p|R_X}p \ll n_X = \prod_{p \leq X} p^{2 \lfloor f(X) \rfloor + 1}\ll e^{O(f(X)X)}
\end{eqnarray*}
and thus $e(X) \ll \frac{f(X) X}{\log{X}}$.
Then,
\begin{eqnarray*}
\log{ \prod_{\substack{p > X \\ p \in E_X}} \left( 1 + O \left( 1/p \right) \right) } \ll  \sum_{\substack{p > X\\ p \in E_X}}
 \log(1+O(1/p))  \ll  \sum_{\substack{p > X\\ p \in E_X}}\frac1p\ll \frac{e(X)}{X}\ll \frac{f(X)}{\log{X}} =o(1)\end{eqnarray*}
as $X \rightarrow \infty$, by hypothesis. Thus, by~\eqref{int_gp}, we have that also the last product over $p$ in~\eqref{formula-averageX} is $1+o(1)$ and so
\es{\label{formula-averageX_2}
\lim_{X\to\infty} \Av_\Q (\eps)=c_{\infty}(P).
}
Finally, we compute that
\est{
c_{\infty}(P)  &=\lim_{N\rightarrow\infty}\frac1{4N^2}\int_{-N}^N\int_{-N}^N\sgn(P(x,y))\,dxdy\\
&=\frac1{4}\int_{-1}^1\int_{-1}^1\sgn(P(x/y))\,dxdy=\frac1{2}\int_{0}^1\int_{-1/y}^{1/y}\sgn(P(x))y\,dxdy\\
&=\frac1{2}\int_{-\infty}^\infty\int_{0}^{\min(1,1/|x|)}\sgn(P(x))y\,dydx\\
&=\frac1{4}\int_{-1}^1\sgn(P(x))\,dx+\frac1{4}\int_{1}^\infty \frac{\sgn(P(x))+\sgn(P(-x))}{x^2}\,dx.
}
Thus, replacing in \eqref{formula-averageX_2},  we get
\est{
\lim_{X\rightarrow\infty}\Av_{\Q}(\eps)& = \frac1{4}\int_{-1}^1\sgn(P(x))\,dx+\frac1{4}\int_{1}^\infty \frac{\sgn(P(x))+\sgn(P(-x))}{x^2}\,dx.
}
To conclude, it is enough to show that the set of values taken by $c_\infty(P)$ as $P$ varies among polynomials of even degree in $\Z[t]$ is dense in $[0,1]$. If $|h/k|\leq1$, then taking $$P(x)=k^2(x^2-(1-h/k)^2),$$ one obtains $c_\infty(P)=h/k$ and so
$\Av_\Q(\mathfrak F_{\Q}')$ is dense in $[0,1]$.
\medskip

Next, we show that $\Av_\Q(\mathfrak F_{i, \Q})\supseteq \Q\cap[0,1]$.
Given a polynomial $P(t)\in\Z[t]$ of even degree $2d>0$ we take
\est{
&Q(t):=-P(t),\qquad a(t):=-2^4P(t),\\
&Q(r,s):=s^{2d}Q(r/s),\qquad a(r,s):=s^{2d}a(r/s),\qquad P(r,s):=s^{2d}P(r/s).
}
for $(r,s)=1$, $s>0$. Then, $\was_{a(t)}(Q(t))$ gives a potentially parity-biased elliptic surface, which is isotrivial.
Also, as before we call $\eps(r/s)$ the root number of the specialization $\was_{a(r/s)}(Q(t))\simeq \was_{a(r,s)}(Q(r,s))$.
Then, assume $t=r/s$ is not a zero of $P(t)$. We have $$\gcd(a(r,s)_2,Q(r,s))=|a(r,s)_2|=|P(r,s)_2|.$$
Also, $v_2(a(r,s))=v_2(Q(r,s))+4$. Thus by Remark \ref{sat_particular_case} after Proposition \ref{sat}, we have $$s_{a(r,s)}(Q(r,s)) \equiv Q(r,s)_2\equiv -P(r,s)_2\mod 4.$$ It follows that
\est{
\varepsilon_{a(s/t)}(Q(s/t))&=\varepsilon_{a(s,t)}(Q(s,t))\equiv P(r,s)_2|P(r,s)_2|\equiv \sgn(P(r,s))\mod 4.\\
}
Thus, using Proposition~\ref{HH-averageQ}, we have
\est{
\Av_{\Q}(\varepsilon_{a(t)}(Q(t)))&=\lim_{N\rightarrow\infty}\frac1{4N^2}\int_{-N}^N\int_{-N}^N\sgn(P(x,y))\,dxdy\\
&=\frac1{4}\int_{-1}^1\sgn(P(x))\,dx+\frac1{4}\int_{1}^\infty \frac{\sgn(P(x))+\sgn(P(-x))}{x^2}\,dx.
}
Choosing $P(x)$ appropriately as above, we obtain $\Av_\Q(\mathfrak F_{i,\Q})\supseteq \Q\cap[0,1]$.
\end{proof}

In order to prove Theorem~\ref{rarn3} we need a few Lemmas.

\begin{lemma}\label{flfst}
Let $\phi,\eta_1,\eta_2:\Z\to\C_{\neq0}$ be periodic functions of period $n_1,n_1$ and $n_2$ respectively with $(n_1,n_2)=1$. Assume there exists $\delta>0$ such that
\est{
|\{t\in\N\mid t\leq N, \phi(t)\neq \eta_1(t)\eta_2(t)\}|< \delta N
}
for all large enough $N$. Then there exists $\xi\in \C_{\neq0}$ such that
\est{
|\{t\in\N\mid t\leq N, \phi(t)\neq \xi \cdot \eta_1(t)\}|< 2\delta N
}
for all large enough $N$.
\end{lemma}
\begin{proof}
There exists a residue class $i$ modulo $n_1$ such that for all large enough $N$
\est{
|\{t\in\N\mid t\leq N, \eta_2(i+n_1t)\neq\xi\}|< \delta N
}
where $\xi=\frac{\phi(i+n_1t)}{ \eta_1(i+n_1t)}$. Since $(n_1,n_2)=1$ then the above inequality is equivalent to
\est{
|\{t\in\N\mid t\leq N, \eta_2(t)\neq\xi\}|< \delta N
}
for $N$ large enough and thus the result follows.
\end{proof}
\begin{lemma}\label{mlftab}
Let $\phi:\Z\to\C_{\neq0}$ be a periodic function of period $\ell$ and let $\eta(t)=\prod_{i=1}^{k}h_{p_i}(t)$ with $h_{p_i}:\Z\to\C_{\neq0}$ periodic modulo $p_i^{r_i}$ and $p_1,\dots,p_k$ distinct primes. Assume we have
\est{
|\{t\in\N\mid t\leq N, \phi(t)\neq \eta(t)\}|<  N/4\ell^2
}
for all large enough $N$. Then for each $p|\ell$ there exist $\rho\in\C_{\neq0}$ and a function $h^*_{p}(t)$ periodic modulo $p^{v_p(\ell)}$ such that
$\phi(t)=\rho \prod_{p|\ell}h^*_{p}(t)$ for all $t$.
\end{lemma}
\begin{proof}
Let $n:=p_1^{r_1}\cdots p_k^{r_k}$. Increasing $n$ if necessary, we can assume $\ell =p_1^{s_1}\cdots p_u^{s_u}$ where $s_i\leq r_i$, $u\leq k$ and $p_1,\dots,p_u$ are distinct primes.
We prove the result by induction over the number  $u$ of distinct primes of $\ell$. If $u=0$, then the result is trivial. Now, let $u\geq1$ and assume the result is true if $\ell$ has $u-1$ distinct prime factors.

Let $\eta_1(t)=\prod_{j=2}^{u}h_{p_j}(t)$ and $\eta_2(t)=h_{p_1}(t)\prod_{j=u+1}^{k}h_{p_j}(t)$ so that $\eta(t)=\eta_1(t)\eta_2(t)$. Then for all $i=1,\dots,p_1^{s_1}$ we have that $\phi(i+p_1^{s_1}t)$ and $\eta_1(i+p_1^{s_1}t)$ are periodic modulo $n_1:=p_2^{r_2}\cdots p_u^{r_u}$ whereas $\eta_2(i+p_1^{s_1}t)$ is periodic modulo $n_2:=n/n_1p_1^{s_1}$. Notice that $(n_1,n_2)=1$. Moreover, we have
\est{
|\{t\in\N\mid t\leq N, \phi(i+p_1^{s_1}t)\neq \eta_1(i+p_1^{s_1}t)\eta_2(i+p_1^{s_1}t)\}|< p_1^{s_1} \cdot { N}/{4\ell^2}%\leq \frac{N}{4p_1^{s_1}(\ell/p_1^{s_1})^2}
}
for all large enough $N$. Then, by Lemma~\ref{flfst}, there exists $\xi_i\in\C_{\neq0}$ such that
\est{
|\{t\in\N\mid t\leq N, \phi(i+p_1^{s_1}t)\neq \xi_i \eta_1(i+p_1^{s_1}t)\}|< 2 p_1^{s_1} { N}/{4\ell^2}
}
for all large enough $N$. Equivalently, writing $h^*_{p_1}(t):=\xi_i$ if $t\equiv i\mod{ p_1^{s_1}}$, we have
\es{\label{fe4w}
|\{t\in\N\mid t\leq N, \frac{\phi(t)}{h^*_{p_1}(t)}\neq  \eta_1(t)\}|< p_1^{s_1} { N}/{2\ell^2}
}
for large enough $N$.
Also, for all $j=1,\dots, \ell/p_1^{s_1}$ we have that $\frac{\phi(j+\ell/p_1^{s_1}t)}{h^*_{p_1}(j+\ell/p_1^{s_1}t)}$ is periodic modulo $p_1^{s_1}$ and $\eta_1(j+\ell/p_1^{s_1}t)$ is periodic modulo $p_1^{s_1}n_1/\ell$
and so, since $(p_1^{s_1},p_1^{s_1}n_1/\ell)=1$, by Lemma~\ref{flfst} we have that there exists $\psi_j\in\C$ such that
\est{
|\{t\in\N\mid t\leq N, \frac{\phi(j+\ell/p_1^{s_1}t)}{h^*_{p_1}(j+\ell/p_1^{s_1}t)}\neq  \psi_j\}|<  p_1^{s_1} { N}/{\ell^2}
}
for large enough $N$ (notice that since $\ell/p_1^{s_1}$ is coprime with the period $p^{s_1}$ having $t$ or $j+\ell/p_1^{s_1}t$ doesn't change the estimate on the right for $N$ large enough). Thus, since $p_1^{s_1} { N}/{\ell^2}\leq N/p_1^{s_1}$ and $\frac{\phi(j+\ell/p_1^{s_1}t)}{h^*_{p_1}(j+\ell/p_1^{s_1}t)}$ is periodic modulo $p_1^{s_1}$, then the set on the left has to be empty which means that $\frac{\phi(t)}{h^*_{p_1}(t)}$ is periodic modulo $\ell/p_1^{s_1}$.
Finally, notice that $p_1^{s_1} { N}/{2\ell^2}\leq N/4(\ell/p_1^{s_1})^{2}$ and so by~\eqref{fe4w} we can apply the inductive hypothesis to $\frac{\phi(t)}{h^*_{p_1}(t)}$ and the Lemma follows.
\end{proof}

\begin{corol}\label{mcfar}
Assume Conjectures \ref{chowla} and \ref{square-free-1}. Let $\GFF$ a family of elliptic curves as defined by~\eqref{eq_family} and such that $\varepsilon_\GFF(t)$, the root number of the specializations $\GFF(t)$, is a periodic function of $t$ with period $\ell$ with $o(1)$ exceptions for $|t|\leq T$ as $T\to\infty$. Then, up to $o(1)$ exceptions, we have
\est{
\varepsilon_\GFF(t)=\rho\prod_{p|\ell}h_{p}(t)
}
where $h_{p}:\Z\to\{\pm1\}$ is periodic modulo $p^{v_p(\ell)}$ and $\rho\in\{\pm1\}$.
\end{corol}
\begin{proof}
By Theorem~\ref{Helfgott}, the root number can be written as
\est{
\varepsilon_{\GFF}(t)= \sign(g_\infty(t)) \, \lambda(M_\GFF(t)) \, \prod_{p \; \mathrm{prime}}g_p(t) %%%\, \prod_{\substack{p^2|B(t),\\ p\notin S}}g_p(t),
}
up to finitely many exceptions, where $g_p(t)=1$ if $p\notin S$ and $p^2\nmid B_{{\GFF}}(t)$.
Under Conjectures~\ref{chowla} and~\ref{square-free-1}, by Theorem~\ref{Helfgott}, we have that $\varepsilon_{\GFF}(t)$ has average $0$ as $t$ varies among any fixed arithmetic progression unless $M_\GFF(t)=1$. Thus, since  $\varepsilon_{\GFF}(t)$ is periodic up to $o(1)$ exceptions, we must have $M_\GFF(t)=1$. Similarly, the periodicity of $\varepsilon_{\GFF}(t)$ implies that
$\sign(g_\infty(t))$ is constant, except for a finite number of values of $t$.

Under Conjecture \ref{square-free-1}, for any fixed $\eps>0$ there exists $N_0>0$ such that
\est{
|\{1\leq n \leq N \mid p^2| B_{\GFF}(n)\Rightarrow p\leq N_0\}|\geq (1-\eps)N.
}
for all $N\geq1$ (see e.g. the proof of Proposition 7.7 in~\cite{helfgott}).
Now, let $S_0=S\cup\{p\leq N_0\}$. Since $g_p(t)$ is locally constant outside a finite set of points for all $p$, then there exist an integer $k\in\N$ (depending on $\eps$) and functions $g_p^*:\Z\to \{\pm1\}$ which are periodic of period $p^k$ such that
\est{
\prod_{p\in S_0}g_p(t)=\prod_{p\in S_0}g^*_p(t)
}
for $t$ in all but $c$ residue classes $r_1,\dots,r_c$ modulo $\ell:=\prod_pp^k$ with $c\leq\eps \ell$. Indeed, if $g_p(t)$ is locally constant on $\Z_p\setminus\{x_1,\dots,x_{c_p}\}$, then $g_p(i+tp^s)$ is $p$-locally constant (in $t\in\Z_p$) for all $1\leq i\leq p^s$ with $i\not\equiv x_1,\dots, x_{c_p}\mod {p^s}$. In particular, since $\Z_p$ is compact then for all $i\not\equiv x_1,\dots, x_{c_p}\mod {p^s}$ one has that $g_p(i+tp^s)$ is periodic modulo $p^{s_i}$ for some $s_i$. Thus, writing $g^*_p(t):=g_p(t)$ if $t\not\equiv x_1,\dots, x_{c_p}\mod {p^s}$ and $g^*_p(t):=1$ otherwise, we have that $g^*_p(t)$ is periodic modulo $p^{k_p}$ where $k_p:=s+\max_{i}s_i$. Taking $s$ large enough so that $p^{-m}\leq \eps/r$, we then have that $g^*_p(t)$ coincides with $g_p(t)$ for all but $\eps p^{k_p}$ congruence classes modulo $p^{k_p}$. Proceeding in the same way for all $p\in S_0$ and taking $k=\max_pk_p$ we obtain the claimed result.

Now, we define
\est{
I_N=\{1\leq n \leq N \mid p^2| B_{\GFF}(n)\Rightarrow p\leq N_0,\ n\not\equiv r_j\mod m\text{ for } j=1,\dots, c\},
}
so that
\est{
|I(N)|\geq(1-3\eps)N
}
for $N$ large enough. For all $t\in I_N$ we have
\est{
\varepsilon_{\GFF}(t)= \prod_{p\in S_0}g_p^*(t)
}
and so, taking $3\eps<1/4\ell^2$ and using Lemma~\ref{mlftab}, we get  the claimed result.
\end{proof}

Corollary~\ref{mcfar} easily gives that $\Av(\mathfrak F_{\mathrm p,\Z})$ is contained in the set on the right of~\eqref{eqfar} (see the end of the proof of Theorem~\ref{rarn3} below for the details). We now give two Lemmas which are needed in the opposite direction. In order to construct subfamilies of $\PFF_{a}(t)$ with periodic root numbers with specified averages, we are led to the problem of finding polynomials in $\Z/p^{\ell}\Z[t]$ with a prescribed number of zeros $m$. This is not always possible for all choices (for example there is no polynomial in $\Z/8\Z[t]$ with exactly $7$ zeros), but we can always find an $r\geq\ell$ and a polynomial in $\Z/p^{r}\Z[t]$ such that a portion of exactly $m/p^\ell$ residue classes modulo $p^{r}\Z$ are zeros.

\begin{lemma}\label{adaw}
Let $p$ be a prime and let $\ell,u\geq1$ and let $m$ be such that $0\leq m\leq p^{\ell }$. Then, there exists $r\geq0$ arbitrary large such that there exists a polynomial $P(t)\in\Z[t]$ with exactly $mp^{r}$ zeros $\mod {p^{r+\ell }}$ and with $v_p(P(t))\leq {r+\ell -u}$ whenever $t$ is not one of such zeros.
\end{lemma}
\begin{proof}
We order the numbers $0\leq j<p^\ell $ in the following way: given $0\leq h,k<p^\ell$ we say $h>^* k$ iff writing $h=a_0+a_1p+\cdots+a_{\ell-1}p^{\ell-1}$, $k=b_0+b_1p+\cdots+b_{\ell-1}p^{\ell-1}$, then $a_d< b_d$ where $d=v_p(k-h)$. Then, we order the numbers $0\leq j<p^\ell $ as  $b_0 >^*b_1 >^*\cdots >^*b_{p^\ell -1}$, with
\begin{eqnarray*} &&b_0=0 >^* b_1=p^{\ell -1} >^* 2p^{\ell -1} >^* \dots >^* (p-1)p^{\ell-1} >^* p^{\ell-2}
\geq p^{\ell-2} + p^{\ell-1}>^* \dots\\
&&>^* p^{\ell-2} + (p-1) p^{\ell-1} >^* 2 p^{\ell-2} >^* 2 p^{\ell-2} + p^{\ell-1}>^* \dots >^* p^{\ell}-1=b_{p^\ell-1}.
\end{eqnarray*}
Notice that with this choice for all $0\leq i<p^\ell$ we have that $v_p(b_j-b_i)$ is decreasing in $j$ with $i<j\leq p^\ell$. In particular, for any sequence of non-negative real numbers  $\kappa_{0},\dots, \kappa_{p^{\ell}-1}$ we have that $\sum_{i=0}^{b}\kappa_iv_{p}(b_j-b_i)$ is also decreasing in $j$ for $0\leq b<j< p^{\ell}$.

Next, for any positive integer $s$, define recursively the sequence $c_j$ for $0\leq j< m$ in the following way:
%\est{
%&c_0=s/\ell, \\
%& c_1=s/\ell-c_0v_p(b_1-b_0)/\ell=s/\ell(1-v_p(b_1-b_0)/\ell)\\
%& c_2=s/\ell-c_0v_p(b_2-b_0)/\ell-c_1v_p(b_2-b_1)/\ell\\
%& \quad=s/\ell(1-v_p(b_2-b_0)/\ell-v_p(b_2-b_1)/\ell
%+v_p(b_1-b_0)v_p(b_2-b_1)/\ell^2)\\
%& c_3=s/\ell-c_0v_p(b_3-b_0)/\ell-c_1v_p(b_3-b_1)/\ell-c_2v_p(b_3-b_2)/\ell\\
%& \quad=s/\ell(1-v_p(b_3-b_0)/\ell-v_p(b_3-b_1)/\ell-v_p(b_3-b_2)/\ell\\
%&\qquad+v_p(b_1-b_0)v_p(b_3-b_1)/\ell^2+v_p(b_2-b_0)v_p(b_3-b_2)/\ell^2+v_p(b_2-b_1)v_p(b_3-b_2)/\ell^2)\\
%&\qquad-v_p(b_1-b_0)v_p(b_2-b_1)v_p(b_3-b_2)/\ell^3)\\
%}
\est{
&c_0=s/\ell, \qquad \ell c_j=s-\sum_{i=0}^{j-1}c_iv_p(b_j-b_i).
}
We have that $c_j>0$ for all $j$. Indeed, this is obvious for $j\leq1$ and if we assume by induction that it is true for $j\leq k$ with $k<m-1$, then
\est{
\ell (c_k-c_{k+1})&=\sum_{i=0}^{k}c_iv_p(b_{k+1}-b_i)-\sum_{i=0}^{k-1}c_iv_p(b_k-b_i)\\
&=c_k v_p(b_{k+1}-b_k)+ \sum_{i=0}^{k-1}c_iv_p(b_{k+1}-b_i)-\sum_{i=0}^{k-1}c_iv_p(b_k-b_i)\\
&\leq c_k v_p(b_{k+1}-b_k).
}
Thus,
\est{
c_{k+1}\geq c_k(1-v_p(b_{k+1}-b_k)/\ell )\geq c_k/\ell>0.
}
Moreover, it's clear that $c_i\ell^{m}/s$ is an integer, so that taking $s$ to be any multiple of $\ell^mu$ we have that $c_i$ is an integer greater than or equal to $u$.

Now, let $P(t):=\prod_{0\leq i< m}(t-b_i)^{c_i}$.
For $t\equiv b_j\mod {p^{\ell}}$ with $j\geq m$, we have
\est{
v_p(P(t))=\sum_{i=0}^{m-1} c_i v_p(b_j-b_i)&=c_{m-1}v_p(b_j-b_{m-1})+\sum_{i=0}^{m-2} c_i v_p(b_j-b_i)\\
&\leq c_{m-1}v_p(b_j-b_{m-1})+\sum_{i=0}^{m-2} c_i v_p(b_{m-1}-b_i)\\
&= s-c_{m-1}(\ell-v_p(b_j-b_{m-1}))\leq s-c_{m-1}\leq s-u,
}
since $b_j\not\equiv b_{m-1}\mod{p^{\ell}}$.
Finally, if $t\equiv b_j\mod {p^{\ell}}$ with $0\leq j\leq m-1$, then
\est{
v_p(P(t))\geq \sum_{i=0}^{j-1} c_i v_p(b_j-b_i)+c_j\ell&=s.
}
Thus, the Lemma follows with $r=s-\ell$.
\end{proof}
\begin{lemma}\label{adaw1}
Let $0\neq |h/k|<1$ with $(h,k)=1$ and let $p_1,\dots p_g$ be the (distinct) prime factors of $k$. Then there exists $d_1,\dots,d_g\in\Z$, $r_1,\dots,r_g\in\N$ such that $-1<d_i/p_i^{r_i}<1$ for $i=1,\dots, g$ and $h/k=\prod_{i=1}^g {d_i}/{p_i^{r_i}}$. Moreover, we can choose $d_1,\dots,d_g$ so that $d_1\cdots d_g | (hk)^{\infty}$.
\end{lemma}
\begin{proof}
We prove the Lemma by induction on the number of distinct prime factors of $k$, the result being obvious if $k$ has only one prime factor. Thus, assume that $k$ has $g+1$ prime factors and that the Lemma is true whenever $k$ has $g$ prime factors with $g\geq1$. Let $p,q$ be two distinct prime factors of $k$. Since $\log p$ and $\log q$ are linearly independent we can find arbitrarily large $u,v\in\N$ such that $|h/k|<q^v/p^u<1$. In particular, writing $h'=h\cdot p^u/(p^u,k)$, $k'=k\cdot q^v/(p^u,k)$ then we have $\frac{h}{k}=\frac{h'}{k'}\frac{q^v}{p^u}$ and, if $u$ is large enough, $k'$ has $g$ prime factors and $0\neq |h'/k'|<1$ and so the result follows by the inductive hypothesis.
\end{proof}

\begin{proof}[Proof of Theorem~\ref{rarn3}]
Let $(h,k)=1$ with $k\geq1$, $h$ odd and, if $k$ even then $|h/k|\leq \pr{2^{v_2(k)}-1}/{2^{v_2(k)}}$. We now construct a subfamily of $\PFF_a(t)$ with average root number $ h/k$. As in the proof of Theorem~\ref{rarn2} we can assume $h/k\neq0,\pm1$.

For simplicity we assume $k$ is divisible by $2$, but it's not a power of $2$. The same proof works with obvious modifications also without these conditions. Then, we can write  $h/k$ as $$\frac{h}{k}=\frac{2^{v_2(k)}-1}{2^{v_2(k)}} \frac{h'}{k'}$$
with $-1< h'/k' < 1$, and $h',k'$ odd.
We use Lemma~\ref{adaw1} to write $h'/k'$ as ${h'}/{k'}=\prod_{i=1}^g(d_i/p_i^{u_i})$ where $p_1,\dots,p_g$ are the (distinct) prime factors of $k'$, $u_1,\dots,u_g$ are positive integers and $-p^{u_i}< d_i< p^{u_i}$ with $d_i$ odd. In particular
\est{%\label{saf}
\frac {h}{k}=\frac{d_0}{2^{u_0}} \; \frac{d_1}{p_1^{u_1}} \cdots \frac{d_g}{p_g^{u_g}},
}
where $u_0=v_2(k)+1$, and $d_0=2^{u_0}-2.$ We will also use the notation $p_0=2$.

For $i=1,\dots g$,
let $m_i$ be such that $d_i=2m_i-p^{u_i}$, so that $0<m_i<p^{u_i}$, and
let
$$m_i':= \begin{cases}
m_i & \mbox{if $p_i\equiv3\mod 4$}, \\
p_i^{u_i}-m_i & \mbox{if $p_i \equiv 1 \mod 4.$}
\end{cases}
$$
 By Lemma~\ref{adaw},  there exist $r_i\in\N$ and a polynomial $Q_i$ such that the set $Z_i$ of zeros of $Q_i(t)$ modulo ${p_i^{r_i+u_i}}$ has cardinality $m_i'p_i^{r_i}$, and
 $v_{p_i}(Q_i(t)) \leq r_i + u_i - 1$ if $t \not\in Z_i$.
Then, let
\est{
B_{i}(t)=p_iQ_i(t)^2 - p_i^{2r_i+2u_i}
}
Notice that if $t\mod{p_i^{r_i+u_i}}$ is not in $Z_i$, then $v_{p_i}(B_{i}(t)))= 1+2v_p(Q_{i}(t))\leq 2r_i+2u_i-1$ is odd and
\est{
\gcd(p_i^{2r_i+2u_i+1},B_{i}(t))=p_i^{1+2v_p(Q_{i}(t))}\equiv p_i \mod 4.
}
Instead, if $t\mod{p_i^{r_i+u_i}}$ is in $Z_i$, then $v_{p_i}(B_{i}(t)))=2r_i+2u_i$ and $B_{i}(t)_{p_i}\equiv -1\mod{p_i}$ so that $\pr{\frac{B_{i}(t)_{p_i}}{p_i}}=\pr{\frac{-1}{p_i}}\equiv p_i\mod 4$. Also, in this case
\est{
\gcd(p_i^{2r_i+2u_i+1},B_{i}(t))=p_i^{2r_i+2u_i}\equiv 1 \mod 4.
}
Thus, summarizing both cases
\est{
(p_i^{2r_i+2u_i+1},B_{i}(t))(-1)^{1+v_{p_i}(B_{i}(t))}\pr{\frac{B_{i}(t)_{p_i}}{p_i}}^{1+v_{p_i}(B_{i}(t))}
&\equiv
\begin{cases}
-p_i\mod 4&\text{if $t\mod p\in Z_i$}\\
p_i\mod 4&\text{if $t\mod p\notin Z_i$}\\
\end{cases}\\
&\equiv
\begin{cases}
1\mod 4&\text{if $t\mod p\in S_i$}\\
-1\mod 4&\text{if $t\mod p\notin S_i$}\\
\end{cases}
}
where $S_i$ is equal to $Z_i$ if $p_i\equiv3\mod 4$, and it is equal to its complement (in $\Z/p_i^{r_i+u_i}\Z$) if $p_i\equiv1\mod 4$. Notice that in both cases we have that $|S_i|=m_ip_i^{r_i}$.

Next, we use Lemma~\ref{adaw} to find $r_0\in\N$ and a polynomial $Q_0$ such that the set $Z_0$ of zeros of $Q_0(t)$ modulo ${2^{r_0+u_0}}$ has cardinality $(2^{u_0}-1)2^{r_0}$ and $v_{2}(Q_0(t))\leq {r_0+u_0-2}$ whenever $t\mod {2^{r_0+u_0}}$ is not in $Z_0$.
Then, we define
\est{
B_0(t)=2Q_{0}(t)^2- 2^{2r_0+2u_0-1}.
}
If $t\mod {2^{r_0+u_0}}\notin Z_0$, then $v_2(B_{0}(t))=1+2v_2(Q_{0}(t))\leq 2r_0+2u_0-3$ is odd and $B_{0}(t)_2\equiv(Q_{0}(t)_2)^2\equiv 1\mod 4$. If $t\mod {2^{r_0+u_0}}\in Z_0$, then $v_2(B_{0}(t))=2r_0+2u_0-1$ and $B_0(t)_2\equiv -1\mod4$. By Remark \ref{sat_particular_case} after Proposition \ref{sat}, in both cases we have $s_{2^{2r_0+2u_0+3 }}(B_0(t))\equiv B_0(t)_2\mod4$ and so
\est{
s_{2^{2r_0+2u_0+3 }}(B_0(t))\equiv
\begin{cases}
1\mod 4&\text{if $t\mod p\in S_0$},\\
-1\mod 4&\text{if $t\mod p\notin S_0$},\\
\end{cases}
}
where $S_0$ is the complement of $Z_0$ in $\Z/2^{r_0+u_0}\Z$, so that $|S_0|=2^{r_0}$.

We can now define $a=4\prod_{i=0}^g p_i^{2r_i+2u_i+1}$ and
\est{
Q(t)=\sum_{i=0}^g(p_0\cdots p_g/p_i)^rx_{p_i}^rB_{i}(t),
}
where $x_{p_i}$ is the inverse of $p_0\cdots p_g/p_i$ modulo $p_i$ and where $r=2\max_{i=0,\dots g}(u_i+r_i+1)$. It follows that the sign $\eps(t)$ of the elliptic curve $\was_{a}(Q(t))$ is
\est{
\varepsilon(t)&\equiv -s_{a}(B_0(t))\prod_{i=1}^{g}\gcd(p_i^{2u_i+2r_i+1},B_i(t))(-1)^{1+v_{p_i}(B_i(t))}\pr{\frac{B_i(t)_{p_i}}{p_i}}^{1+v_{p_i}(B_i(t))}\mod 4\\
&\equiv -\prod_{i=0}^gh_i(t)\mod 4\\
}
where $h_i(t)=1$ if $t\mod {p_{i}^{r_i+u_i}}$ is in $S_i$ and it is equal to $h_i(t)=-1$ otherwise.
Thus, by the Chinese remainder theorem
\begin{eqnarray*}
\Av_\Z(\varepsilon)&=&-\prod_{i=0}^g\Big(2\frac{|S_i|}{p_i^{u_i+r_i}}-1\Big)=- \left( \frac{2^{r_0+1}}{2^{r_0+u_0}} - 1\right)
\prod_{i=1}^g\Big(2\frac{m_ip_i^{r_i}}{p_i^{u_i+r_i}}-1\Big) \\ &=&  \frac{2^{u_0}-2}{2^{u_0}} \; \prod_{i=1}^g\frac{d_i}{p_i^{u_i}}=\frac hk,
\end{eqnarray*}
as desired.

For the converse, let's assume Conjecture \ref{chowla} and Conjecture \ref{square-free-1}, and prove that the equality holds in~\eqref{eqfar}. By Corollary~\ref{mcfar}, we have that if the root number $\varepsilon_{\GFF}$ of a family of elliptic curves $\GFF$ is periodic modulo $\ell$ up to $o(1)$ exceptions then
\est{
\varepsilon_\GFF(t)=\rho\prod_{p|\ell}h_{p}(t)
}
up to $o(1)$ exceptions, where $h_{p}:\Z\to\{\pm1\}$ is periodic modulo $p^{v_p(\ell)}$ and $\rho\in\{\pm1\}$. Thus, by the Chinese remainder theorem we have
\est{
\Av_\Z(\varepsilon_F)=\prod_{p|\ell}\Big(\sum_{m\mod {p^{v_p(\ell)}}}h_{p}(m)\Big)=\prod_{p|\ell}\Big(1-\frac{2\kappa_p}{p^{v_p(\ell)}}\Big)=\frac {2^{v_2(\ell)-1}-\kappa_2}{2^{v_2(\ell)-1}}\prod_{2\neq p|\ell}\Big(\frac{p^{v_p(\ell)}-2\kappa_p}{p^{v_p(\ell)}}\Big)
}
where $\kappa_p:=|\{t\mod {p^{v_p(\ell)}}\mid h_p(t)=-1\}|$. Since the product over $2\neq p|\ell$ is a rational number with numerator and denominator which are both odd, it follows that if $\Av_\Z(\eps_{\GFF})=\frac hk$ with $(h,k)=1$, then $|\frac hk|\leq\frac {2^{v_2(k)}-1}{2^{v_2(k)}}$.
\end{proof}

\bibliographystyle{alpha}
\bibliography{biased_families}

\appendix

\section{Root number of ${\PFF}_\r$}\label{appen_1}
The proofs of the propositions in Appendix~\ref{appen_1} and~\ref{appen_2} can be obtained by a long case by case analysis in the same way as in the proof of Proposition~\ref{p5}.

\medskip

In this appendix we give the local root numbers for the family
$$
{\PFF_{\r}} \colon y^2=x^3 + 3tx^2+ 3\r x + \r t
$$
for which we have
\begin{eqnarray*}
c_4 & = & 2^4 \times 3^2 (t^2-\r),\\
c_6 & = &-2^6\times3^3\times t(t^2-\r), \\
\Delta & = & -2^6 3^3 \r(t^2-\r)^2, \\
j & = & \frac{-2^6 3^3}{\r} (t^2-\r).
\end{eqnarray*}

For a prime $p$, we denote by $w_p(t)$ the local root number of $\PFF_{\r}$ at $p$.

\begin{prop} If $p\geq 5$, we have:
\begin{itemize}
\item if $0\leq v_p(\r) <2v_p(t)$ then if $v_p(\r)$ is even $w_p(t)=\leg{-1}{p}^{v_p(\r)/2}$  and otherwise $w_p(t)=\leg{-2}{p}$ (this case also holds for $t=0$ for which $v_p(t)=+\infty$);
\item if $0\leq 2v_p(t) < v_p(\r)$ then if $v_p(t)$ is even $w_p(t)=-\leg{3t_p}{p}$  and otherwise $w_p(t)=\leg{-1}{p}$;
\item if $0\leq 2v_p(t)=v_p(\r)$ then
\begin{itemize}
\item if $v_p(t^2-\r)\equiv v_p(t) \pmod{2}$ then if $v_p(t^2-\r) + v_p(t) \equiv 0 \pmod{3}$ then $w_p(t)=1$ otherwise $w_p(t)=\leg{-3}{p}$;
\item if  $v_p(t^2-\r) \not \equiv v_p(t) \pmod{2}$ then $w_p(t)=\leg{-1}{p}$.
\end{itemize}
\end{itemize}
\end{prop}
\begin{prop} If $p=3$, we have:
\begin{itemize}
\item if $0\leq v_3(\r) <2v_3(t)$:
\begin{itemize}
\item if $v_3(\r) \equiv 0 \pmod{4}$ if $v_3(t)=1+v_3(\r)/2$ then $w_3(t)=1$ if and only if $t_3\equiv 1 \pmod{3}$ if $v_3(t)>1+v_3(\r)/2$ then $w_3(t)=1$;
\item if $v_3(\r) \equiv 1 \pmod{4}$ if $v_3(t)=1/2+v_3(\r)/2$ then $w_3(t)=1$ if and only if $\r_3 \equiv 1 \pmod{3}$ if $v_3(t)>1/2+v_3(\r)/2$ then $w_3(t)=-1$;
\item if $v_3(\r) \equiv 2 \pmod{4}$ if $v_3(t)=1+v_3(\r)/2$ then $w_3(t)=1$ if and only if $t_3 \not \equiv \r_3 \pmod{3}$ if $v_3(t)>1+v_3(\r)/2$ then $w_3(t)=1$;
\item if $v_3(\r) \equiv 3 \pmod{4}$ then $w_3(t)=1$.
\end{itemize}
\item If $0\leq 2v_3(t)< v_3(\r)$:
\begin{itemize}
\item if $v_3(t) \equiv 0 \pmod{2}$ if $v_3(\r)-2v_3(t)=1$ then $w_3(t)=1$, if $v_3(\r)-2v_3(t)=2$ then $w_3(t)=1$ if and only if $t_3 \equiv \r_3 \pmod{3}$ and if $v_3(\r)-2v_3(t)\geq 3$ then $w_3(t)=-1$;
\item if $v_3(t) \equiv 1 \pmod{2}$ if $v_3(\r)-2v_3(t)=1$ then $w_3(t)=1$ if and only if $\r_3 \equiv 1 \pmod{3}$, if $v_3(\r)-2v_3(t)=2$ then $w_3(t)=1$ if and only if $t_3\equiv 2 \pmod{3}$,
if $v_3(\r)-2v_3(t)=3$  then $w_3(t)=1$ and if $v_3(\r)-2v_3(t)\geq 4$ then $w_3(t)=1$  if and only if $t_3\equiv 2 \pmod{3}$.
\end{itemize}
\item If $0 \leq 2v_3(t) = v_3(\r)$ and $v_3(t)$ even:
\begin{itemize}
\item if $v_3(t^2-\r)=2v_3(t)$ then $w_3(t)=1$ if and only $\r_3 \equiv 2 \pmod{3}$ and $\r_3t_3 \not\equiv 2,4 \pmod{9}$;
\item if $v_3(t^2-\r)-2v_3(t) \equiv 0 \pmod{6}$ and $v_3(t^2-\r)-2v_3(t)>0$ then $w_3(t)=1$ if and only if $t_3(t^2-\r)_3 \not\equiv 7,8 \pmod{9}$;
\item if $v_3(t^2-\r)-2v_3(t) \equiv 1$ or $2 \pmod{6}$ then $w_3(t)=1$ if and only if $t_3(t^2-\r)_3 \equiv 1 \pmod{3}$;
\item if $v_3(t^2-\r)-2v_3(t) \equiv 3$ then $w_3(t)=1$ if and only if $t_3(t^2-\r)_3 \not\equiv 1,2 \pmod{9}$;
\item  if $v_3(t^2-\r)-2v_3(t) \equiv 4$ or $5 \pmod{6}$ then $w_3(t)=1$ if and only if $t_3(t^2-\r)_3 \equiv 2 \pmod{3}$.
\end{itemize}
\item If $0 \leq 2v_3(t) = v_3(\r)$ and $v_3(t)$ odd:
\begin{itemize}
\item if $v_3(t^2-\r)=2v_3(t)$ then $w_3(t)=1$ if and only $\r_3 \equiv 2 \pmod{3}$ and $\r_3t_3 \not\equiv 2,4 \pmod{9}$;
\item if $v_3(t^2-\r)-2v_3(t) \equiv 0 \pmod{6}$ and $v_3(t^2-\r)-2v_3(t)>0$ then $w_3(t)=1$ if and only if $t_3(t^2-\r)_3 \not\equiv 1,2 \pmod{9}$;
\item if $v_3(t^2-\r)-2v_3(t) \equiv 1$ or $2 \pmod{6}$ then $w_3(t)=1$ if and only if $t_3(t^2-\r)_3 \equiv 2 \pmod{3}$;
\item if $v_3(t^2-\r)-2v_3(t) \equiv 3 \pmod{6}$ then $w_3(t)=1$ if and only if $t_3(t^2-\r)_3 \not\equiv 7,8 \pmod{9}$;
\item  if $v_3(t^2-\r)-2v_3(t) \equiv 4$ or $5 \pmod{6}$ then $w_3(t)=1$ if and only if $t_3(t^2-\r)_3 \equiv 1 \pmod{3}$.
\end{itemize}
\end{itemize}
\end{prop}

\begin{prop} If $p=2$, we have:
\begin{itemize}
\item if $0\leq v_2(\r) <2v_2(t)$:
\begin{itemize}
\item if $v_2(\r) \equiv 0 \pmod{4}$ then
\begin{itemize}
\item if $v_2(t)-v_2(\r)/2=1$ then $w_2(t)=1$ if and only if
$$
\left\{ \begin{array}{l}
\r_2 \equiv 3 \pmod{4} \\
or \\
\r_2 \equiv 1 \mbox{ or } 13 \pmod{16} \mbox{ and } t_2 \equiv 3 \pmod{4} \\
or \\
\r_2 \equiv 5 \mbox{ or } 9 \pmod{16} \mbox{ and } t_2 \equiv 1 \pmod{4};
\end{array} \right.
$$
\item if $v_2(t)-v_2(\r)/2=2$ then $w_2(t)=1$ if and only if $\r_2 \equiv 5$ or $9 \pmod{16}$;
\item if $v_2(t)-v_2(\r)/2\geq2$ then $w_2(t)=1$ if and only if $\r_2 \equiv 1$ or $13 \pmod{16}$.
\end{itemize}
\item If $v_2(\r) \equiv 1 \pmod{4}$, if $v_2(t)-v_2(\r)/2=1/2$ then $w_2(t)=1$ if and only if
$$
\left\{
\begin{array}{l}
\r_2 \equiv 1 \mbox{ or } 3 \pmod{8} \mbox{ and } t_2 \equiv 3 \pmod{4} \\
or \\
\r_2 \equiv 5 \mbox{ or } 7 \pmod{8} \mbox{ and } t_2 \equiv 1 \pmod{4},
\end{array} \right.
$$
and if $v_2(t)-v_2(\r)/2\geq 1$ then $w_2(t)=1$ if and only if $\r_2 \equiv 5$ or $7 \pmod{8}$.
\item if $v_2(\r) \equiv 2 \pmod{4}$ then
\begin{itemize}
\item if $v_2(t)-v_2(\r)/2=1$ then $w_2(t)=1$ if and only if
$$
\left\{ \begin{array}{l}
\r_2 \equiv 1 \pmod{4} \\
or \\
\r_2 \equiv 3 \mbox{ or } 7 \pmod{16} \mbox{ and } t_2 \equiv 1 \pmod{4} \\
or \\
\r_2 \equiv 11 \mbox{ or } 15 \pmod{16} \mbox{ and } t_2 \equiv 3 \pmod{4} ;
\end{array} \right.
$$
\item if $v_2(t)-v_2(\r)/2=2$ then $w_2(t)=1$ if and only if $\r_2 \equiv 7$ or $11 \pmod{16}$;
\item if $v_2(t)-v_2(\r)/2\geq2$ then $w_2(t)=1$ if and only if $\r_2 \equiv 3$ or $15 \pmod{16}$.
\end{itemize}
\item if $v_2(\r) \equiv 3 \pmod{4}$, if $v_2(t)-v_2(\r)/2=1/2$ then $w_2(t)=1$ if and only if
$$
\left\{
\begin{array}{l}
\r_2 \equiv 1 \mbox{ or } 7 \pmod{8} \mbox{ and } t_2 \equiv 1 \pmod{4} \\
or \\
\r_2 \equiv 3 \mbox{ or } 5 \pmod{8} \mbox{ and } t_2 \equiv 3 \pmod{4},
\end{array} \right.
$$
and if $v_2(t)-v_2(\r)/2\geq 1$ then $w_2(t)=1$ if and only if $\r_2 \equiv 1$ or $3 \pmod{8}$.
\end{itemize}
\item If $0\leq 2v_2(t) < v_2(\r)$ and $v_2(t)$ even:
\begin{itemize}
\item if $v_2(\r)-2v_2(t)=1$ then $w_2(t)=1$ if and only if $r_2 \equiv 1 \pmod{4}$ and $t_2\equiv 1$ or $7 \pmod{8}$ or if $\r_2\equiv 3 \pmod{4}$ and $t_2\equiv 1$ or $3 \pmod{8}$;
\item if $v_2(\r)-2v_2(t)=2$ then $w_2(t)=1$ if and only if $\r_2 \equiv 1 \pmod{8}$ and $t_2\equiv 3, 5$ or $7 \pmod{8}$ or if $\r_2\equiv 5 \pmod{8}$ and $t_2\equiv 1, 3$ or $7 \pmod{8}$;
\item if $v_2(\r)-2v_2(t)=3$ then $w_2(t)=1$ if and only if $\r_2 \equiv 1 \pmod{4}$ and $t_2\equiv 3$ or $5 \pmod{8}$ or if $\r_2\equiv 3 \pmod{4}$ and $t_2\equiv 1$ or $3 \pmod{8}$;
\item if $v_2(\r)-2v_2(t)=4$ then $w_2(t)=1$ if and only if $t_2\equiv 1 \pmod{4}$ or if $\r_2\equiv 1 \pmod{4}$ and $t_2\equiv 3 \pmod{8}$ or if $\r_2 \equiv 3 \pmod{4}$ and $t_2 \equiv 7 \pmod{8}$;
\item if $v_2(\r)-2v_2(t)=5$ then $w_2(t)=1$ if and only if $t_2\equiv 1 \pmod{4}$ or if $t_2 \equiv 7 \pmod{8}$;
\item if $v_2(\r)-2v_2(t)=6$ then $w_2(t)=1$ if and only if $t_2\equiv 3 \pmod{4}$;
\item if $v_2(\r)-2v_2(t)\geq 7$ then $w_2(t)=1$ if and only if $t_2\equiv 7 \pmod{8}$.
\end{itemize}
\item If $0\leq 2v_2(t) < v_2(\r)$ and $v_2(t)$ odd:
\begin{itemize}
\item if $v_2(\r)-2v_2(t)=1$ then $w_2(t)=1$ if and only if $t_2 \equiv r_2$ or $\r_2+2 \pmod{8}$;
\item if $v_2(\r)-2v_2(t)=2$ then $w_2(t)=1$ if and only if $t_2 \equiv \r_2 \pmod{4}$;
\item if $v_2(\r)-2v_2(t)=3$ then $w_2(t)=1$ if and only if $\r_2 \equiv 3 \pmod{4}$;
\item if $v_2(\r)-2v_2(t)\geq 4$ then $w_2(t)=1$ if and only if $t_2 \equiv 3 \pmod{4}$.
\end{itemize}
\item If $0 \leq 2v_2(t) = v_2(\r)$ and $v_2(t)$ even:
\begin{itemize}
\item if $v_2(t^2-\r) - 2v_2(t) \equiv 0 \pmod{6}$ then $w_2(t)=1$ if and only if $t_2\equiv (t^2-\r)_2 \pmod{4}$;
\item if $v_2(t^2-\r) - 2v_2(t)=1$ then $w_2(t)=1$ if and only if $t_2 \equiv 1 \pmod{4}$ and $t_2(t^2-\r)_2 \equiv 1$ or $7 \pmod{8}$ or if $t_2 \equiv 3 \pmod{4}$ and $t_2(t^2-\r)_2 \equiv 5$ or $7 \pmod{8}$;
\item if $v_2(t^2-\r) - 2v_2(t) \equiv 1 \pmod{6}$ and $v_2(t^2-\r) - 2v_2(t)>1$  then $w_2(t)=-1$;
\item if $v_2(t^2-\r) - 2v_2(t)=2$ then $w_2(t)=1$ if and only if $t_2\equiv 3 \pmod{4}$;
\item if $v_2(t^2-\r) - 2v_2(t) \equiv 2 \pmod{6}$ and $v_2(t^2-\r) - 2v_2(t)>2$  then $w_2(t)=1$ if and only if $t_2\equiv (t^2-\r)_2 \pmod{4}$;
\item if $v_2(t^2-\r) - 2v_2(t)=3$ then $w_2(t)=1$ if and only if $t_2 \equiv 1 \pmod{4}$ and $t_2(t^2-\r)_2 \equiv 5$ or $7 \pmod{8}$ or if $t_2 \equiv 3 \pmod{4}$ and $t_2(t^2-\r)_2 \equiv 3$ or $5 \pmod{8}$;
\item if $v_2(t^2-\r) - 2v_2(t) \equiv 3 \pmod{6}$ and $v_2(t^2-\r) - 2v_2(t)>3$  then $w_2(t)=1$ if and only if $t_2\equiv (t^2-\r)_2 \pmod{4}$;
\item if $v_2(t^2-\r) - 2v_2(t) \equiv 4 \pmod{6}$ then $w_2(t)=1$ if and only if $t_2\equiv (t^2-\r)_2 \pmod{4}$;
\item if $v_2(t^2-\r) - 2v_2(t)=5$ then $w_2(t)=1$ if and only if $t_2(t^2-\r)_2 \equiv 1, 3$ or $7\pmod{8}$;
\item if $v_2(t^2-\r) - 2v_2(t) \equiv 5 \pmod{6}$ and $v_2(t^2-\r) - 2v_2(t)>5$  then $w_2(t)=-1$.
\end{itemize}
\item If $0 \leq 2v_2(t) = v_2(\r)$ and $v_2(t)$ odd:
\begin{itemize}
\item if $v_2(t^2-\r) - 2v_2(t) \equiv 0 \pmod{6}$ then $w_2(t)=1$ if and only if $t_2\equiv (t^2-\r)_2 \pmod{4}$;
\item if $v_2(t^2-\r) - 2v_2(t)=1$ then $w_2(t)=1$ if and only if $t_2\equiv 3 \pmod{8}$ or if $t_2\equiv 1 \pmod{8}$ and $(t^2-\r)_2 \equiv 1$ or $5 \pmod{8}$ or if
$t_2\equiv 5 \pmod{8}$ and $(t^2-\r)_2 \equiv 3$ or $7 \pmod{8}$;
\item if $v_2(t^2-\r) - 2v_2(t) \equiv 1 \pmod{6}$ and $v_2(t^2-\r) - 2v_2(t)>1$  then $w_2(t)=1$ if and only if $t_2 \equiv (t^2-\r)_2 \pmod{4}$;
\item if $v_2(t^2-\r) - 2v_2(t)=2$ then $w_2(t)=1$ if and only if $t_2 \equiv (t^2-\r)_2 \equiv 1 \pmod{4}$ or if $t_2 \equiv 7 \pmod{8}$ and $(t^2-\r)_2 \equiv 1 \pmod{4}$;
 \item if $v_2(t^2-\r) - 2v_2(t) \equiv 2 \pmod{6}$ and $v_2(t^2-\r) - 2v_2(t)>2$  then $w_2(t)=-1$;
 \item if $v_2(t^2-\r) - 2v_2(t)=3$ then $w_2(t)=1$ if and only if $(t^2-\r)_2 \equiv 3 \pmod{4}$;
 \item if $v_2(t^2-\r) - 2v_2(t) \equiv 3 \pmod{6}$ and $v_2(t^2-\r) - 2v_2(t)>3$  then $w_2(t)=1$ if and only if $t_2 \equiv (t^2-\r)_2 \pmod{4}$;
 \item if $v_2(t^2-\r) - 2v_2(t)=4$ then $w_2(t)=1$ if and only if $t_2 \equiv 1 \pmod{4}$ and $t_2(t^2-\r)_2 \equiv 3, 5$ or $7 \pmod{8}$ of if $t_2 \equiv 3 \pmod{4}$ and $t_2(t^2-\r)_2 \equiv 1, 3$ or $7 \pmod{8}$;
\item if $v_2(t^2-\r) - 2v_2(t) \equiv 4 \pmod{6}$ and $v_2(t^2-\r) - 2v_2(t)>4$  then $w_2(t)=-1$;
\item if $v_2(t^2-\r) - 2v_2(t) \equiv 5 \pmod{6}$ then $w_2(t)=1$ if and only if $t_2\equiv (t^2-\r)_2 \pmod{4}$.
\end{itemize}
\end{itemize}
\end{prop}

%%%%%%%%%%%%
%
%
%%%%%%%%%%%%

\section{Root number at $2$ and $3$ for $\Hold(t)$}\label{appen_2}
We give the local root number at $p=2$ and $p=3$ for the family
$$
\Hold \colon y^2=x^3 + 3tx^2+ 3\hr tx + \hr^2t.
$$
The local root number at $p\geq5$ is given in Lemma~\ref{H-p5}.
For the family $\Hold$ we have
\begin{eqnarray*}
c_4 & = & 2^4 3^2 t (t-\hr), \\
c_6 & = &-2^5 3^3 t (t-\hr)(2t-\hr), \\
\Delta & = & -2^4 3^3 \hr^2 t^2 (t-\hr)^2, \\
j & = & \frac{-2^8 3^3}{\hr^2} t(t-\hr).
\end{eqnarray*}

For a prime $p$, we denote by $w_p(t)$ the local root number of $\Hold(t)$ at $p$. We shall also assume $t\neq0,\hr$.

\begin{prop}  \label{app-RN3} We have
\begin{itemize}
\item For $0 \leq v_3(\hr) < v_3(t)$ then
\begin{itemize}
\item if $v_3(t)-v_3(\hr) \equiv 0 \pmod{3}$, then $w_3(t)=-1$ if and only if
$$
(-1)^{v_3(t)}\hr_3^2t_3 \equiv 5 \mbox{ or } 7 \pmod{9};
$$
\item if $v_3(t)-v_3(\hr) \equiv 1 \mbox{ or } 2 \pmod{6}$, then $w_3(t)=-1$ if and only if
$$
(-1)^{v_3(\hr)} t_3 \equiv 1 \pmod{3};
$$
\item if $v_3(t)-v_3(\hr) \equiv 4 \mbox{ or } 5 \pmod{6}$, then $w_3(t)=-1$ if and only if
$$
(-1)^{v_3(\hr)} t_3 \equiv 2 \pmod{3}.
$$
\end{itemize}
\item If $0 \leq v_3(t) = v_3(\hr)$ and $v_3(t-\hr) -v_3(t)>0$ then
\begin{itemize}
\item if $v_3(t-\hr)-v_3(\hr) \equiv 0 \pmod{3}$, then $w_3(t)=-1$ if and only if
$$
(-1)^{v_3(t)}\hr_3^2(t-\hr)_3 \equiv 5 \mbox{ or } 7 \pmod{9};
$$
\item if $v_3(t-\hr)-v_3(\hr) \equiv 1 \mbox{ or } 2 \pmod{6}$, then $w_3(t)=-1$ if and only if
$$
(-1)^{v_3(\hr)} (t-\hr)_3 \equiv 1 \pmod{3};
$$
\item if $v_3(t-\hr)-v_3(\hr) \equiv 4 \mbox{ or } 5 \pmod{6}$, then $w_3(t)=-1$ if and only if
$$
(-1)^{v_3(\hr)} (t-\hr)_3 \equiv 1 \pmod{3}.
$$
\end{itemize}
\item If $0 \leq v_3(t) = v_3(\hr)$ and $v_3(2t-\hr) -v_3(t)>0$ then
\begin{itemize}
\item if $v_3(2t-\hr)-v_3(\hr)=1$, $w_3=1$ if and only if $(2t_3-\hr_3) \equiv (-1)^{v_3(\hr)} 6 \pmod{9}$;
\item if $v_3(2t-\hr)-v_3(\hr)>1$, then $w_3=1$.
\end{itemize}
\item If $0 \leq v_3(t) <v_3(\hr)$ then
\begin{itemize}
\item if $v_3(t) \equiv 0 \pmod{2}$ and $v_3(\hr)-v_3(t)=1$, $w_3(t)=1$ if and only if $t_3 \equiv 1 \pmod{3}$;
\item if $v_3(t) \equiv 0 \pmod{2}$ and $v_3(\hr)-v_3(t)>1$, then $w_3(t)=-1$;
\item if $v_3(t) \equiv 1 \pmod{2}$ then $w_3(t)=1$ if and only if $t_3 \equiv 2 \pmod{3}$.
\end{itemize}
\end{itemize}
\end{prop}

\begin{prop} \label{app-RN2} We have
\begin{itemize}
\item For $0 \leq v_2(\hr) < v_2(t)$ and $v_2(\hr)$ even then
\begin{itemize}
\item if $v_2(t)-v_2(\hr) \equiv 0 \pmod{6}$, $w_2(t)=-1$;
\item if $v_2(t)-v_2(\hr) = 1$ then $w_2(t)=-1$ if and only if $t_2 \equiv \hr_2 \pmod{4}$;
\item if $v_2(t)-v_2(\hr) \equiv 1 \pmod{6}$ and $v_2(t)-v_2(\hr) > 1$ then $w_2(t) \equiv t_2 \pmod{4}$;
\item if $v_2(t)-v_2(\hr) = 2$ then $w_2(t)=1$ if and only one of the following conditions hold
$$
\left\{ \begin{array}{l}
t_2 \equiv 3 \pmod{4} \\
\mbox{ or } \\
t_2 \equiv 1 \pmod{8} \mbox{ and } \hr_2\equiv 3 \mbox{ or } 7 \pmod{8} \\
\mbox{ or } \\
t_2 \equiv 5 \pmod{8} \mbox{ and } \hr_2\equiv 1 \mbox{ or } 5 \pmod{8}
\end{array} \right. ;
$$
\item if $v_2(t)-v_2(\hr) \equiv 2 \pmod{6}$ and $v_2(t)-v_2(\hr) > 2$ then $w_2(t)=-1$;
\item if $v_2(t)-v_2(\hr) \equiv 3, 4$ or $5 \pmod{6}$ then $w_2(t) \equiv t_2 \pmod{4}$.
\end{itemize}
\item For $0 \leq v_2(\hr) < v_2(t)$ and $v_2(\hr)$ odd then
\begin{itemize}
\item if $v_2(t)-v_2(\hr) \equiv 0, 2$ or $4 \pmod{6}$ then $w_2(t) \equiv t_2 \pmod{4}$;
\item if $v_2(t)-v_2(\hr) =1$ then $w_2(t)=1$ if and only if one of the following conditions hold
$$
\left\{ \begin{array}{l}
t_2 \equiv 1 \pmod{8} \\
\mbox{ or } \\
(t_2,\hr_2) \equiv (3,1), (3,5), (7,3) \mbox{ or } (7,7) \pmod{8}  \end{array} \right. ;
$$
\item if $v_2(t)-v_2(\hr) \equiv 1 \pmod{6}$ and $v_2(t)-v_2(\hr)>1$ then $w_2(t) \equiv t_2 \pmod{4}$;
\item if $v_2(t)-v_2(\hr) =3$ then $w_2(t)=-1$ if and only if $t_2 \equiv 5 \pmod{8}$;
\item if $v_2(t)-v_2(\hr) \equiv 3 \pmod{6}$ and $v_2(t)-v_2(\hr)>3$ then $w_2(t)=-1$;
\item if $v_2(t)-v_2(\hr) \equiv 5 \pmod{6}$ then $w_2(t)=-1$.
\end{itemize}
\item For $0 \leq v_2(t)<v_2(\hr)-1$ and $v_2(t)$ even then
\begin{itemize}
\item if $v_2(\hr)-v_2(t)=2$ then $w_2(t)=1$ if and only if one of the following conditions hold
$$
\left\{ \begin{array}{l}
t_2 \equiv 1, 5, 7 \pmod{8}  \mbox{ and } \hr_2 \equiv 1 \pmod{4}  \\
\mbox{ or } \\
t_2\equiv 1,3, 5\pmod{8}  \mbox{ and } \hr_2 \equiv 3 \pmod{4} \end{array} \right. ;
$$
\item if $v_2(\hr)-v_2(t)=3$ then $w_2(t)=1$ if and only $t_2\equiv 1,5, 7 \pmod{8}$;
\item if $v_2(\hr)-v_2(t)=4$ then $w_2(t)=1$ if and only if $t_2\equiv 3 \pmod{4}$;
\item if $v_2(\hr)-v_2(t)\geq 5$ then $w_2(t)=1$ if and only if $t_2\equiv 7\pmod{8}$.
\end{itemize}
\item For $0 \leq v_2(t)<v_2(\hr)-1$ and $v_2(t)$ odd then $w_2(t)=1$ if and only if $t_2\equiv 3 \pmod{4}$.
\item For $0 \leq v_2(t)=v_2(\hr)-1$ and $v_2(t)$ even then $w_2(t)=1$ if and only if one of the following conditions hold
$$
\left\{ \begin{array}{l}
t_2 \equiv 7 \pmod{8} \\
\mbox{ or } \\
t_2 \equiv 1 \pmod{8} \mbox{ and } \hr_2 \equiv 1 \pmod{4} \\
\mbox{ or } \\
t_2 \equiv 5 \pmod{8} \mbox{ and } \hr_2 \equiv 3 \pmod{4}.
  \end{array} \right.
$$
\item For $0 \leq v_2(t)=v_2(\hr)-1$ and $v_2(t)$ odd then $w_2(t)=-1$ if and only if $t_2 \equiv \hr_2 \pmod{4}$.
\item For $0 \leq v_2(t)=v_2(\hr)$ and $v_2(t)$ even then
\begin{itemize}
\item if $v_2(t-\hr)-v_2(\hr)\equiv 0 \pmod{6}$ then $w_2(t)=-1$;
\item if $v_2(t-\hr)-v_2(\hr) =1$ then $w_2(t)=1$ if and only if $(t_2,\hr_2) \equiv (1,3), (3,1), (5,7)$ or $(7,5) \pmod{8}$;
\item if $v_2(t-\hr)-v_2(\hr)\equiv 1 \pmod{6}$ and $v_2(t-\hr)-v_2(\hr) >1$ then $w_2(t)=1$ if and only if
$(t-\hr)_2 \equiv 1 \pmod{4}$;
\item if $v_2(t-\hr)-v_2(\hr) =2$ then $w_2(t)=1$ if and only if one of the following conditions hold
$$
\left\{ \begin{array}{l}
t_2-\hr_2 \equiv 12  \pmod{16} \\
\mbox{ or } \\
\hr_2 \equiv 1 \pmod{4} \mbox{ and } t_2 \equiv \hr_2 + 4 \pmod{32} \\
\mbox{ or } \\
\hr_2 \equiv 3 \pmod{4} \mbox{ and } t_2 \equiv \hr_2 + 20 \pmod{32} ;
  \end{array} \right.
$$
\item if $v_2(t-\hr)-v_2(\hr)\equiv 2 \pmod{6}$ and $v_2(t-\hr)-v_2(\hr) >2$ then $w_2(t)=-1$;
\item $v_2(t-\hr)-v_2(t) \equiv 3, 4$ or $5 \pmod{6}$ then $w_2(t) \equiv (t-\hr)_2 \pmod{4}$.
\end{itemize}
\item For $0 \leq v_2(t)=v_2(\hr)$ and $v_2(t)$ odd then
\begin{itemize}
\item if $v_2(t-\hr)-v_2(\hr)\equiv 0 \pmod{6}$ then $w_2(t) \equiv (t-\hr)_2 \pmod{4}$;
\item if $v_2(t-\hr)-v_2(\hr)=1$ then $w_2(t)=1$ if and only if one of the following conditions hold
$$
\left\{ \begin{array}{l}
t_2-\hr_2 \equiv 2  \pmod{16} \\
\mbox{ or } \\
\hr_2 \equiv 1 \pmod{4} \mbox{ and } t_2 \equiv \hr_2 + 14 \pmod{16} \\
\mbox{ or } \\
\hr_2 \equiv 3 \pmod{4} \mbox{ and } t_2 \equiv \hr_2 + 6 \pmod{16};
  \end{array} \right.
$$
\item if $v_2(t-\hr)-v_2(\hr) \equiv 1 \pmod{6}$ and $v_2(t-\hr)-v_2(\hr) >1$ then $w_2(t) \equiv (t-\hr)_2 \pmod{4}$;
\item if $v_2(t-\hr)-v_2(\hr) \equiv 2 \pmod{6}$ then $w_2(t) \equiv (t-\hr)_2 \pmod{4}$;
\item if $v_2(t-\hr)-v_2(\hr)=3$ then $w_2(t)=-1$ if and only $(t-\hr)_2 \equiv 5 \pmod{8}$;
\item if $v_2(t-\hr)-v_2(\hr) \equiv 3 \pmod{6}$ and $v_2(t-\hr)-v_2(\hr) >3$ then $w_2(t)=-1$;
\item if $v_2(t-\hr)-v_2(\hr) \equiv 4\pmod{6}$ then $w_2(t) \equiv (t-\hr)_2 \pmod{4}$;
\item if $v_2(t-\hr)-v_2(\hr) \equiv 5\pmod{6}$ then $w_2(t) =-1$.
\end{itemize}
\end{itemize}
\end{prop}

\end{document}